\RequirePackage{silence}
\documentclass[a4paper, leqno, 11pt]{amsart}

\usepackage[T1]{fontenc}	
\usepackage{graphicx}
\usepackage{amssymb,centernot}
\usepackage{upgreek}
\usepackage{mathrsfs}
\usepackage[usenames,dvipsnames]{xcolor}
\usepackage[inline,shortlabels]{enumitem}
\usepackage[sort,numbers]{natbib}						
\usepackage{verbatim}								
\usepackage{dsfont}									

\usepackage{pdflscape}
\usepackage{afterpage}

\usepackage[
						colorlinks,				
						breaklinks,unicode,            
						hypertexnames=false,
						citecolor=OliveGreen,
						linkcolor=Maroon
						]{hyperref} 

\usepackage{multirow}

\usepackage{mathtools}
\usepackage{xspace}

\usepackage[a4paper,top=3.2cm,bottom=2.7cm,left=3cm,right=3cm, bindingoffset=0mm]{geometry}

\usepackage{booktabs}

\usepackage{xr}
\makeatletter
\newcommand*{\centerfloat}{%
  \parindent \z@
  \leftskip \z@ \@plus 1fil \@minus \textwidth
  \rightskip\leftskip
  \parfillskip \z@skip}
\makeatother

\newcommand{\SCF}{$(\mathsf{SCF})_{\ref{d:wFe}}$\xspace}
\newcommand{\qSCF}{$(\mathsf{qSCF})_{\ref{d:qSCF}}$\xspace}

\newcommand{\boldSigma}{{\boldsymbol\A}}

\newcommand{\boldkappa}{{\boldsymbol \kappa}}

\usepackage{xparse}

\ExplSyntaxOn
\NewDocumentCommand{\makeabbrev}{mmm}
 {
  \yoruk_makeabbrev:nnn { #1 } { #2 } { #3 }
 }

\cs_new_protected:Npn \yoruk_makeabbrev:nnn #1 #2 #3
 {
  \clist_map_inline:nn { #3 }
   {
    \cs_new_protected:cpn { #2 } { #1 { ##1 } }
   }
 }
 \ExplSyntaxOff

\makeabbrev{\textbf}{tbf#1}{a,b,c,d,e,f,g,h,i,j,k,l,m,n,o,p,q,r,s,t,u,v,w,x,y,z,A,B,C,D,E,F,G,H,I,J,K,L,M,N,O,P,Q,R,S,T,U,V,W,X,Y,Z}

\makeabbrev{\textbf}{bf#1}{a,b,c,d,e,f,g,h,i,j,k,l,m,n,o,p,q,r,s,t,u,v,w,x,y,z,A,B,C,D,E,F,G,H,I,J,K,L,M,N,O,P,Q,R,S,T,U,V,W,X,Y,Z}

\makeabbrev{\textsf}{tsf#1}{a,b,c,d,e,f,g,h,i,j,k,l,m,n,o,p,q,r,s,t,u,v,w,x,y,z,A,B,C,D,E,F,G,H,I,J,K,L,M,N,O,P,Q,R,S,T,U,V,W,X,Y,Z}

\makeabbrev{\mathsf}{mss#1}{a,b,c,d,e,f,g,h,i,j,k,l,m,n,o,p,q,r,s,t,u,v,w,x,y,z,A,B,C,D,E,F,G,H,I,J,K,L,M,N,O,P,Q,R,S,T,U,V,W,X,Y,Z}

\makeabbrev{\mathfrak}{mf#1}{a,b,c,d,e,f,g,h,i,j,k,l,m,n,o,p,q,r,s,t,u,v,w,x,y,z,A,B,C,D,E,F,G,H,I,J,K,L,M,N,O,P,Q,R,S,T,U,V,W,X,Y,Z}

\makeabbrev{\mathrm}{mrm#1}{a,b,c,d,e,f,g,h,i,j,k,l,m,n,o,p,q,r,s,t,u,v,w,x,y,z,A,B,C,D,E,F,G,H,I,J,K,L,M,N,O,P,Q,R,S,T,U,V,W,X,Y,Z}

\makeabbrev{\mathbf}{mbf#1}{a,b,c,d,e,f,g,h,i,j,k,l,m,n,o,p,q,r,s,t,u,v,w,x,y,z,A,B,C,D,E,F,G,H,I,J,K,L,M,N,O,P,Q,R,S,T,U,V,W,X,Y,Z}

\makeabbrev{\mathcal}{mc#1}{A,B,C,D,E,F,G,H,I,J,K,L,M,N,O,P,Q,R,S,T,U,V,W,X,Y,Z}

\makeabbrev{\mathbb}{mbb#1}{A,B,C,D,E,F,G,H,I,J,K,L,M,N,O,P,Q,R,S,T,U,V,W,X,Y,Z}

\makeabbrev{\mathscr}{ms#1}{A,B,C,D,E,F,G,H,I,J,K,L,M,N,O,P,Q,R,S,T,U,V,W,X,Y,Z}

\makeabbrev{\mathrm}{#1}{
Id,id,ran,rk,diag,stab,ann,conv,pr,ev,tr,End,Hom,sgn,im,op,can,fin,ext,red,tot,
rot,usc,lsc,Lip,LocLip,lip,bSymLip,osc,AC,loc,uloc,spec,coz,z,ul,
supp,Opt,Adm,Cpl,Geo,GeoSel,GeoOpt,GeoAdm,GeoCpl,reg,
bd,co,Ric,Exp,dExp,dist,seg,Seg,cut,fcut,Cut,SDiff,Iso,Isom,diam,cl,Homeo,Diff,Der,vol,dvol,inj,relint, Graph, sub,codim,Graff,
var,law,Var,Poi,Gam,pa,so,iso,fs,inv,pqi,mix,
TestF,
}

\makeabbrev{\mathsf}{#1}{DP,CD,BE,MCP,Ent,wMTW,MTW,RCD,RQCD,ncRCD,QCD,EVI,Irr,IH,SC,wFe,VA,UP,Curv,Alex,CAT}

\newcommand{\wBE}{\mathsf{BE}}
\newcommand{\bLip}{\mathrm{Lip}_b}

\newcommand{\T}{\tau} 
\newcommand{\A}{\Sigma} 
\newcommand{\Bo}[1]{\msB_{#1}} 

\newcommand{\Bdd}[1]{\msO_{#1}}
\newcommand{\Ed}{{\msE_\mssd}}  

\newcommand{\eps}{\varepsilon}

\newcommand{\defeq}{\eqqcolon}
\renewcommand{\complement}{\mathrm{c}}

\newcommand{\emparg}{{\,\cdot\,}}

\DeclareMathOperator{\Span}{span}

\newcommand{\slo}[2][]{\abs{\mathrm{D}#2}_{#1}}

\newcommand{\Sb}{\A_b}

\newcommand{\dint}[2][]{\sideset{^{#1}\!\!\!}{_{#2}^{\scriptstyle\oplus}}\int}

\newcommand{\Ch}[1][]{\mathsf{Ch}_{#1}}

\newcommand{\forallae}[1]{{\textrm{\,for ${#1}$-a.e.\,}}}
\newcommand{\as}[1]{\quad #1\text{-a.e.}}

\newcommand{\dom}[1]{\msD(#1)}

\newcommand{\domloc}[1]{\msD(#1)_{\loc}^\bullet}
\newcommand{\domext}[1]{\msD_{e}(#1)}
\newcommand{\DzLoc}[1]{\mbbL^{#1}_{\loc}}
\newcommand{\DzLocB}[1]{\mbbL^{#1}_{\loc,b}}

\newcommand{\DzB}[1]{\mbbL^{#1}_b}

\newcommand{\Lipu}{\mathrm{Lip}^1}
\newcommand{\bLipu}{\mathrm{Lip}^1_b}

\newcommand{\Rad}[2]{\mathsf{Rad}_{#1,#2}}
\newcommand{\kWC}[1]{\mathsf{WC}^{\mathrm{ker}}_{#1}}
\newcommand{\sgWC}[1]{\mathsf{WC}^{\mathrm{sg}}_{#1}}
\newcommand{\dRad}[2]{{#1}\textrm{-}\mathsf{Rad}_{#2}}

\newcommand{\cSL}[3]{\mathsf{cSL}_{#1,#2,#3}}
\newcommand{\SL}[2]{\mathsf{SL}_{#1,#2}}
\newcommand{\dcSL}[3]{{#1}\textrm{-}\mathsf{cSL}_{#2,#3}}
\newcommand{\dSL}[2]{{#1}\textrm{-}\mathsf{SL}_{#2}}

\DeclareMathOperator{\eqdef}{\coloneqq}

\let\epsilon\varepsilon

\let\temp\phi
\let\phi\varphi
\let\varphi\temp

\newcommand{\longrar}{\longrightarrow}
\newcommand{\rar}{\rightarrow}

\newcommand{\nlim}{\lim_{n}}								

\newcommand{\klim}{\lim_{k }}

\newcommand{\nliminf}{\liminf_{n }}

\newcommand{\diff}{\mathop{}\!\mathrm{d}}						

\newcommand{\ttabs}[1]{\lvert#1\rvert}	
\newcommand{\tabs}[1]{\big\lvert#1\big\rvert}	
\newcommand{\abs}[1]{\left\lvert#1\right\rvert}						
\newcommand{\tnorm}[1]{\big\lVert#1\big\rVert}						
\newcommand{\ttnorm}[1]{\lVert#1\rVert}
\newcommand{\norm}[1]{\left\lVert#1\right\rVert}					
\newcommand{\set}[1]{\left\{#1\right\}}							
\newcommand{\tset}[1]{\big\{#1\big\}}							
\newcommand{\ttset}[1]{\{#1\}}									
\newcommand{\paren}[1]{\left(#1\right)}							
\newcommand{\tparen}[1]{\big({#1}\big)}
\newcommand{\ttparen}[1]{({#1})}
\newcommand{\quadre}[1]{\left[#1\right]}							
\newcommand{\class}[2][]{\left[#2\right]_{#1}}						
\newcommand{\tclass}[2][]{\big [#2\big]_{#1}}						
\newcommand{\ttclass}[2][]{[#2]_{#1}}							
\newcommand{\spclass}[2][]{#2_{#1}}

\newcommand{\Li}[2][]{\mathrm{L}_{#1}(#2)}						
\newcommand{\bLi}[2][]{\mathrm{L}_{#1}\Big({#2}\Big)}						

\newcommand{\rep}[1]{\hat #1}									
\newcommand{\widerep}[1]{\widehat{#1}}									
\newcommand{\reptwo}[1]{\tilde{#1}}							
\newcommand{\tbraket}[1]{\big[#1\big]}							
\newcommand{\tscalar}[2]{\big\langle #1 \, \big |\, #2\big\rangle}			
\newcommand{\hotimes}{\widehat{\otimes}}

\newcommand{\asym}[1]{{\scriptscriptstyle{[#1]}}}
\newcommand{\sym}[1]{{\scriptscriptstyle{(#1)}}}
\newcommand{\tym}[1]{{\scriptscriptstyle{\times #1}}}
\newcommand{\otym}[1]{{\scriptscriptstyle{\otimes #1}}}

\newcommand{\hotym}[1]{{\scriptscriptstyle{\widehat{\otimes}{#1}}}}

\DeclareSymbolFont{symbolsC}{U}{pxsyc}{m}{n}
\SetSymbolFont{symbolsC}{bold}{U}{pxsyc}{bx}{n}
\DeclareFontSubstitution{U}{pxsyc}{m}{n}
\DeclareMathSymbol{\medcirc}{\mathbin}{symbolsC}{7}
\DeclareSymbolFont{symbolsZ}{OMS}{pxsy}{m}{n}
\SetSymbolFont{symbolsZ}{bold}{OMS}{pxsy}{bx}{n}
\DeclareFontSubstitution{OMS}{pxsy}{m}{n}

\newcommand{\seq}[1]{\paren{#1}}								
\newcommand{\tseq}[1]{{\big(#1\big)}}

\newcommand{\Cb}{\mcC_b}									
\newcommand{\Cz}{\mcC_0}									
\newcommand{\Ccompl}{\mcC_\infty}									
\newcommand{\Cbinfty}{{\mcC_{b}^{\infty}}}
\newcommand{\Czinfty}{{\mcC_0^{\infty}}}

\newcommand{\Meas}{\mathscr M}

\newcommand{\pfwd}{\sharp}
\DeclareMathOperator*{\esssup}{esssup}
\DeclareMathOperator*{\essinf}{essinf}

\DeclareMathOperator{\car}{\mathds 1}

\DeclareMathOperator{\emp}{\varnothing} 
\newcommand{\N}{{\mathbb N}}
\newcommand{\R}{{\mathbb R}}

\newcommand{\TLDS}{\textsc{tlds}\xspace}
\newcommand{\MLDS}{\textsc{mlds}\xspace}
\newcommand{\EMLDS}{\textsc{emlds}\xspace}
\newcommand{\parEMLDS}{\textsc{(e)mlds}\xspace}

\newcommand{\lb}{\mfl}
\newcommand{\llb}{\scriptstyle\lb}
\newcommand{\Lb}{\mathfrak{L}}

\newcommand{\restr}{\big\lvert}

\newcommand{\mrestr}[1]{\!\downharpoonright_{#1}}
\newcommand{\trid}{{\star}}

\allowdisplaybreaks

\newcommand{\comma}{\,\,\mathrm{,}\;\,}
\newcommand{\semicolon}{\,\,\mathrm{;}\;\,}
\newcommand{\fstop}{\,\,\mathrm{.}}

\newcommand{\cdc}{\Gamma}
\newcommand{\LL}[2]{\mcL^{#1, #2}}
\newcommand{\TT}[2]{\mcT^{#1, #2}}
\newcommand{\hh}[2]{\mssh^{#1, #2}}
\newcommand{\TTt}{\mcT}

\newcommand{\EE}[2]{\mcE^{#1, #2}}
\newcommand{\FF}[2]{\mcI^{#1, #2}}
\newcommand{\repSF}[2]{\rep\cdc^{#1, #2}}
\newcommand{\repLL}[2]{\rep\mcL^{#1, #2}}

\newcommand{\SF}[2]{\cdc^{#1, #2}}

\renewcommand{\iint}{\int\!\!\!\!\int}

\DeclareMathOperator{\zero}{{\mathbf 0}}

\usepackage{scrextend}						

\newcommand{\Cyl}[1]{\mcF^\dUpsilon\mcC^\infty_b(#1)}
\newcommand{\CylQP}[2]{\mcF^\dUpsilon\mcC^{\infty}_b(#2)_{#1}}

\newcommand{\CylSet}[1]{\mcE^{\!\times\!}\mcC(#1)}

\newcommand{\Dz}{\mcD} 
 
\newcommand{\De}{\mcD_e} 
\newcommand{\Daux}{\mcD_a} 
\newcommand{\preW}{\msW^{1,2}_{\mathrm{pre}}} 

\newcommand{\Ex}[1]{\mathrm{Exp}^{\!\dUpsilon}(#1)}

\newcommand{\locfin}{\mathrm{lcf}}

\newcommand{\PP}{{\pi}}

\newcommand{\cpl}{q}
\newcommand{\QP}{{\mu}}

\newcommand{\hr}[1]{\bar\mssd_{#1}} 									

\newcommand{\coK}[2]{\co_{#2}({#1})}

\newcommand{\dUpsilon}{{\boldsymbol\Upsilon}}

\numberwithin{equation}{section}
\theoremstyle{plain}
\newtheorem{thm}{Theorem}[section]
\newtheorem*{thm*}{Theorem}
\newtheorem{Thm}{Theorem}
\newtheorem{Cor}[Thm]{Corollary}
\newtheorem*{Cor*}{Corollary}
\newtheorem*{mthm*}{Main Theorem}
\newtheorem{prop}[thm]{Proposition}
\newtheorem*{prop*}{Proposition}
\newtheorem{lem}[thm]{Lemma}
\newtheorem{cor}[thm]{Corollary}

\theoremstyle{definition}
\newtheorem{defs}[thm]{Definition}
\newtheorem{notat}[thm]{Notation}
\newtheorem*{defs*}{Definition}

\theoremstyle{remark}
\newtheorem{rem}[thm]{Remark}
\newtheorem{ese}[thm]{Example}

\renewcommand{\paragraph}[1]{\medskip\emph{#1}.\quad}

\begin{document}

\title[Configuration spaces over singular spaces II]{Configuration Spaces Over Singular Spaces
\vspace{2mm}\\ 
-- {\tiny II. Curvature --}}

\thanks{The authors are very grateful to Prof.s Dr.s Matthias Erbar and Martin Huesmann for several conversations on the subject of this work and on~\cite{ErbHue15}.
This work was written while the authors were guests of the Mathematematisches Forschungsinstitut Oberwolfach (MFO) as a Research in Pairs (2147p).
They are very grateful to the MFO and its staff for their very kind hospitality.}

\author[L.~Dello Schiavo]{Lorenzo Dello Schiavo}
\address{Institute of Science and Technology Austria, Am Campus 1, 3400 Klosterneuburg, Austria}
\email{lorenzo.delloschiavo@ist.ac.at}
\thanks{
L.D.S.\ gratefully acknowledges funding of his current position by the Austrian Science Fund (FWF) grant F65, and by the European Research Council (ERC, grant No.~716117, awarded to Prof.\ Dr.~Jan Maas).
A part of this work was completed while he was a member of the Institut f\"ur Angewandte Mathematik of the University of Bonn. He acknowledges funding of his position at that time by the Collaborative Research Center 1060.
}

\author[K.~Suzuki]{Kohei Suzuki}
\address{Fakult\"at f\"ur Mathematik, Universit\"at Bielefeld, D-33501, Bielefeld, Germany}
\email{ksuzuki@math.uni-bielefeld.de}
\thanks{K.S.~gratefully acknowledges funding by: the JSPS Overseas Research Fellowships, Grant Nr. 290142; World Premier International Research Center Initiative (WPI), MEXT, Japan; JSPS Grant-in-Aid for Scientific Research on Innovative Areas ``Discrete Geometric Analysis for Materials Design'', Grant Number 17H06465; and the Alexander von Humboldt Stiftung, Humboldt-Forschungsstipendium}



\renewcommand{\abstractname}{\normalsize Abstract}
\begin{abstract}
\normalsize
This is the second paper of a series on configuration spaces $\dUpsilon$ over singular spaces $X$. Here, we focus on geometric aspects of the extended metric measure space~$(\dUpsilon, \mssd_{\dUpsilon}, \QP)$ equipped with the $L^2$-transportation distance~$\mssd_{\dUpsilon}$, and a mixed Poisson measure~$\QP$.
Firstly, we establish the essential self-adjointness and the $L^p$-uniqueness for the Laplacian on~$\dUpsilon$ lifted from~$X$.
Secondly, we prove the equivalence of Bakry--\'Emery curvature bounds on~$X$ and on~$\dUpsilon$, without any metric assumption on~$X$.
We further prove the Evolution Variation Inequality on~$\dUpsilon$, and introduce the notion of synthetic Ricci-curvature lower bounds for the extended metric measure space~$\dUpsilon$.
As an application, we prove the Sobolev-to-Lipschitz property on~$\dUpsilon$ over singular spaces~$X$, originally conjectured in the case when~$X$ is a manifold by M.~R\"ockner and A.~Schield.
As a further application, we prove the $L^\infty$-to-$\mssd_{\dUpsilon}$-Lipschitz regularization of the heat semigroup on~$\dUpsilon$ and gives a new characterization of the ergodicity of the corresponding particle systems in terms of optimal transport.
\end{abstract}

\maketitle

\setcounter{tocdepth}{3}
\makeatletter
\def\l@subsection{\@tocline{2}{0pt}{2.5pc}{5pc}{}}
\def\l@subsubsection{\@tocline{3}{0pt}{4.75pc}{5pc}{}}
\makeatother

\tableofcontents

\section{Introduction}\label{s:Introduction}

This paper is the second in a series on configuration spaces over non-smooth spaces.
In the first paper~\cite{LzDSSuz21}, we focused on the fundamental construction of a Dirichlet space and of an extended metric measure space on configuration spaces over non-smooth base spaces.
Based on the foundation developed there, in this  second paper we study \emph{curvature} on configuration spaces over non-smooth base spaces.  

\medskip

For the sake of simplicity, throughout this introduction let~$(X,\T)$ be a locally compact Polish space, $\mssm$~be a fully supported atomless Radon measure, and~$(\cdc, \Dz)$ be a closable local square-field operator defined on some sub-algebra~$\Dz$ of the space~$\Cz(\T)$ of $\T$-continuous compactly supported functions on $X$.
We call~$\mcX\eqdef (X, \T, \mssm, \cdc)$ a \emph{topological local diffusion space} (in short: \TLDS, see Dfn.~\ref{d:TLSConfig2}). 
We further let~$\tparen{\EE{X}{\mssm},\dom{\EE{X}{\mssm}}}$ be the Dirichlet form obtained by integration of the closure~$\SF{X}{\mssm}$ of~$(\cdc,\Dz)$ with respect to~$\mssm$, and~$\TT{X}{\mssm}_\bullet\eqdef\tseq{\TT{X}{\mssm}_t}_{t\geq 0}$ the associated strongly continuous contraction semigroup.

\paragraph{Configuration spaces}
The configuration space $\dUpsilon$ over~$X$ is the set
\begin{equation*} 
\dUpsilon\eqdef \set{\gamma=\sum_{i=1}^N \delta_{x_i}: x_i\in X\comma N \in \N\cup \set{+\infty}\comma \gamma K<\infty \quad K \Subset X }\fstop
\end{equation*}
 of all locally finite point measures on $X$.
It is endowed with the \emph{vague topology}~$\T_\mrmv$, induced by duality with functions in~$\Cz(\T)$, and with a reference Borel probability measure~$\QP$, understood as the law of a proper point process on~$X$.
For any~$f \in \Cz(\T)$ and any configuration $\gamma \in \dUpsilon$, we denote by~$f^\trid \gamma=\gamma f$ the integral of~$f$ with respect to~$\gamma$.

\paragraph{Mixed Poisson measures}
Recall that a probability~$\lambda$ on~$\R^+$ is a \emph{L\'evy measure} if~$\lambda (1\wedge t)<\infty$.
In this paper, we always take $\mu$ to be either: 
\begin{enumerate*}[$(a)$]
\item the \emph{Poisson measure~$\pi_\mssm$ with intensity $\mssm$}, i.e.\ the law of a completely independent point process~$\gamma$ on~$X$ satisfying~$\mbfP(\gamma A=0)=e^{-\mssm A}$ for every Borel~$A\subset X$ (see Dfn.~\ref{d:PoissonConfig2});
or
\item the \emph{mixed Poisson measure~$\QP_{\lambda,\mssm}$ with intensity~$\mssm$ and L\'evy measure~$\lambda$}, i.e.\ the probability measure~$\QP_{\lambda,\mssm}=\int \PP_{s\cdot\mssm}\diff\lambda(s)$ on~$\dUpsilon$.
\end{enumerate*}

\paragraph{Dirichlet form on $\dUpsilon$} Define a space of \emph{cylinder functions} 
\begin{align*}
\Cyl{\Dz}\eqdef \set{\begin{matrix} u\colon \dUpsilon\rar \R : u=F\circ\mbff^\trid \comma  F\in \mcC^\infty_b(\R^k)\comma \\  f_1,\dotsc,  f_k\in  \Dz\comma\quad k\in \N_0 \end{matrix}}\comma
\end{align*}
where $\mbff\eqdef\seq{f_1, \ldots, f_k}$ and $\mbff^\trid\eqdef \seq{f_1^\trid, \ldots, f_k^\trid}$. 
We lift~$\cdc$ to a square field 
\begin{equation*}
\begin{gathered}
\SF{\dUpsilon}{\mu}: \Cyl{\Dz}^{\otimes 2} \longrightarrow \R
\\
\SF{\dUpsilon}{\mu}(u, v)(\gamma)\eqdef \sum_{i,j=1}^{k,m} (\partial_i F)(\mbff^\trid\gamma) \cdot (\partial_j G)(\mbfg^\trid\gamma) \cdot \cdc(f_i, g_j)^\trid \gamma \comma
\\
u=F\circ \mbff^\trid\in \Cyl{\Dz} \comma\qquad  v=G\circ \mbfg^\trid\in \Cyl{\Dz}\fstop
\end{gathered}
\end{equation*}
As shown in~\cite{LzDSSuz21}, the pre-Dirichlet form~$\EE{\dUpsilon}{\QP}$ defined by integration of~$\SF{\dUpsilon}{\QP}$ with respect to~$\mu$ is well-defined, densely defined, and closable.
We respectively denote by
\begin{align}\label{eq:IntroDirGenSemi}
\tparen{\EE{\dUpsilon}{\QP},\dom{\EE{\dUpsilon}{\QP}}}\comma \qquad \tparen{\LL{\dUpsilon}{\QP},\dom{\LL{\dUpsilon}{\QP}}}\comma \qquad \TT{\dUpsilon}{\QP}_\bullet\eqdef\tseq{\TT{\dUpsilon}{\QP}_t}_{t\geq 0}\comma
\end{align}
its closure, and the corresponding generator and semigroup on~$L^2(\mu)$.

\subsection{Bakry--\'Emery curvature bounds}
Let~$c\geq 1$ and~$K\in\R$ be fixed. A \TLDS $\mathcal X$ satisfies the \emph{weak Bakry--\'Emery gradient estimate} $\wBE_c(K,\infty)$ (Dfn.~\ref{d:BE}) if the following gradient estimate holds:
\begin{equation*}
\SF{X}{\mssm}(\TT{X}{\mssm}_t f) \leq c\, e^{-2Kt}\, \TT{X}{\mssm}_t\,\SF{X}{\mssm}(f)\comma \qquad f \in \dom{\EE{X}{\mssm}} \comma \quad t>0\fstop
\end{equation*}
For~$c=1$, this notion was originally introduced by D.~Bakry and M.~\'Emery, see, e.g., the monograph~\cite{BakGenLed14} and references therein.
It extends to a large variety of settings the definition of pointwise lower bound~$\Ric_g\geq K g$ for the Ricci curvature of a smooth Riemannian manifold with metric~$g$.
For~$c> 1$, the weak Bakry--\'Emery gradient estimates~$\wBE_c(K,\infty)$ have been so far investigated on some classes of sub-Riemannian manifolds (e.g., \emph{$H$-type Heisenberg groups}, see~\cite{Li06,BakBauBonCha08,Eld10}), and on metric networks, see~\cite{BauKel18}.
Importantly, the aforementioned structures do not satisfy the standard ($c=1$) $\BE(K,\infty)$ condition, but (a large part of) the theory of Bakry--\'Emery curvature bounds is readily adapted to the case~$c>1$ and thus applies as well to these structures.

\medskip

We further say that a \TLDS~$\mcX$ satisfies~$(\mathsf{SCF})$ (see Dfn.~\ref{d:wFe}) if
\begin{enumerate*}[$(a)$]
\item $\TT{X}{\mssm}_\bullet$ is \emph{stochastically complete}, i.e., $\TT{X}{\mssm}_t \car=\car$ for some (hence any) $t>0$, and
\item $\TT{X}{\mssm}_\bullet$ is \emph{$L^\infty(\mssm)$-to-$\Cb(\T)$ Feller}, i.e.\ such that~$\TT{X}{\mssm}_t u \in \Cb(\T)$ for every~$u \in L^\infty(\mssm)$ and every~$t>0$.
\end{enumerate*}
Our first main result is the following.

\begin{Thm}[Thm.\ \ref{t:BE}] \label{thm:BE}
Let~$(\mcX,\cdc)$ be a \TLDS satisfying~$(\mathsf{SCF})$, and fix~$c\geq 1$ and~$K\in\R$. Then,
\begin{align*}
\text{$(\mcX,\cdc)$\quad satisfies\quad $\wBE_c(K,\infty)$} \qquad \iff \qquad \text{$\tparen{\dUpsilon,\SF{\dUpsilon}{\QP}}$\quad satisfies\quad $\wBE_c(K,\infty)$} \fstop
\end{align*}
\end{Thm}

\paragraph{Overview}
When the base space $\mcX$ is a complete Riemannian manifold with Ricci curvature bounded below by some real constant~$K$, the bound~$\BE(K,\infty)$ for~$\tparen{\dUpsilon,\SF{\dUpsilon}{\PP}}$ was proved M.~Erbar and M.~Huesmann in~\cite{ErbHue15}.
The same result for weighted manifolds was conjectured in~\cite[Rmk.~1.5]{ErbHue15}.
In Theorem~\ref{thm:BE}, we do not assume any manifold nor even metric structure, thus confirming the conjecture of~\cite{ErbHue15} in far greater generality. 

The main technical novelties are as follows:
\begin{enumerate*}[$(a)$]
\item\label{rem:BE:1} We obtain an explicit \emph{semigroup representation} for~$\TT{\dUpsilon}{\QP}_\bullet$ (Thm.~\ref{t:AKR4.1}); 
\item As a consequence, we prove the \emph{essential self-adjointness} of~$\LL{\dUpsilon}{\PP}$ on a suitable core (Cor.~\ref{c:ESA}), which is an extension of~\cite[Thm.~4.2]{AlbKonRoe98}, proved in the case of manifold~$\mcX$.
\item We establish the \emph{heat-kernel identification} (Prop.~\ref{p:KonLytRoe}) in the absence of any metric structure.
Again in the case when~$\mcX$ is a Riemannian manifold, the corresponding statement was originally shown by Yu.~G.~Kondratiev, E.~W.~Lytvynov, and M.~R\"ockner in~\cite[Thm.~5.1]{KonLytRoe02}, and later in \cite[Thm.~2.4]{ErbHue15} under slightly weaker assumptions.
Both proofs, however, rely on volume growth estimates and Gaussian-type heat kernel estimates.
The same line of reasoning would therefore apply ---~if at all~--- only to the case when~$\mcX$ is a metric space; see Remark \ref{r:ComparisonKonLytRoe} for further details.

\item We make full use of the results developed in~\cite{LzDSSuz21} regarding the construction of the Dirichlet form in~\eqref{eq:IntroDirGenSemi}, assuming that~$\mcX$ is \emph{merely} a \TLDS.

\item In our setting, the Bakry--\'Emery gradient estimate~$\wBE_c(K,\infty)$ for~$\dUpsilon$ does not follow from the formulation of the Bakry--\'Emery curvature bounds in terms of the iterated square field operator~$\cdc_2$ and the general arguments for Markov Diffusion Triples in~\cite{BakGenLed14}.
This is due to the fact that \emph{none} of the several cores we consider is an algebra preserved by the action of both the generator~$\LL{\dUpsilon}{\QP}$ and the semigroup~$\TT{\dUpsilon}{\QP}_\bullet$, which is necessary to the discussion of Markov Diffusion Triples in~\cite[\S\S3.3.3, 3.3.4]{BakGenLed14}, see~\cite[Dfn.~3.3.1(iv) and (viii)]{BakGenLed14}.
\end{enumerate*}

\subsection{Riemannian curvature-dimension condition}
In the rest of this Introduction, let us further assume that~$(X,\T)$ is equipped with a complete and separable distance~$\mssd$ generating the topology~$\T$, and that~$\mssm$ is finite on $\mssd$-bounded sets.
We denote by~$\Ch=\Ch[\mssd,\mssm]\colon L^2(\mssm) \to \R \cup\{+\infty\}$ the \emph{Cheeger energy}~\cite[Thm.~4.5]{AmbGigSav14} of~$\mcX\eqdef (X,\mssd,\mssm)$.
If~$\Ch$ is a quadratic functional (i.e.\ if it satisfies the parallelogram identity), we say that~$(X,\mssd,\mssm)$ is \emph{infinitesimally Hilbertian}~\cite{Gig13}.
In this case, the quadratic form obtained from~$\Ch$ by polarization (also denoted by~$\Ch$) is a Dirichlet form with square field operator~$\slo[*]{\emparg}^2$, where~$\slo[*]{\emparg}$ denotes the \emph{minimal $2$-relaxed slope}~\cite[Dfn.~4.2]{AmbGigSav14}.
Provided that~$(\mcX,\cdc)$ be a \TLDS for the choice~$\cdc\eqdef \slo[*]{\emparg}^2$, we call~$(\mcX, \cdc, \mssd)$ a \emph{metric local diffusion space} (in short: \MLDS).

Fix~$K\in\R$ and~$N\in (1,\infty)$. We say that an \MLDS~$(\mcX,\cdc,\mssd)$ satisfies the \emph{Riemannian Curvature-Dimension condition} $\RCD(K, N)$ if $\mcX$ is infinitesimally Hilbertian and the following conditions hold: 
\begin{enumerate}[$(i)$]
\item if $\slo[*]{f} \in L^\infty(\mssm)$, then $f$ has a $\mssd$-continuous $\mssm$-modification;  
\item the following \emph{Bakry--Ledoux} gradient estimate holds: 
\begin{align}\label{eq:BakryLedoux}
\slo[*]{\tparen{\TT{X}{\mssm}_t f}}^2+\frac{4Kt^2}{N(e^{2Kt}-1)}\, \abs{\LL{X}{\mssm} \TT{X}{\mssm}_t f}^2\leq e^{-2Kt}\TT{X}{\mssm}_t\tparen{\slo[*]{f}^2}\quad \as{\mssm}\comma\quad t>0 \fstop
\end{align}
\end{enumerate}
This definition ---~rather: an equivalent formulation~--- was originally introduced by N.~Gigli in~\cite{Gig13}.
It makes sense also for $N=\infty$, in which case the second term in the left-hand side of~\eqref{eq:BakryLedoux} is set to be~$0$, and the condition is denoted by~$\RCD(K, \infty)$~\cite{AmbGigSav14b}.
A complete understanding of several equivalent characterizations was subsequently achieved in~\cite{ErbKuwStu15,AmbGigSav15}. 
For~$N=\infty$, one of these characterizations, given in terms of optimal transport theory, is the \emph{Evolution Variational Inequality}~$\EVI(K,\infty)$~\cite{AmbGigSav14b}  
\begin{align*}
\diff_t^+ \tfrac{1}{2}W_{2, \mssd}(\TT{X}{\mssm}_t\mu,\nu)^2+\tfrac{K}{2} W_{2, \mssd}(\TT{X}{\mssm}_t\mu_t,\nu)^2\leq \Ent_\mssm(\nu)-\Ent_\mssm(\TT{X}{\mssm}_t\mu) \comma \qquad t>0\comma
\end{align*}
where $W_{2, \mssd}$ is the \emph{$L^2$-Kantorovich--Rubinstein} (also: \emph{Wasserstein}) distance on the space $\msP_2(X)$ of Borel probabilities with finite second $\mssd$-moment, $\Ent_\mssm$ is the \emph{entropy} with respect to~$\mssm$, and $\mu, \nu \in \msP_2(X)$.
By a careful adaptation of these definitions to the case of \emph{extended} metric measure spaces, we can formulate $\EVI(K,\infty)$ also for the configuration space $\tparen{\dUpsilon, \SF{\dUpsilon}{\PP},\mssd_\dUpsilon}$.

Precise definitions and adaptations to extended metric measure spaces will be addressed in~\S\ref{s:RicciBounds}.

\smallskip

Our second main result is the following:
\begin{Thm}[Thm.~\ref{t:EVI}]\label{thm:EVI}
Let~$K \in \R$ and $2\leq N<\infty$. Then,
\begin{equation*}
\text{$\mcX$ satisfies~$\RCD^*(K,N)$} \qquad \implies \qquad \text{$\tparen{\dUpsilon, \T_\mrmv, \mssd_{\dUpsilon}, \QP}$ satisfies~$\EVI(K,\infty)$} \fstop
\end{equation*}
\end{Thm}

As a standard consequence, we further obtain:

\begin{Cor*}[Cor.~\ref{c:GradFlowEnt}]
Let~$\mcX$ be an \MLDS satisfying the $\RCD^*(K,N)$ condition.
Then, the heat semigroup~$\TT{\dUpsilon}{\PP}_\bullet$ defines the gradient flow of~$\Ent_\PP$ on measures absolutely continuous w.r.t.~$\PP$.
\end{Cor*}

\paragraph{Overview}
\begin{enumerate*}[$(a)$]
\item Theorem~\ref{thm:EVI} extends the analogous result~\cite[Thm.~5.10]{ErbHue15} for the case when the base space~$\mcX$ is a manifold, see Remark~\ref{r:ComparisonErbHue}. 
\item In the setting of extended metric measure spaces, L.~Ambrosio, M.~Erbar, and G.~Savar\'e proved in \cite[Cor.\ 11.3]{AmbErbSav16} the implication~$\BE(K,\infty)$ to~$\EVI(K,\infty)$.
However, their definition of~$\EVI(K,\infty)$ is ---~at least \emph{a priori}~--- different from ours, since it replaces the $L^2$-Kantorovich--Rubinstein distance~$W_{2, \mssd_{\dUpsilon}}$ with either of  the extended distances~$W_{\mcE}$ or~$W_{\mcE, *}$ (see \cite[Def.~10.4]{AmbErbSav16}), respectively modelled after the \emph{Benamou--Brenier formulation} and the \emph{Kantorovich duality formulation} of~$W_2$.
Since our space~$\dUpsilon$ is an extended metric measure space, understanding the relations between of~$W_{\mcE}$, $W_{\mcE, *}$, and~$W_{2, \mssd_{\dUpsilon}}$ is highly non-trivial and will be addressed in future work.
\end{enumerate*}

\smallskip

\subsection{Identification of analytic and metric structure}
Two natural structures coexist on the configuration space over an \MLDS: the analytic structure induced by the Dirichlet form $(\EE{\dUpsilon}{\QP}, \dom{\EE{\dUpsilon}{\QP}})$, and the geometric structure induced by the $L^2$-transportation extended distance~$\mssd_{\dUpsilon}$.
We investigate the relations between the two. 

Let $\mcX$ be an \MLDS.
Write~$\Li[\mssd]{f}$ for the Lipschitz constant of a Lipschitz function~$f\colon X\to\R$.
We say that~$\mcX$ possesses
\begin{enumerate*}[$(a)$]
\item the \emph{Rademacher-type property} $(\Rad{\mssd}{\mssm})$ if any bounded Lipschitz~$f$ belongs to~$\dom{\EE{X}{\mssm}}$ and~$\SF{X}{\mssm}(f) \le \Li[\mssd]{f}^2$.
\item the \emph{Sobolev-to-Lipschitz property} $(\SL{\mssm}{\mssd})$ if any $f \in \dom{\EE{X}{\mssm}}\cap L^\infty(\mssm)$ with $\SF{X}{\mssm}(f) \in L^\infty(\mssm)$ has a Lipschitz $\mssm$-representative $\rep{f}$ with $\Li[\mssd]{\rep f}\leq \sqrt{\norm{\SF{X}{\mssm}(f)}_{L^\infty}}$.
\end{enumerate*}
The properties~$(\Rad{\mssd}{\mssm})$ and~$(\SL{\mssm}{\mssd})$ hold for wide classes of base \MLDS's, such as: complete Riemannian manifolds; (ideal) sub-Riemannian manifolds; $\RCD(K,\infty)$-spaces (hence in particular~$\RCD^*(K,N)$-spaces), see~\cite{AmbGigSav14b, AmbGigMonRaj12}; and spaces satisfying the \emph{regular Riemannian quasi-Curvature-Dimension condition}~$\RQCD_\reg$, a further extension of the $\RCD^*(K,N)$ class recently introduced by E.~Milman in~\cite{Mil21}, see~\cite{LzDSSuz21a}.

Finally, we recall from Dirichlet-form theory that the \emph{intrinsic distance}~$\mssd_\QP$ of the form $\EE{\dUpsilon}{\QP}$ is
\begin{align*}
\mssd_{\QP}(\gamma, \eta):=\sup\set{u(\gamma)- u(\eta): \SF{\dUpsilon}{\QP}(u) \le 1, \quad u \in \mathcal C_b(\T_\mrmv) \cap \dom{\EE{\dUpsilon}{\QP}} } \fstop
\end{align*}
Combining Theorem~\ref{t:dSLConfig2} ($\mssd_\dUpsilon \ge \mssd_{\QP}$) with \cite[Thm.~5.2]{LzDSSuz21} ($\mssd_\dUpsilon \le \mssd_{\QP}$), we obtain
\begin{Thm}
Let~$(\mcX,\cdc,\mssd)$ be an \MLDS satisfying $(\mathsf{SCF})$,  $(\Rad{\mssd}{\mssm})$  and $(\SL{\mssm}{\mssd})$.
Then,
\begin{align}\label{eq:dSLIntro}
\mssd_\dUpsilon = \mssd_{\QP} \fstop
\end{align}
\end{Thm}

In particular, when~$\mcX$ is an~$\RCD(K,N)$ space ---~or, more generally, an~$\RQCD_\reg$ space~\mbox{---,} then the analytic and metric structures on~$\dUpsilon$ coincide:
 
\begin{Cor}[{\cite[Thm.s~5.2,~5.8]{LzDSSuz21} and Thm.~\ref{t:dSLConfig2}}]
Let $\mcX$ be an $\RCD(K,N)$ space. Then, 
\begin{align*}
\tparen{\Ch[\mssd_\dUpsilon,\QP],\dom{\Ch[\mssd_\dUpsilon,\QP]}}=\tparen{\EE{\dUpsilon}{\QP},\dom{\EE{\dUpsilon}{\QP}}} \qquad \text{and} \qquad  \mssd_\dUpsilon = \mssd_{\QP} \fstop
\end{align*}
\end{Cor}

\paragraph{Overview} 
The coincidence of the Cheeger energy of~$(\dUpsilon,\mssd_\dUpsilon,\QP)$ with the canonical form~$\EE{\dUpsilon}{\QP}$ was shown in great generality for both~$X$ and~$\QP$ in the first work in this series,~\cite{LzDSSuz21}.
For Poisson measures over manifolds, it is claimed in~\cite[Prop.~2.3]{ErbHue15}; see however Remark~\ref{r:ComparisonErbHue}.

For configurations spaces over manifolds, the identification of the intrinsic distance~$\mssd_\QP$ with the $L^2$-transportation distance~$\mssd_\dUpsilon$ was shown in~\cite[Thm.~1.5(ii)]{RoeSch99}.
Both the assumptions and the proof strategy of the results in~\cite{RoeSch99} heavily rely on the smooth structure of the base space, and in particular on the notion of quasi-invariance of measures on~$\dUpsilon$ with respect to a natural action on~$\dUpsilon$ of the group of diffeomorphisms of the base space.
Indeed, as noted in~\cite[after Lem.~5.2]{RoeSch99}, their strategy does not apply even to Lipschitz manifolds.

For configuration spaces over \MLDS's, the same identification of distances was also shown in~\cite[Thm.~5.25]{LzDSSuz21}.
Here, we prove the same assertion under a different (skew) set of assumptions to those in~\cite{LzDSSuz21}.
Our proof strategy, both here and in~\cite{LzDSSuz21}, applies to non-smooth \MLDS's and thus it does not rely on any smooth structure, nor on quasi-invariance of measures.
The proof strategy in~\cite{LzDSSuz21} relies on a localization argument via configuration spaces over balls. Whereas this argument applies to the far more general class of measures considered there, it poses some restrictions on the choice of the base spaces.
Here, we rather rely on a global property, the identification of the heat kernel.
Whereas this argument only applies to (mixed) Poisson measures, it poses non restriction on the choice of the base space, allowing us to consider, for example, the whole class of~$\RCD^*(K,N)$ spaces.

\subsection{On the \texorpdfstring{$\RCD$}{RCD} condition for extended metric spaces}
The Bakry--\'Emery gradient estimates, the \emph{logarithmic Harnack inequality} (Dfn.~\ref{d:LogH}), the \emph{Wasserstein contractivity estimate} (Dfn.~\ref{d:WC}), and the $\EVI(K,\infty)$ estimate all concur to a possible synthetic definition of Ricci-curvature lower bounds in the non-smooth setting.
For metric spaces (as opposed to: \emph{extended} metric spaces), a complete characterization was given in~\cite{AmbGigSav15} (see Thm.~\ref{t:RCD}).
In particular, the $\RCD(K,\infty)$ condition is equivalent to the validity of both the $\BE(K,\infty)$ gradient estimate and the Sobolev-to-Lipschitz property; also cf.~\cite{Hon18}.

For \emph{extended} metric spaces, no such characterization is available.
Whereas we refrain from proposing here a definition for the $\RCD(K,\infty)$ condition for extended metric spaces, our results show that configuration spaces over manifolds with Ricci curvature bounded below or, more generally, over $\RCD^*(K,N)$ spaces, would satisfy any such definition.

In~\cite{FanShaStu09} S.~Fang, J.~Shao, and K.-T.~Sturm proved synthetic Ricci-curvature lower bounds for the \emph{Wiener space}.
To date, configuration spaces constitute the only other example of infinite-dimensional extended metric measure spaces satisfying Ricci-curvature lower bounds in a very strong and purely metric measure sense.
In comparison with the Wiener space, analysis on configuration spaces is considerably more challenging, in that configuration spaces are not (embedded in) any linear space.

\paragraph{Overview}
Fix~$K\in\R$ and~$N\in [1,\infty)$, and let~$\mcX$ be an $\RCD^*(K,N)$ space.
For integer~$n$ denote by~$\mcX^\tym{n}$ the $n$-fold Cartesian product of~$\mcX$, and let~$\mfS_n$ be symmetric group of order~$n$, acting on~$\mcX^\tym{n}$ by permutation of coordinates.
Note that this is an action by measure-preserving isometries.

F.~Galaz-Garc\'ia, M.~Kell, A.~Mondino, and G.~Sosa proved in~\cite{GalKelMonSos18} that the quotient of any $\RCD^*(K,N)$ space by a compact isomorphic group action is again an~$\RCD^*(K,N)$-space.
Now, in light of the tensorization of the~$\RCD^*(K,N)$ condition~\cite{BacStu10, AmbGigSav14b}, their result implies in particular that the quotient~$\mcX^\tym{n}/\mfS_n$ is an~$\RCD^*(K, nN)$ space.

At least heuristically, our results on the Poisson configuration space~$\dUpsilon$ over~$\mcX$ may be regarded as the analogue of this fact in the case~$n=\infty$.
For this analogy to be rigorously justified, however, a thorough understanding of the \emph{extended}-metric measure space structure of~$\dUpsilon$ is required, which is outside the scope of the standard theory of~$\RCD(K,\infty)$ spaces.

\subsection{Applications} As a first application, we confirm the \emph{Sobolev-to-Lipschitz property} conjectured by M.~R\"ockner and A.~Schied in~\cite{RoeSch99}.
Let~$(M,g)$ be a Riemannian manifold.
We always assume Riemannian manifolds to be smooth and connected.
A \emph{weighted} Riemannian manifold~$(M,e^{-\psi}g)$ is always assumed to have smooth weight~$\psi$.
We say that any such manifold has Ricci curvature bounded below by~$K$ if~$\Ric_g +\mathrm{Hess}\psi\geq K$.

\begin{Thm}[Sobolev-to-Lipschitz, Thm.~\ref{t:SL}]\label{t:IntroSL} 
Let~$(\mcX,\cdc,\mssd)$ be the \MLDS arising from a weighted Riemannian manifold with Ricci curvature bounded below.
Then, every~$u\in \dom{\EE{\dUpsilon}{\QP}}$ with~$\SF{\dUpsilon}{\QP}(u)\in L^\infty(\QP)$ has a $\mssd_\dUpsilon$-Lipschitz $\QP$-represen\-tative~$\rep{u}$ with~$\Li[\mssd_\dUpsilon]{\rep u}\leq \sqrt{\norm{\SF{\dUpsilon}{\QP}(u)}_{L^\infty}}$.
\end{Thm}

In light of the Sobolev-to-Lipschitz property, we further obtain (see Cor.~\ref{c:IrreducibilityMixed} for the precise statement)
\begin{Cor}\label{c:IntroErgo}
Let~$(\mcX,\cdc,\mssd)$ be the \MLDS arising from a weighted Riemannian manifold with Ricci curvature bounded below, and~$\QP=\QP_\lambda$ be a mixed Poisson measure.
Then, the following conditions are equivalent:
\begin{enumerate}[$(i)$]
\item\label{i:c:IntroErgodicity:1} $\QP=\PP$ is a Poisson measure, i.e.\ the L\'evy measure~$\lambda$ is a Dirac mass;
\item\label{i:c:IntroErgodicity:2} the form~$\EE{\dUpsilon}{\QP}$ is irreducible, i.e.\ it admits no non-trivial invariant set;
\item\label{i:c:IntroErgodicity:3} $\QP\text{-}\essinf_{\gamma\in \Lambda_1}\inf_{\eta\in\Lambda_2} \mssd_\dUpsilon(\gamma,\eta)<\infty$ for each~$\Lambda_i$ with~$\QP\Lambda_i>0$ for~$i=1,2$;
\end{enumerate}
\end{Cor}

Corollary~\ref{c:IntroErgo} completes the equivalence~\ref{i:c:IntroErgodicity:1}$\iff$\ref{i:c:IntroErgodicity:2} obtained in~\cite{AlbKonRoe98}, further characterizing the ergodicity of~$\EE{\dUpsilon}{\QP}$ in terms of the set-distance induced by~$\mssd_\dUpsilon$.
In particular, it establishes an insightful bridge between measure-theoretical properties (e.g., the $\QP$-size of sets) with geometric properties (e.g., the $\mssd_\dUpsilon$-shape of sets).

\medskip

As a second application, we obtain the following regularizing property of the heat semigroup~$\TT{\dUpsilon}{\PP}_\bullet$.

\begin{Thm}[$L^\infty$-to-Lipschitz regularization, Thm.~\ref{t:LinftyLip}]\label{t:IntroLinftyLip} 
Let~$(\mcX,\cdc,\mssd)$ be the \MLDS arising from a weighted Riemannian manifold with Ricci curvature bounded below.
Then,~$\TT{\dUpsilon}{\QP}_t$ maps~$L^\infty(\QP)$ into~$\Lip_b(\mssd_\dUpsilon)$.
\end{Thm}

Theorem~\ref{t:IntroLinftyLip} provides a natural regularization of functions in $L^\infty(\PP)$ by \emph{smoothing}.
This kind of smoothing property of the semigroup is a substitute for the $L^\infty(\PP)$-to-$\Cb(\T_\mrmv)$-regularization property, which should not be expected in infinite-dimensional settings.

\medskip

Both Theorem~\ref{t:IntroSL} and Theorem~\ref{t:IntroLinftyLip} hold in greater generality for \MLDS's satisfying suitable assumptions; see the corresponding statements in~\S\ref{s:ApplicationsConfig2}.

\paragraph{Overview}
It has been recently shown by the second named author in~\cite{Suz22} that measures~$\QP$ \emph{rigid in number} in the sense of Ghosh--Peres~\cite{GhoPer17} satisfy the same characterization of irreducibility as in Corollary~\ref{c:IntroErgo}.
This case is antithetic to the one presented here, in that (mixed) Poisson measures are ---~as much as possible~--- \emph{not} rigid in number.

\section{Background}

\subsection{Base spaces} We briefly recall the necessary definitions, referring to~\cite{LzDSSuz21} for a more exhaustive discussion.

\subsubsection{Topological local structures}
Let~$(X,\A,\mssm)$ be a measure space. We denote by~$\A_\mssm\subset \A$ the algebra of sets of finite $\mssm$-measure, and by~$(X,\A^\mssm,\hat\mssm)$ the Carath\'eodory completion of~$(X,\A,\mssm)$ w.r.t.~$\mssm$.
A sub-ring $\msE\subset \A$ is a \emph{ring ideal} of~$\A$, if it is closed under finite unions and intersections, and~$A\cap E\in \msE$ for every~$A\in \A$ and~$E\in \msE$.

\begin{defs}[Topological local structures]\label{d:TLSConfig2}
Let~$X$ be a non-empty set. A \emph{topological local structure} is a tuple $\mcX\eqdef (X,\T,\A,\mssm,\msE)$ so that
\begin{enumerate}[$(a)$]
\item\label{i:d:TLSConfig2:1} $(X,\T)$ is a separable metrizable Luzin topological space, with Borel $\sigma$-algebra~$\Bo{\T}$;

\item\label{i:d:TLSConfig2:2} $\A$ is a $\sigma$-algebra on~$X$, and $\mssm$ is an atomless Radon measure on~$(X,\T,\A)$ with full $\T$-support;

\item\label{i:d:TLSConfig2:3} $\Bo{\T}\subset \A\subset \Bo{\T}^\mssm$ and $\A$ is \emph{$\mssm$-essentially countably generated}, i.e.\ there exists a countably generated $\sigma$-sub\-al\-ge\-bra~$\A_0$ of~$\A$ so that for every~$A\in \A$ there exists $A_0\in \A_0$ with~$\mssm(A\triangle A_0)=0$.

\item\label{i:d:TLSConfig2:4} $\msE \subset \A_\mssm$ is a \emph{localizing ring}, i.e.\ it is a ring ideal of~$\A$, and there exists a \emph{localizing sequence}~$\seq{E_h}_h\subset \msE$ so that $\msE=\cup_{h\geq 0} (\A\cap E_h)$.

\item\label{i:d:TLSConfig2:5} for every~$x\in X$ there exists a $\T$-neighborhood~$U_x$ of~$x$ so that~$U_x\in\msE$.
\end{enumerate}
\end{defs}

\begin{rem}
Definition~\ref{d:TLSConfig2} above is equivalent to the definition of topological local structure in~\cite[Dfn.~3.1]{LzDSSuz21} in light of~\cite[Rmk.~3.2]{LzDSSuz21}.
For the importance of considering Luzin spaces (rather than, e.g., Polish spaces), see~\cite[Rmk.~\ref{r:Luzin}]{LzDSSuz21}.
\end{rem}

Let~$\mcX$ be a topological local structure.
Write~$\mcL^\infty(\mssm)$ or~$\Sb(X)$ for the Banach lattice of real-valued \emph{bounded} (as opposed to: $\mssm$-essentially bounded) $\A$-measurable functions.
It will be of great importance to distinguish between measurable functions on~$\mcX$ and their $\mssm$-classes.

\begin{notat}\label{n:ClassesConfig2}
We denote $\mssm$-classes of functions by~$f$,~$g$, etc., measurable representatives by~$\rep f$,~$\rep g$, etc.
Whenever~$f$ is an $\mssm$-class,~$\rep f$ is taken to be a representative of~$f$. Whenever~$\rep f$ is a measurable function, possibly undefined on an $\mssm$-negligible set,~$f=\ttclass[\mssm]{\rep f}$ is taken to be the corresponding $\mssm$-class.
We shall adopt the same convention for other objects, thus writing e.g.,~$\Dz\subset L^\infty(\mssm)$, resp.~$\rep\Dz\subset \mcL^\infty(\mssm)$.
Any $\mssm$-class~$f$ defined on a topological local structure has at most \emph{one} continuous $\mssm$-representa\-tive, by the properties of~$\mssm$.
Everywhere in the following, if~$f$ has a continuous $\mssm$-representative, say~$\rep f$, we shall always assume to be concerned with~$\rep f$; in this case we drop the notation for representatives, thus writing~$f$ for both the $\mssm$-class \emph{and} the continuous $\mssm$-representative.
\end{notat}

\begin{defs}[Function spaces]
Write $\Sb(\msE)$ for the space of bounded $\A$-measurable $\msE$-\emph{eventual\-ly vanishing} functions on~$X$, viz.
\begin{align*}
\Sb(\msE)\eqdef \set{f\in\Sb(X) : f\equiv 0 \text{~~on } E^\complement \text{~~for some }E\in\msE}  \fstop
\end{align*}
Write $\Cb(X)$ for the space of $\T$-continuous bounded functions on~$X$, and~$\Cz(\msE)\eqdef \Cb(X)\cap \Sb(\msE)$ for the space of $\T$-continuous bounded $\msE$-\emph{eventually vanishing} functions on~$X$; $\Ccompl(\msE)$ for the space of $\T$-continuous bounded functions on~$X$ \emph{vanishing at $\msE$-infinity}, viz.
\begin{equation*}
\Ccompl(\msE)\eqdef \set{f\in\Cb(\T): \forall \eps>0 \quad \exists E_\eps\in\msE : \abs{f(x)}<\eps \quad x\in E_\eps^\complement}\fstop
\end{equation*}
\end{defs}

\subsubsection{Topological local diffusion spaces}
We write~$\Cbinfty(\R^k)$ for the space of real-valued bounded smooth functions on~$\R^k$ with bounded derivatives of all orders.
For~$\rep f_i\in \mcL^\infty(\mssm)$,~$i\leq k$, set
\begin{align*}
\mbff\eqdef \seq{f_1,\dotsc, f_k}\in L^\infty(\mssm;\R^k)\comma \qquad \text{resp.}\qquad \rep\mbff \eqdef \tseq{\rep f_1,\dotsc, \rep f_k}\in\mcL^\infty(\mssm;\R^k) \fstop
\end{align*}

\begin{defs}[Square field operators]\label{d:SFConfig2}
Let~$\mcX$ be a topological local structure. For every~$\phi\in \Cb^\infty(\R^k)$ set $\phi_0\eqdef \phi-\phi(\zero)$. A \emph{square field operator on $\mcX$} is a pair~$(\cdc, \Dz)$ so that, for every~$\mbff\in\Dz^\otym{k}$, every~$\phi\in \Cb^\infty(\R^k)$ and every~$k\in \N_0$,
\begin{enumerate}[$(a)$]
\item\label{i:d:SFConfig2:1} $\Dz$ is a subalgebra of~$L^\infty(\mssm)$ with~$\phi_0\circ\mbff\in \Dz$;

\item\label{i:d:SFConfig2:2} $\cdc\colon \Dz^\otym{2}\longrar L^\infty(\mssm)$ is a symmetric non-negative definite bilinear form;

\item\label{i:d:SFConfig2:3} $(\cdc,\Dz)$ satisfies the following \emph{diffusion property}
\begin{align}\label{eq:i:d:SFConfig2:3}
\cdc\tparen{\phi_0\circ \mbff, \psi_0\circ \mbfg}=
\sum_{i,j}^k (\partial_i \phi)\circ \mbff \cdot (\partial_j\psi)\circ\mbfg \cdot \cdc(f_i, g_j) \as{\mssm}\fstop
\end{align}
\end{enumerate}

Let the analogous definition of a \emph{pointwise defined square field operator}~$(\rep\cdc, \rep\Dz)$ be given, with~$\rep\Dz\subset \mcL^\infty(\mssm)$ in place of~$\Dz\subset L^\infty(\mssm)$, and with~\eqref{eq:i:d:SFConfig2:3} to hold pointwise (as opposed to: $\mssm$-a.e.).
\end{defs}

A pointwise defined square field operator~$(\rep\cdc, \rep\Dz)$ defines a square field operator~$(\cdc,\Dz)$ on $\mssm$-classes as soon as, cf.~\cite[p.~282]{MaRoe00},
\begin{align*}
\rep\cdc(\rep f,\rep g)=0\comma \qquad \rep f,\rep g\in \rep\Dz\comma \rep f \equiv 0 \as{\mssm}\fstop
\end{align*}

\begin{defs}[Topological local diffusion spaces,~{\cite[Dfn.~3.7]{LzDSSuz21}}]\label{d:TLDSConfig2}
A \emph{topological local diffusion space} (in short: \TLDS) is a pair $(\mcX,\cdc)$ so that
\begin{enumerate}[$(a)$]
\item\label{i:d:TLDSConfig2:1} $\mcX$ is a topological local structure;
\item\label{i:d:TLDSConfig2:2} $\Dz\subset \Cz(\msE)$ is a subalgebra of~$\Cz(\msE)$ generating the topology~$\T$ of~$\mcX$;
\item\label{i:d:TLDSConfig2:3} $\cdc\colon \Dz^\otym{2}\rar \class[\mssm]{\Sb(\msE)}$ is a square field operator;
\item\label{i:d:TLDSConfig2:4} the bilinear form~$(\EE{X}{\mssm},\Dz)$ defined by
\begin{align*}
\EE{X}{\mssm}(f,g)\eqdef \int \cdc(f,g) \diff\mssm \comma \qquad f, g\in \Dz\comma
\end{align*}
is closable and densely defined in~$L^2(\mssm)$;
\item\label{i:d:TLDSConfig2:5} its closure~$\tparen{\EE{X}{\mssm},\dom{\EE{X}{\mssm}}}$ is a quasi-regular Dirichlet form in the sense of e.g.~\cite{CheMaRoe94}.
\end{enumerate}
\end{defs}

We stress that the square field operator~$\cdc$ on a \TLDS takes values in a space of $\mssm$-classes, \emph{not} representatives.
Let us introduce some further notation relative to Definition~\ref{d:TLDSConfig2}\ref{i:d:TLDSConfig2:5}.

\begin{notat}\label{n:FormConfig2} Let~$(\mcX,\cdc)$ be a~\TLDS.
\begin{enumerate*}[$(a)$]
\item\label{i:n:FormConfig2:1} It is readily verified that~$\tparen{\EE{X}{\mssm},\dom{\EE{X}{\mssm}}}$ admits square field operator~$\tparen{\SF{X}{\mssm},\dom{\SF{X}{\mssm}}}$ with $\dom{\SF{X}{\mssm}}=\dom{\EE{X}{\mssm}}\cap L^\infty(\mssm)$, and extending~$\tparen{\cdc,\Dz}$.
Further let:
\item\label{i:n:FormConfig2:3} $\tparen{\LL{X}{\mssm}, \dom{\LL{X}{\mssm}}}$ be the generator of~$\tparen{\EE{X}{\mssm},\dom{\EE{X}{\mssm}}}$;
\linebreak
\item\label{i:n:FormConfig2:4} $\TT{X}{\mssm}_\bullet\eqdef \tseq{\TT{X}{\mssm}_t}_{t\geq 0}$ be the semigroup of~$\tparen{\LL{X}{\mssm}, \dom{\LL{X}{\mssm}}}$, defined on~$L^p(\mssm)$ for every~$p\in [1,\infty)$.
\item by~$\hh{X}{\mssm}_\bullet\eqdef \tseq{\hh{X}{\mssm}_t(\emparg,\diff\emparg)}_{t\geq 0}$ the corresponding Markov kernel of measures, satisfying
\end{enumerate*}
\begin{align}\label{eq:n:FormConfig2:1}
\tparen{\TT{X}{\mssm}_t f}(x)=\int f(y)\, \hh{X}{\mssm}_t(x,\diff y)\comma \quad \forallae{\mssm} x\in X\comma \qquad f\in L^2(\mssm)\comma t\geq 0 \semicolon
\end{align}

Finally, let
\begin{enumerate*}[$(a)$]\setcounter{enumi}{5}
\item\label{i:n:FormConfig2:6}$\domext{\EE{X}{\mssm}}$ be the extended Dirichlet space of~$\tparen{\EE{X}{\mssm},\dom{\EE{X}{\mssm}}}$, i.e.\ the space of $\mssm$-classes of functions~$f\colon X\rar \R$ so that there exists an $(\EE{X}{\mssm})^{1/2}$-fun\-da\-men\-tal sequence $\seq{f_n}_n\subset \dom{\EE{X}{\mssm}}$ with $\nlim f_n=f$ $\mssm$-a.e..
The form~$\EE{X}{\mssm}$ naturally extends to a quadratic form on~$\domext{\EE{X}{\mssm}}$, denoted by the same symbol~$\EE{X}{\mssm}$, and we always consider~$\domext{\EE{X}{\mssm}}$ as endowed with this extension;
\item\label{i:n:FormConfig2:7}$\domext{\SF{X}{\mssm}}\eqdef\domext{\EE{X}{\mssm}}\cap L^\infty(\mssm)$ be the extended space of~$\tparen{\SF{X}{\mssm},\dom{\SF{X}{\mssm}}}$, endowed with the non-relabeled extension of~$\SF{X}{\mssm}$.
\end{enumerate*}
\end{notat}

\subsubsection{Extended-metric local diffusion spaces}
Let~$\mcX$ be a topological local structure, and $\mssd\colon X^\tym{2}\rar[0,\infty]$ be an extended distance.

\begin{defs}[Extended-metric local structure,~{\cite[Dfn.~\ref{d:EMLS}]{LzDSSuz21}}]\label{d:EMLSConfig2}
We say that~$(\mcX,\mssd)$ is an \emph{extend\-ed-metric local structure}, if
\begin{enumerate}[$(a)$]
\item\label{i:d:EMLSConfig2:0} $\mcX$ is a topological local structure in the sense of Definition~\ref{d:TLSConfig2};

\item\label{i:d:EMLSConfig2:1} $(X,\T,\mssd)$ is a \emph{complete} extended metric-topological space in the sense of \cite[Dfn.~\ref{d:AES}]{LzDSSuz21};

\item\label{i:d:EMLSConfig2:2} $\msE=\msE_\mssd\eqdef \Bdd{\mssd}\cap \A$ is the localizing ring of $\A$-measurable $\mssd$-bounded sets.
\end{enumerate}
If~$\mssd$ is finite, then it metrizes~$\T$,~\ref{i:d:EMLSConfig2:1} reduces to the requirement that~$(X,\mssd)$ be a complete metric space, and we say that~$(\mcX,\mssd)$ is a \emph{metric local structure}.
\end{defs}

The Definition~\ref{d:TLDSConfig2} of \TLDS can be now combined with that of a metric local structure above.

\begin{defs}[Metric local diffusion spaces,~{\cite[Dfn.~\ref{d:EMLDS}]{LzDSSuz21}}]\label{d:EMLDSConfig2}
An (\emph{extended}) \emph{metric local diffusion space} (in short: \parEMLDS) is a triple $(\mcX,\cdc,\mssd)$ so that~$(\mcX,\mssd)$ is an (extended) metric local structure, and~$(\mcX,\cdc)$ is a \TLDS.
\end{defs}

\subsubsection{Intrinsic distances and maximal functions}
We recall some basic properties of intrinsic distances and maximal functions of Dirichlet spaces.
We assume the reader to be familiar with `quasi-notions' and broad local spaces in the sense of Dirichlet forms theory, see e.g.~\cite{Kuw98, LzDSSuz20}.

Let~$(\mcX,\cdc)$ be a \TLDS, and~$\tparen{\EE{X}{\mssm},\dom{\EE{X}{\mssm}}}$ be the (quasi-regular strongly local) Dirichlet form with square field operator~$\SF{X}{\mssm}$.
We denote by~$\domloc{\EE{X}{\mssm}}$ its broad local domain, and consider the non-relabeled extension of~$\cdc$ to~$\domloc{\EE{X}{\mssm}}$; see e.g.~\cite[\S2.4]{LzDSSuz20}.
The \emph{broad local space of functions with $\mssm$-uniformly bounded $\EE{X}{\mssm}$-energy} is the space
\begin{align*}
\DzLoc{\mssm}\eqdef \set{f\in \domloc{\EE{X}{\mssm}}: \cdc(f)\leq 1 \as{\mssm}}\fstop
\end{align*}
We additionally set~$\DzLocB{\mssm,\T}\eqdef \DzLoc{\mssm}\cap \Cb(\T)$ and $\DzB{\mssm,\T}\eqdef \DzLocB{\mssm,\T}\cap \dom{\EE{X}{\mssm}}$.

\begin{defs}[Intrinsic distance]
Let~$(\mcX,\cdc)$ be a \TLDS. The \emph{intrinsic distance associated to~$\EE{X}{\mssm}$} is the extended pseudo-distance
\begin{align*}
\mssd_\mssm(x,y)\eqdef \sup_{f\in \DzLocB{\mssm,\T}} \abs{f(x)-f(y)} \fstop
\end{align*}
\end{defs}
For more information on intrinsic distances, we refer the reader to~\cite[\S{2.6}]{LzDSSuz20}.

\paragraph{Maximal functions}
A second class of functions playing a role in the analysis of Dirichlet spaces is that of \emph{maximal functions}.

\begin{defs}[Maximal functions,~{\cite{HinRam03}}]\label{d:MaximalFunctionConfig2}
For each~$A\in \A$ there exists an $\mssm$-a.e.\ unique function $\hr{\mssm, A}\colon X\to[0,\infty]$ so that, for each~$r>0$,
\begin{align*}
\hr{\mssm, A}\wedge r = \mssm\text{-}\esssup \set{f: f\in \DzLocB{\mssm}\comma f\equiv 0 \as{\mssm} \text{ on~$A$}\comma f\leq r \as{\mssm}} \fstop
\end{align*}
\end{defs}

The study of maximal functions is motivated by the following result.

\begin{thm}[Ariyoshi--Hino~{\cite[Thm.~5.2(i)]{AriHin05}}]
Let~$(\mcX,\cdc)$ be a \TLDS, and~$\nu\sim \mssm$ be any probability measure on~$(X,\A)$.
Further let~$A\in \A$ be so that~$\mssm A\in(0,\infty)$, and set~$u_t\eqdef -2t \log \TT{X}{\mssm}_t\car_A$.
Then, $\nu$-$\lim_{t\downarrow 0} u_t\cdot \car_{\set{u_t <\infty}} = \hr{\mssm, A}^2$.
\end{thm}

\paragraph{Regularized distances} 
Whereas intrinsic distances play an important role in the study of the relative (quasi-regular) Dirichlet forms (and of the associated Markov processes), many properties of such distances, and especially of their point-to-set counterparts, are typically quite difficult to establish, especially when~$\mssd_\mssm$ does not generate the underlying topology~$\T$ (which will be the case for configuration spaces), see e.g.~\cite{LzDSSuz20}. 
In order to overcome such difficulties, we introduce the following regularization.

\begin{defs}[Regularized intrinsic point-to-set distance]\label{d:RegIntrinsicD}
Let~$(\mcX,\cdc)$ be a \TLDS. The \emph{regularized intrinsic point-to-set distance associated to~$\EE{X}{\mssm}$} is the function
\begin{align*}
\mssd^\reg_{\mssm,A}(x)\eqdef \sup_{f\in \DzLocB{\mssm,\T}}\inf_{y\in A} \abs{f(x)-f(y)} \comma \qquad A\subset X \fstop
\end{align*}
As usual for (extended pseudo-)distances, we denote by~$\mssd^\reg_{\mssm,A}$ both \emph{the} representative (defined as above) and the~$\mssm$-class of the function.
\end{defs}

\begin{rem}
We stress that the regularized distance~$\mssd^\reg_{\mssm,A}$ is different from the point-to-set distance~$\mssd_\mssm(\emparg,A)$ defined by~$\mssd_\mssm$, cf.~\eqref{eq:Point-to-SetDistance}.
In particular, it is in general not true that~$\mssd_\mssm(\emparg, A)$ is $\A$-measurable, whereas~$\mssd^\reg_{\mssm,A}$ is $\A$-measurable by definition, since the infimum and supremum in the definition are exchanged.
The $\A$-measurability of~$\mssd^\reg_{\mssm,A}$, shown in Proposition~\ref{p:RegIntrinsic}\ref{i:p:RegIntrinsic:3} below, is indeed one main motivation in introducing regularized intrinsic point-to-set distances.
\end{rem}

Let us collect some properties of~$\mssd^\reg_{\mssm,A}$.

\begin{prop}[Properties of~$\mssd^\reg_{\mssm,A}$]\label{p:RegIntrinsic}
Let~$(\mcX,\cdc)$ be a \TLDS. Then, for every~$A\subset X$ (not necessarily measurable)
\begin{enumerate}[$(i)$]
\item\label{i:p:RegIntrinsic:1} $\mssd^\reg_{\mssm, \set{y}}(x)=\mssd_\mssm(x,y)$ for each~$x,y\in X$, and~$\mssd^\reg_{\mssm,A}\leq \mssd_\mssm(\emparg,A)$;

\item\label{i:p:RegIntrinsic:2} for every~$y\in X$ and every net of sets~$\seq{A_\alpha}_\alpha$ with~$A_\alpha\downarrow_\alpha \set{y}$,
\begin{align*}
\mssd^\reg_{\mssm, A_\alpha} \uparrow_\alpha \mssd_\mssm(y,\emparg) \fstop
\end{align*}

\item\label{i:p:RegIntrinsic:3} $\mssd^\reg_{\mssm,A}$ is $\T$-l.s.c.; in particular, it is $\Bo{\T}$-measurable;

\item\label{i:p:RegIntrinsic:4} $\mssd^\reg_{\mssm,A}\leq \hr{\mssm,A}$ $\mssm$-a.e.\ whenever the latter is well-defined (e.g.\ if~$A\in\A$).
\end{enumerate}
\begin{proof}
The first assertion in~\ref{i:p:RegIntrinsic:1} is straightforward; the second follows by the standard $\inf$-$\sup$ inequality.
\ref{i:p:RegIntrinsic:2} Since~$\mssd^\reg_{\mssm,A}$ is monotone increasing in decreasing~$A$, the assertion follows by the exchange of isotone monotone limits.
\ref{i:p:RegIntrinsic:3} For~$A\subset X$ and~$f\in \Cb(\T)$ set $\mssd_{f,A}\eqdef \inf_{z\in A} \abs{f(\emparg)-f(z)}$, and note that
\begin{align}\label{eq:NormalContraction}
\mssd_{f,A}(x)-\mssd_{f,A}(y) \leq \abs{f(x)-f(y)}\comma \qquad x,y\in X\comma
\end{align}
thus the continuity of~$\mssd_{f,A}$ follows from that of~$f$. The $\T$-lower semi-continuity of~$\mssd^\reg_{\mssm, A}$ follows, since the supremum over an arbitrary family of $\T$-(l.s.-)continuous functions is $\T$-l.s.c..

\ref{i:p:RegIntrinsic:4} 
Let~$\mssd_{f,A}$ be as above, and note that~$\mssd_{f,A}= 0$ on~$A$. 
Also note that
\[
\mssd_{f,x_0}(\emparg)\eqdef \abs{f(\emparg)-f(x_0)}
\]
defines a pseudo-metric on~$X$, separable by separability of~$\T$ and $\T$-continuity of~$f$.
By~\cite[Thm.~1.1]{LzDSSuz20}, every $\mssd_{f,x_0}$-Lipschitz function~$g$ satisfies~$g\in \DzLocB{\mssm}$.
By~\eqref{eq:NormalContraction}, this applies to~$g\eqdef \mssd_{f,A}$.
Thus, we have that~$\mssd_{f,A}\in \DzLocB{\mssm,\T}$ whenever~$f\in\DzLocB{\mssm,\T}$, where the $\T$-continuity follows from~\eqref{eq:NormalContraction} and the $\T$-continuity of~$f$.
The same assertion holds for~$\mssd_{f,A}\wedge r$ for any~$r\geq 0$.
Letting~$r$ to infinity we thus have~$\mssd_{f,A}\leq\hr{\mssm, A}$ for every~$f\in\DzLocB{\mssm,\T}$, by maximality of~$\hr{\mssm, A}$.
The conclusion follows by extremizing the inequality over~$f\in\DzLocB{\mssm,\T}$.
\end{proof}
\end{prop}

\subsubsection{Rademacher and Sobolev-to-Lipschitz properties}
In this section, we focus on the interplay between the diffusion-space structure and the metric structure of an \parEMLDS.
We refer the reader to~\cite{LzDSSuz20} for a detailed discussion on the subject.

\begin{defs}[Rademacher and Sobolev-to-Lipschitz properties]\label{d:RadSL}
We say that an \EMLDS \linebreak $(\mcX,\cdc,\mssd)$ has:
\begin{itemize}
\item[($\Rad{\mssd}{\mssm}$)] the \emph{Rademacher property} if, whenever~$\rep f\in \Lipu(\mssd,\A)$, then~$f\in \DzLoc{\mssm}$;
\item[($\dRad{\mssd}{\mssm}$)] the \emph{distance-Rademacher property} if~$\mssd\leq \mssd_\mssm$;

\item[($\SL{\mssm}{\mssd}$)] the \emph{Sobolev--to--Lipschitz property} if each~$f\in\DzLoc{\mssm}$ has an $\mssm$-representat\-ive $\rep f\in\Lip^1(\mssd,\A)$;

\item[($\dcSL{\mssd}{\mssm}{\mssd}$)] the \emph{$\mssd$-continu\-ous-Sobolev--to--Lipschitz property} if each~$f\in \DzLoc{\mssm}$ having a $\mssd$-continu\-ous $\A$-measurable representative~$\rep f$ also has a representative $\reptwo f\in\Lip^1(\mssd,\A^\mssm)$ (possibly,~$\reptwo f\neq \rep f$);

\item[($\cSL{\T}{\mssm}{\mssd}$)] the \emph{continuous-Sobolev--to--Lipschitz property} if each~$f\in \DzLoc{\mssm,\T}$ satisfies $f\in\Lip^1(\mssd,\T)$;

\item[($\dSL{\mssd}{\mssm}$)] the \emph{distance Sobolev-to-Lipschitz property} if~$\mssd\geq \mssd_\mssm$.
\end{itemize}
\end{defs}

We refer the reader to~\cite[Rmk.s~3.2, 4.3]{LzDSSuz20} for comments on the terminology and to~\cite[Lem.~3.6, Prop.~4.2]{LzDSSuz20} for the interplay of all such properties. In the setting of~\EMLDS's, they reduce to:
\begin{equation}\label{eq:EquivalenceRadStoLConfig2}
\begin{aligned}
(\Rad{\mssd}{\mssm}) \Longrightarrow &\ (\dRad{\mssd}{\mssm}) && \text{\cite[Lem.~3.6]{LzDSSuz20}}\comma
\\
(\SL{\mssm}{\mssd}) \Longrightarrow &\ (\cSL{\mssm}{\T}{\mssd}) \iff (\dSL{\mssd}{\mssm}) && \text{\cite[Prop.~4.2]{LzDSSuz20}} \fstop
\end{aligned}
\end{equation}

The delicate interplay between the Rademacher and Sobolev-to-Lipschitz properties and maximal functions was investigated in the setting of quasi-regular Dirichlet spaces by the authors, in~\cite{LzDSSuz20}.
For the case of configuration spaces, see also~\cite[Rmk.s~4.12 and~5.24]{LzDSSuz21}.

In the setting of~\MLDS (\emph{not}: \EMLDS), the Sobolev-to-Lipschitz property~$(\SL{\mssm}{\mssd})$ implies the irreducibility of the form~$\tparen{\EE{X}{\mssm},\dom{\EE{X}{\mssm}}}$.

\begin{prop}\label{p:IrreducibilityBase}
Let~$(\mcX,\cdc,\mssd)$ be an \MLDS satisfying~$(\SL{\mssm}{\mssd})$. Then, $\tparen{\EE{X}{\mssm},\dom{\EE{X}{\mssm}}}$ is irreducible.
\begin{proof}
Since~$(\mcX,\cdc,\mssd)$ is an~\MLDS, the distance~$\mssd$ metrizes~$\T$.
Thus, the proof of~\cite[Cor.~4.6]{LzDSSuz21a} applies \emph{verbatim} having care to substitute the use of~\cite[Thm.~4.21]{LzDSSuz20} with~\cite[Lem.~4.19]{LzDSSuz20}.
\end{proof}
\end{prop}

\begin{rem}
The assertion in Proposition~\ref{p:IrreducibilityBase} is not true in the case of \EMLDS's, see Example~\ref{e:SLvsIrrMixedPoisson}.
\end{rem}

\subsection{Configuration spaces}
Let~$\mcX$ be a topological local structure.
A (\emph{multiple}) \emph{configuration} on~$\mcX$ is any $\overline\N_0$-valued measure~$\gamma$ on~$(X,\A)$, finite on~$E$ for every~$E\in\msE$. By assumption on~$\mcX$, cf.\ e.g.~\cite[Cor.~6.5]{LasPen18},
\begin{align*}
\gamma= \sum_{i=1}^N \delta_{x_i} \comma \qquad N\in \overline\N_0\comma \qquad \seq{x_i}_{i\leq N}\subset X \fstop
\end{align*}
In particular, we allow for~$x_i=x_j$ if~$i\neq j$.
Write~$\gamma_x\eqdef\gamma\!\set{x}$,~$x\in \gamma$ whenever~$\gamma_x>0$, and~$\gamma_A\eqdef \gamma\mrestr{A}$ for every~$A\in\A$.

The \emph{multiple configuration space}~$\dUpsilon=\dUpsilon(X,\msE)$ is the space of all (multiple) configurations over~$\mcX$. The \emph{configuration space} is the space
\begin{align*}
\Upsilon(\msE)=&\Upsilon(X,\msE)\eqdef\set{\gamma\in\dUpsilon: \gamma_x\in \set{0,1} \text{~~for all } x\in X} \fstop
\end{align*}
The $N$-particles multiple configuration spaces, resp.~configuration spaces, are
\begin{equation*}
\dUpsilon^\sym{N}\ \eqdef \set{\gamma\in \dUpsilon : \gamma X=N}\comma 
\qquad
\text{resp.} \qquad \Upsilon^\sym{N}\ \eqdef \dUpsilon^\sym{N}\cap \Upsilon\comma 
\qquad N\in\overline\N_0 \fstop
\end{equation*}
Let the analogous definitions of~$\dUpsilon^\sym{\geq N}(\msE)$, resp.~$\Upsilon^\sym{\geq N}(\msE)$, be given.

For fixed~$A\in\A$ further set~$\msE_A\eqdef \set{E\cap A: E\in\msE}$ and
\begin{align*}
\pr^A\colon \dUpsilon\longrar \dUpsilon(\msE_A)\colon \gamma\longmapsto \gamma_A\comma \qquad A\in\A \fstop
\end{align*}

Finally set~$\dUpsilon(E)\eqdef \dUpsilon(\msE_E)=\dUpsilon(\A_E)$ for all~$E\in\msE$, and analogously for~$\Upsilon(E)$.

We endow~$\dUpsilon$ with the $\sigma$-algebra $\A_\mrmv(\msE)$ generated by the functions $\gamma\mapsto \gamma E$ with~$E\in\msE$.

\begin{defs}[Concentration set]
For~$E\in\msE$ we define the $n$-concentration sets of~$E$ as
\begin{equation}\label{eq:ConcentrationSetConfig2}
\Xi_{=n}(E)\eqdef\ \set{\gamma\in\dUpsilon: \gamma E= n} \comma
\end{equation}
and similarly for~`$\geq$' or~`$\leq$' in place of~`$=$'.
Analogously to the notation established for configuration spaces, we write~$\Xi^\sym{\infty}_{=n}(E)=\Xi_{=n}(E)\cap \dUpsilon^\sym{\infty}$.
\end{defs}

\subsubsection{Poisson random measures}
Contrary to~\cite{LzDSSuz21}, in this work we restrict our attention to configuration spaces endowed with Poisson measures.

For~$E\in\msE$ set~$\mssm_E\eqdef \mssm\mrestr{E}$.
Let~$\mfS_n$ be the $n^\text{th}$ symmetric group, and denote by
\begin{equation*}
\pr^\sym{n}\colon E^\tym{n}\rar E^\sym{n}\eqdef E^\tym{n}/\mfS_n
\end{equation*}
the quotient projection, and by~$\mssm_E^\sym{n}$ the quotient measure~$\mssm_E^\sym{n}\eqdef \pr^\sym{n}_\pfwd \mssm_E^\hotym{n}$.
The space~$E^\sym{n}$ is naturally isomorphic to~$\dUpsilon^\sym{n}(E)$.
Under this isomorphism, define the \emph{Poisson--Lebesgue measure of~$\PP_\mssm$} on~$\dUpsilon^{\sym{<\infty}}(E)$, cf.\ e.g.~\cite[Eqn.~(2.5),~(2.6)]{AlbKonRoe98}, by
\begin{align}\label{eq:PoissonLebesgueConfig2}
\PP_{\mssm_E}\eqdef e^{-\mssm E}\sum_{n=0}^\infty \frac{\mssm_E^\sym{n}}{n!} \fstop
\end{align}

\begin{defs}[Poisson measures]\label{d:PoissonConfig2}
The \emph{Poisson} (\emph{random}) \emph{measure~$\PP_\mssm$} with intensity~$\mssm$ is the unique probability measure on~$\tparen{\dUpsilon, \A_\mrmv(\msE)}$ satisfying either of the following equivalent conditions:
\begin{itemize}
\item the \emph{projective-limit characterization} 
$\pr^{E}_\pfwd \PP_\mssm=\PP_{\mssm_E}$ 
for every~$E\in\msE$.

\item the \emph{Mecke identity}~\cite[Satz~3.1]{Mec67}, viz.
\begin{align}\label{eq:MeckeConfig2}
\iint_{\dUpsilon\times X} u(\gamma, x) \diff\gamma(x) \diff\PP_\mssm(\gamma)= \iint_{\dUpsilon\times X} u(\gamma+\delta_x,x) \diff\mssm(x) \diff\PP_\mssm(\gamma)
\end{align}
for every bounded $\A_\mrmv(\msE)\hotimes \A$-measurable~$u\colon \dUpsilon\times X\rar \R$.

\item the \emph{Laplace transform characterization}, cf.~\cite[Thm.~3.9]{LasPen18},
\begin{align}\label{eq:LaplacePoissonConfig2}
\int_{\dUpsilon} e^{f^\trid\gamma} \diff\PP_\mssm(\gamma)=\exp\paren{\int_X (e^{f}-1) \diff\mssm}\comma \qquad f\in \Sb(X)^+\fstop
\end{align}
\end{itemize}
\end{defs}

As it is clear from the first characterization, we have that
\begin{itemize}
\item if~$\mssm X=\infty$, then~$\PP_\mssm$ is concentrated in~$\Upsilon^\sym{\infty}(\msE)$, viz.\ $\PP_\mssm \Upsilon^\sym{\infty}(\msE)=1$;

\item if~$\mssm X<\infty$, then~$\PP_\mssm$ is concentrated in~$\Upsilon^\sym{<\infty}(\msE)$, viz.\ $\PP_\mssm \Upsilon^\sym{<\infty}(\msE)=1$, and~$\PP_\mssm$ coincides with the right-hand side of~\eqref{eq:PoissonLebesgueConfig2} with~$X$ in place of~$E$
\end{itemize}

Everywhere in the following we omit the specification of the insity measure whenever apparent from context, thus writing~$\PP$ in place of~$\PP_\mssm$, and~$\PP_E$ in place of~$\PP_{\mssm_E}$. 

\subsubsection{Dirichlet forms}
Let us briefly recall the construction and main analytical properties of the Dirichlet form~\eqref{eq:IntroDirGenSemi} constructed in~\cite{LzDSSuz21}.
We specialize all the statements in~\cite{LzDSSuz21} to the case of our interest, namely that of Poisson measures.

\paragraph{Cylinder functions} We shall start by defining a core of \emph{cylinder functions} for the form \eqref{eq:IntroDirGenSemi}.
For~$\gamma\in \dUpsilon$ and~$\rep f\in\Sb(\msE)$ let~$\rep f^\trid\colon \dUpsilon\rar \R$ be defined as 
\begin{equation}\label{eq:TridConfig2}
\rep f^\trid\colon \gamma\longmapsto \int_X \rep f(x)\diff\gamma(x) 
\end{equation}
and set further
\begin{align*}
\rep\mbff^\trid\colon \gamma\longmapsto\tparen{\rep f_1^\trid\gamma,\dotsc, \rep f_k^\trid\gamma}\in \R^k\comma \qquad \rep f_1,\dotsc, \rep f_k\in \Sb(\msE)\fstop
\end{align*}

\begin{defs}[Cylinder functions on~$\dUpsilon$] Let~$\mcX$ be a topological local structure, and~$\rep D$ be a linear subspace of~$\Sb(\msE)$. We define the space of \emph{cylinder functions}
\begin{align*}
\Cyl{\rep D}\eqdef \set{\begin{matrix} \rep u\colon \dUpsilon\rar \R : \rep u=F\circ\rep\mbff^\trid \comma  F\in \mcC^\infty_b(\R^k)\comma \\ \rep f_1,\dotsc, \rep f_k\in \rep D\comma\quad k\in \N_0 \end{matrix}}\fstop
\end{align*}
\end{defs}

It is readily seen that cylinder functions of the form~$\Cyl{\Sb(\msE)}$ are $\A_\mrmv(\msE)$-measurable.
If~$\rep D$ generates the $\sigma$-algebra~$\A$ on~$X$, then~$\Cyl{\rep D}$ generates the $\sigma$-algebra~$\A_\mrmv(\msE)$ on~$\dUpsilon$. We stress that the representation of~$\rep u$ by~$\rep u=F\circ\rep\mbff$ is \emph{not} unique.

\paragraph{Lifted square field operators}
Let~$(\mcX,\cdc)$ be a~\TLDS, and recall that~$\rep\Dz\subset \Cb(X)$, that is, every $\mssm$-class in~$\Dz$ has a continuous representative.
We say that a pointwise defined square field operator~$\rep\cdc$ on~$\rep\Dz$ is \emph{compatible} with~$(\cdc,\Dz)$ if~$\tclass[\mssm]{\rep\cdc(\emparg)}=\cdc(\emparg)$ on~$\rep\Dz$.
As discussed in~\cite[\S\ref{ss:Liftings}]{LzDSSuz21}, every~\TLDS admits a compatible pointwise defined square field operator by the theory of \emph{liftings}.

We may now lift a pointwise defined square field operator~$\rep\cdc$ on~$\mcX$ to a pointwise defined square field operator on~$\dUpsilon$, by setting
\begin{equation*}
\begin{gathered}
\rep\cdc^{\dUpsilon}(\rep u, \rep v)(\gamma)\eqdef \sum_{i,j=1}^{k,m} (\partial_i F)(\rep\mbff^\trid\gamma) \cdot (\partial_j G)(\rep\mbfg^\trid\gamma) \cdot \rep\cdc(\rep f_i, \rep g_j)^\trid \gamma \comma
\\
\rep u=F\circ \rep\mbff^\trid\in \Cyl{\rep\Dz} \comma\qquad \rep v=G\circ \rep\mbfg^\trid\in \Cyl{\rep\Dz}\fstop
\end{gathered}
\end{equation*}

By~\cite[Lem.~1.2]{MaRoe00} the bilinear form~$\rep\cdc^\dUpsilon$ is well-defined on~$\Cyl{\rep\Dz}^\tym{2}$, in the sense that~$\rep\cdc^\dUpsilon(\rep u, \rep v)$ does not depend on the choice of representatives $\rep u=F\circ\rep\mbff$ and $\rep v=G\circ\rep\mbfg$ for~$\rep u$ and~$\rep v$.
We refer the reader to~\cite[\S.1]{MaRoe00} and~\cite[\S2.3.1]{LzDSSuz21} for details on the issue of well-posedness on cylinder functions.

Now, let us set
\begin{align*}
\EE{\dUpsilon}{\PP}(u,v)\eqdef \int_{\dUpsilon} \rep\cdc^{\dUpsilon}(\rep u, \rep v) \diff\PP\comma \qquad u,v\in \Cyl{\rep\Dz} \fstop
\end{align*}
Since~$\rep\cdc^{\dUpsilon}(\rep u, \rep v)(\gamma)$ is everywhere well-defined in the sense above,~$\EE{\dUpsilon}{\PP}$ is a well-defined bilinear form on the space of representatives~$\mcL^2(\PP)$.
It is shown in~\cite{LzDSSuz21} that~$\rep\cdc^\dUpsilon$ descends to a bilinear symmetric functional~$\cdc^\dUpsilon$ on the space~$\CylQP{\PP}{\rep\Dz}$ of $\PP$-classes of cylinder functions in~$\Cyl{\rep\Dz}$, which proves that~$\EE{\dUpsilon}{\PP}$ descends to a non-relabeled well-defined pre-Dirichlet form on~$L^2(\PP)$.
The closure of this form will be the main object of our study throughout this work.
\begin{prop}[Closability,~{\cite[Prop.~3.9]{LzDSSuz21}}]
Let~$(\mcX,\cdc)$ be a \TLDS.
Then,\linebreak $\tparen{\EE{\dUpsilon}{\PP},\CylQP{\PP}{\rep\Dz}}$ is well-defined, densely defined and closable on~$L^2(\PP)$.
Its closure $\tparen{\EE{\dUpsilon}{\PP},\dom{\EE{\dUpsilon}{\PP}}}$ is a Dirichlet form with carr\'e du champ operator~$\tparen{\SF{\dUpsilon}{\PP},\dom{\SF{\dUpsilon}{\PP}}}$ satisfying
\begin{equation*}
\SF{\dUpsilon}{\PP}(u,v)=\cdc^\dUpsilon(u,v) \as{\PP}\comma\qquad u,v\in \CylQP{\PP}{\rep\Dz}\fstop
\end{equation*}
\end{prop}
It is a consequence of the construction in~\cite{LzDSSuz21} that the form does not depend on the chosen compatible pointwise defined square field operator~$(\rep\cdc,\rep\Dz)$, but rather only on its value on $\mssm$-classes, i.e.\ only on~$(\cdc,\Dz)$.

We denote by
\begin{align*}
\tparen{\LL{\dUpsilon}{\PP},\dom{\LL{\dUpsilon}{\PP}}}\comma \qquad \text{resp.} \qquad \TT{\dUpsilon}{\PP}_\bullet\eqdef\tseq{\TT{\dUpsilon}{\PP}_t}_{t\geq 0}\comma
\end{align*}
the $L^2(\PP)$-generator, resp.~$L^2(\PP)$-semigroup, corresponding to~$\tparen{\EE{\dUpsilon}{\PP},\dom{\EE{\dUpsilon}{\PP}}}$.

The domain~$\dom{\EE{\dUpsilon}{\PP}}$ is in fact much larger than the class of cylinder functions~$\CylQP{\PP}{\rep\Dz}$, as recalled below.
Let~$\coK{\Dz}{1,\mssm}$ be the abstract linear completion of~$\Dz$ w.r.t.~the norm (cf.~\cite[p.~301]{MaRoe00})
\begin{align*}
\norm{\emparg}_{1,\mssm}\eqdef \EE{X}{\mssm}(\emparg)^{1/2} + \norm{\emparg}_{L^1(\mssm)} 
\comma \end{align*}
endowed with the unique (non-relabeled) continuous extension of~$\norm{\emparg}_{1,\mssm}$ to the completion $\coK{\Dz}{1,\mssm}$.

\begin{defs}[{\cite[\S4.2, p.~300]{MaRoe00}}] A function~$\rep u\colon \dUpsilon\rar \R\cup\set{\pm\infty}$ is called \emph{extended cylinder} if there exist~$k\in \N_0$, functions~$\rep f_1,\dotsc, \rep f_k$ with~$f_1,\dotsc,f_k\in \coK{\Dz}{1,\mssm}$, and a function~$F\in \Cb^\infty(\R^k)$, so that~$\rep u=F\circ\rep\mbff$.
We denote by~$\Cyl{\coK{\Dz}{1,\mssm}}$ the space of all extended cylinder functions, and by~$\CylQP{\PP}{\coK{\Dz}{1,\mssm}}$ the space of their $\PP$-representa\-tives.
\end{defs}

\begin{prop}[{\cite[Prop.~4.6]{MaRoe00}}, {\cite[Prop.~3.52]{LzDSSuz21}}]\label{p:ExtDomConfig2}
Let~$(\mcX,\cdc)$ be a \TLDS.
Then,
\begin{enumerate}[$(i)$]
\item\label{i:p:ExtDomConfig2:1}~$u=\tclass[\PP]{F\circ \rep\mbff^\trid}\in \CylQP{\PP}{\coK{\Dz}{1,\mssm}}$ is defined in~$\dom{\EE{\dUpsilon}{\PP}}$ and independent of the $\mssm$-rep\-re\-sen\-ta\-tives~$\rep f_i$ of~$f_i$;

\item\label{i:p:ExtDomConfig2:2} for~$\rep u=F\circ \rep\mbff^\trid$ and $\rep v=G\circ \rep\mbfg^\trid\in \Cyl{\coK{\Dz}{1,\mssm}}$,
\begin{align}\label{eq:d:LiftCdCConfig2}
\SF{\dUpsilon}{\PP}(u,v)(\gamma)=& \sum_{i,j=1}^{k,m} (\partial_i F)(\rep\mbff^\trid\gamma) \cdot (\partial_j G)(\rep\mbfg^\trid\gamma) \cdot \rep\cdc(\rep f_i, \rep g_j)^\trid \gamma \as{\PP} \semicolon
\end{align}

\item\label{i:p:ExtDomConfig2:3} for every~$\rep f$ with~$f \in \coK{\Dz}{1,\mssm}$ one has~$\ttclass[\PP]{\rep f^\trid}\in\CylQP{\PP}{\coK{\Dz}{1,\mssm}}$, and
\begin{align*}
\SF{\dUpsilon}{\PP}\tparen{\ttclass[\PP]{\rep f^\trid}}=\class[\PP]{\rep\cdc(\rep f)^\trid}
\end{align*}
is independent of the chosen $\mssm$- (i.e.,~$\mssm$-)representative~$\rep f$ of~$f$.
\end{enumerate}
\end{prop}

\subsubsection{Geometric structure}
For~$i=1,2$ let~$\pr^i\colon X^{\times 2}\rar X$ denote the projection to the~$i^\text{th}$ coordinate.
For~$\gamma,\eta\in \dUpsilon$, let~$\Cpl(\gamma,\eta)\subset \Meas(X^{\tym{2}},\A^{\hotimes 2})$ be the set of all couplings of~$\gamma$ and~$\eta$, viz.
\begin{align*}
\Cpl(\gamma,\eta)\eqdef \set{\cpl\in \Meas(X^{\tym{2}},\A^{\hotimes 2}) \colon \pr^1_\pfwd \cpl =\gamma \comma \pr^2_\pfwd \cpl=\eta} \fstop
\end{align*}
A distance function~$\mssd$ on~$X$ induces a distance on~$\dUpsilon$, in the sense of the following definition.

\begin{defs}[$L^2$-transportation distance]
The \emph{$L^2$-transportation} (\emph{extended}) \emph{distance} on~$\dUpsilon$ is
\begin{align*}
\mssd_\dUpsilon(\gamma,\eta)\eqdef \inf_{\cpl\in\Cpl(\gamma,\eta)} \paren{\int_{X^{\times 2}} \mssd^2(x,y) \diff\cpl(x,y)}^{1/2}\comma \qquad \inf{\emp}=+\infty \fstop
\end{align*}
\end{defs}

As shown by several works about configuration spaces over Riemannian manifolds, e.g.~\cite{RoeSch99,ErbHue15}, the $L^2$-transportation distance on~$\dUpsilon$ induced by the Riemannian distance on the underlying manifold is a natural object.
In particular, the same metric properties of the underlying manifold, such as completeness, hold for the corresponding multiple configuration space as well.

In fact~$\mssd_\dUpsilon$ is an extended distance, attaining the value~$+\infty$ on a set of positive~$\QP^\otym{2}$-measure in~$\dUpsilon^\tym{2}$, making~$\tparen{\dUpsilon,\mssd_\dUpsilon, \QP}$ into an extended metric measure space.

\subsection{Product spaces}
In order to address properties of the configuration space, we rely on the transfer method developed in~\cite{LzDSSuz21}, relating objects on~$\dUpsilon$ to the corresponding objects defined on a suitable subset of the infinite product~$X^\tym{\infty}$.

\subsubsection{Labeling maps and cylinder sets}
We recall here the main constructions on infinite-product spaces after~\cite{LzDSSuz21}.
For proofs of the facts stated here, as well as for some detailed heuristics, we refer the reader to~\cite[\S3.2.1]{LzDSSuz21}.
For a topological local structure~$\mcX$, let
\begin{align*}
\mbfX\eqdef \set{\emp} \sqcup \bigsqcup_{N\in \overline\N_1} X^\tym{N}
\end{align*}
be endowed with its natural (direct sum of products) $\sigma$-algebra~$\boldSigma$. For~$M\in\overline\N_1$, we define the truncation~$\tr^M\colon \mbfX\rar \mbfX$ by
\begin{align*}
\tr^M\colon \mbfx\longmapsto \mbfx^\asym{M}\eqdef& \begin{cases} \seq{x_1,\dotsc, x_M} &\text{if } \mbfx=\seq{x_p}_{p\leq N} \text{ and } M\leq N \text{ and } M<\infty\comma
\\
\mbfx &\text{if } \mbfx=\seq{x_p}_{p\leq N} \text{ and } M> N \text{ or } N=M=\infty
\end{cases} \fstop
\end{align*}
The map~$\tr^M$ is clearly both $\boldSigma/\boldSigma$- and Borel measurable for every~$M$.

\medskip

We denote by~$\mbfX_\locfin(\msE)$ the space of $\msE$-locally finite, finite or infinite sequences in~$X$, viz.
\begin{align*}
\mbfX_\locfin(\msE)\eqdef& \set{\mbfx\eqdef \seq{x_i}_{i\leq N}\in \mbfX: \#\set{i:x_i\in E}<\infty \text{~for all } E\in\msE}\comma
\end{align*}
endowed with the subspace $\sigma$-algebra~$\boldSigma_\locfin(\msE)$. Analogously to the case of~$\dUpsilon$, let us also set
\begin{align*}
\mbfX_\locfin^\asym{\infty}(\msE)\eqdef&X^\tym{\infty}\cap\mbfX_\locfin(\msE)\fstop
\end{align*}
In light of the discussion in~\cite[\S{3.2.1}]{LzDSSuz21}, this is a $\Bo{\T^\tym{\infty}}$-measurable set.

\paragraph{Labeling maps} Further set~$\Lb\colon \mbfX_\locfin(\msE)\rar \dUpsilon$,
\begin{equation*}
\begin{aligned}
\Lb\colon \mbfx\eqdef\seq{x_p}_{p\leq N} &\longmapsto \gamma \eqdef \sum_{p=1}^N\delta_{x_p} \comma\qquad \emp \longmapsto \zero\fstop
\end{aligned}
\end{equation*}
The map~$\Lb$ is $\boldSigma_\locfin(\msE)/\A_\mrmv(\msE)$-measurable, however not $\T^\tym{\infty}/\T_\mrmv$-continuous.
The correspondence $\Lb^{-1}\colon \dUpsilon\rightrightarrows \mbfX_\locfin(\msE)\subset \mbfX$ has a $\A_\mrmv(\msE)^*/\Bo{\mbfX_\locfin(\msE)}$-measur\-able selection, where~$\A_\mrmv(\msE)^*$ is the $\sigma$-algebra of all universally measurable subsets of~$\tparen{\dUpsilon,\T_\mrmv(\msE)}$.
We call any\linebreak $\A_\mrmv(\msE)^\QP/\Bo{\mbfX_\locfin(\msE)}$-measur\-able right inverse of~$\Lb$ a \emph{labeling map}~$\lb\colon \dUpsilon\rar \mbfX_\locfin(\msE)$.
Labeling maps are never continuous.

\paragraph{Cylinder sets}
We recall the definition of \emph{cylinder sets} in~$\mbfX^\asym{\infty}_\locfin(\msE)$, cf.~\cite[\S3.2.1]{LzDSSuz21}.
Let~$\mfS_0(\N_1)$ denote the group of bijections of~$\N_1$ with cofinitely many fixed points, and, for~$\mbfA\subset X^\tym{\infty}$ and~$\sigma\in\mfS_0(\N_1)$, set $\mbfA_\sigma\eqdef\set{\mbfx_\sigma:\mbfx\in \mbfA}$.

\begin{defs}[Cylinder sets]
A set~$\mbfA\subset X^\tym{\infty}$ is a \emph{cylinder set} if there exist~$n\in\N$ and sets
\begin{align*}
A_1,\dotsc, A_n\in \Bo{\T}\cap\msE\qquad \text{with} \qquad \mssm A_i>0 \comma \quad i\in [n]\comma
\end{align*}
so that
\begin{align}\label{eq:CylinderSet}
\mbfA=\tr_n^{-1}(A_1\times\cdots\times A_n)=A_1\times \cdots \times A_n\times X \times X\times\cdots
\end{align}
We denote by~$\CylSet{\msE}$ the family of all cylinder sets.
Finally, for every cylinder set~$\mbfA$, we let
\begin{align}\label{eq:CylinderRestricted}
\tilde\mbfA\eqdef \mbfA\cap\mbfX^\asym{\infty}_\locfin(\msE) \fstop
\end{align}
\end{defs}

Recall the definition~\eqref{eq:ConcentrationSetConfig2} of \emph{concentration sets}.
We have the following:

\begin{prop}[Concentration $\iff$ Cylinder, {\cite[Prop.~3.24]{LzDSSuz21}}]\label{p:FundamentalSetsConfig2}
Let~$\mbfA=\tr_n^{-1}(A_1\times \cdots \times A_n)$ be a cylinder set, and set~$j_i\eqdef \min\set{j\leq n: A_i = A_j }$, $m\eqdef\max_{i\leq n} j_i$, and\linebreak $k_i\eqdef \#\set{j\leq n: A_i=A_j}$.
Then,
\begin{align}\label{eq:p:FundamentalSetsConfig2:0}
\Lb(\tilde\mbfA)=\bigcap_{i=1}^m \Xi_{\geq k_i}^\sym{\infty} (A_{j_i}) \fstop
\end{align}

Viceversa, let~$\mbfk\eqdef\seq{k_i}_i$, and
\begin{subequations}
\begin{gather}
\label{eq:p:FundamentalSetsConfig2:1}
\Xi^\sym{\infty}_{\geq \mbfk}(A_1,\dotsc, A_m)\eqdef \set{\gamma\in\dUpsilon^\sym{\infty}:\gamma E_j\geq k_j\comma j\leq m}\comma \qquad n\eqdef \sum_{j=1}^m j_i\comma
\\
\label{eq:p:FundamentalSetsConfig2:2}
\mbfA\eqdef \tr_n^{-1}\tparen{\underbrace{A_1\times \cdots \times A_1}_{k_1} \times \cdots \times \underbrace{A_m\times \cdots \times A_m}_{k_m}} \fstop
\end{gather}
\end{subequations}
Then,
\begin{equation*}
\Lb^{-1}\tparen{\Xi^\sym{\infty}_{\geq \mbfk}(A_1,\dotsc, A_m)}= \bigcup_{\sigma\in\mfS_0(\N_1)} \tilde\mbfA_\sigma \subset \mbfX^\asym{\infty}_\locfin(\msE) \fstop
\end{equation*}
\end{prop}

\subsubsection{Finite products}
Let~$\mcX$ be a topological local structure, and let~$n\geq 2$.
We denote by~$\A^{\hotym n}$ the product $\sigma$-algebra on~$X^{\tym{n}}$, by~$\mssm^{\otym{n}}$ the product measure on~$(X^{\otym n}, \A^{\hotym n})$, by~$\msE^{\otym n}$ the localing ring generated by the algebra of pluri-rectangles generated by the family of rectangles~$\msE^\tym{n}$.
Finally, if~$(\rep\cdc, \rep\Dz)$ is a pointwise defined square-field operator (Dfn.~\ref{d:SFConfig2}), we denote by~$\rep\Dz^{\otym n}$ the $n$-fold product algebra generated by~$\rep\Dz$, endowed with the natural product operator~$\rep\cdc^\tym{n}\colon (\rep \Dz^{\otym n})^{\tym 2}\rar \mcL^\infty(\mssm^{\otym{n}})$, defined as follows.

Now, let~$\rep f^\asym{n}\colon X^\tym{n}\rar \R$ be $\A^{\hotimes{n}}$-measurable.
For~$\mbfx^\asym{n}\in X^{\tym{n}}$ and~$p\in[n]$, define the $p$-section~$\rep f^\asym{n}_{\mbfx, {p}}\colon X\rar \R$ at~$\mbfx^\asym{n}$ by
\begin{align*}
\rep f^\asym{n}_{\mbfx, {p}}\colon y \longmapsto \rep f^\asym{n}(x_1,\dotsc, x_{p-1},y,x_{p+1},\dotsc, x_n)\comma \qquad \mbfx^\asym{n}\eqdef\seq{x_1,\dotsc, x_n}\in X^{\tym{n}}\fstop
\end{align*}
Further define the \emph{product square field operator}~$(\rep\cdc^{\tym{n}},\rep\Dz^{\otym{n}})$
\begin{align*}
\rep\cdc^{\tym{n}}(\rep f^\asym{n})(\mbfx^\asym{n})\eqdef \sum_{p=1}^n \rep\cdc(\rep f^\asym{n}_{\mbfx, {p}})(x_p) \comma \qquad \rep f^\asym{n}\in \rep\Dz^{\otym{n}}\comma \qquad \mbfx^\asym{n}\in X^{\tym{n}}\fstop
\end{align*}

The following assertions are readily verified.
\begin{prop}[Product spaces]
For every~$n\geq 2$, 
\begin{itemize}
\item the quadruple~$\mcX^{\otym n}\eqdef (X^{\tym{n}}, \A^{\hotym{n}},\mssm^{\otym{n}}, \msE^{\otym n})$ is a topological local structure;
\item the pair~$(\rep\cdc^\tym{n}, \rep\Dz^{\otym n})$ is a pointwise defined square field operator;
\item the pair~$(\mcX^{\otym n}, \rep\cdc^\tym{n})$ is a~\TLDS.
\end{itemize}
\end{prop}

\subsubsection{Pre-domains on infinite products}
In order to transfer objects from the configuration space to~$\mbfX$, we shall need to make sense of the pullback map
\begin{align*}
\Lb^*\colon u\longmapsto u\circ \Lb \fstop
\end{align*}
We shall need to interpret all functions of the form~$\Lb^*u$ and~$\Lb^* \tparen{\SF{\dUpsilon}{\PP}(u)}$ as pointwise defined everywhere on~$\mbfX$, which motivated the thorough study of the choice of representatives in~\cite{LzDSSuz21}.
To this end, let us start by defining a suitable core of differentiable functions.

\begin{notat}
For a labeling map~$\lb$ and a probability measure~$\QP$ on~$\tparen{\dUpsilon,\A_\mrmv(\msE)}$ set
\begin{align*}
\QP^\asym{\infty}\eqdef \lb_\pfwd \QP \qquad \text{and} \qquad \QP^\asym{n}\eqdef (\tr^n\circ\lb)_\pfwd\QP\comma \quad n\in \N_1\fstop
\end{align*}
The labeling map~$\lb$ implicit in the notation~$\QP^\asym{N}$ will always be apparent from context. 
Further denote by~$\Bo{\T^\tym{\infty}}^{\lb_\pfwd\QP}$ the completion of~$\Bo{\T^\tym{\infty}}$ w.r.t.~$\QP^\asym{\infty}$, and define the \emph{labeling-universal} $\sigma$-algebra on~$\mbfX^\asym{\infty}_\locfin(\msE)$ by
\begin{align*}
\boldSigma^*(\msE)\eqdef \bigcap_{\lb \text{ labeling map}} \Bo{\T^\tym{\infty}}^{\lb_\pfwd\QP} \fstop
\end{align*}
\end{notat}
Note that~$\Lb\colon \mbfX^\asym{\infty}_\locfin(\msE)\to \dUpsilon$ is $\boldSigma^*(\msE)/\A_\mrmv(\msE)$-measurable.

\medskip

As a consequence of~\cite[Prop.~3.29]{LzDSSuz21}, we have the following result about $\PP^\asym{n}$-negligible sets.

\begin{lem}[Absolute continuity of projections]\label{l:AbsContPoissonProj}
Let~$\mcX$ be a topological local structure with~$\mssm X=\infty$. Then,~$\PP^\asym{n}\ll \mssm^\otym{n}$ for every~$n\in\N$.
\end{lem}

\paragraph{Predomains}
Let us now introduce a suitable space of test functions on~$X^\tym{\infty}$.
We refer the reader to~\cite[\S3.3.2]{LzDSSuz21} for further details about this construction.

\begin{notat}
For a bounded $\boldSigma^*(\msE)$-measurable function~$\rep U\colon X^\tym{\infty}\rar \R$ let
\begin{align}\label{eq:d:DiConfig2:0}
\rep U_{\mbfx, {p}}\colon z\longmapsto \rep U(x_1,\dotsc, x_{p-1},z, x_{p+1},\dotsc)\comma \qquad \mbfx\in X^\tym{\infty}\comma \quad p\in \N_1\comma
\end{align}
and set, for every~$N\in\overline\N_1$,
\begin{align}\label{eq:d:DiConfig2:1}
\begin{aligned}
\rep\cdc^p(\rep U)(\mbfx)\eqdef&\ \repSF{X}{\mssm}(\rep U_{\mbfx, p})(x_p) \comma
\\
\rep\cdc^\asym{N}(\rep U)(\mbfx)\eqdef&\ \sum_{p=1}^N \rep\cdc^p(\rep U)(\mbfx) \comma
\end{aligned}
\qquad \mbfx\in X^\tym{\infty}\comma
\end{align}
whenever this makes sense.
We denote by~$\rep\cdc^\asym{N}(\emparg,\emparg)$ the bilinear form induced by~$\rep\cdc^\asym{N}(\emparg)$ by polarization.
\end{notat}

We define the Sobolev \emph{semi}-norm
\begin{align*}
\ttnorm{\rep U}_{\llb}\eqdef \tnorm{\ttabs{\rep U}+ \rep\cdc^\asym{\infty}(\rep U)^{1/2}}_{L^2(\lb_\pfwd \PP)} \fstop
\end{align*}

\begin{defs}[pre-Sobolev class]
We say that $\rep U\colon \mbfX^{\asym{\infty}}_\locfin(\msE) \rar \R$ is \emph{pre-Sobolev} if
\begin{enumerate}[$(a)$]
\item
$\rep U$ is bounded $\boldSigma^*(\msE)$-measur\-able;
\item
there exists a constant~$M>0$ so that~$\ttnorm{\rep U}_{\llb}\leq M$ for \emph{every} labeling map~$\lb$.
\end{enumerate}
We denote by~$\preW(\msE)$ the space of all pre-Sobolev functions on~$\mbfX^{\asym{\infty}}_\locfin(\msE)$, and by~$\spclass[\llb]{\preW(\msE)}$ the corresponding space of~$\lb_\pfwd \PP$-classes for some fixed labeling map~$\lb$.
\end{defs}

\section{Essential self-adjointness and \texorpdfstring{$L^p$}{Lp}-uniqueness}
In this section we provide some explicit expressions for the generator~$\LL{\dUpsilon}{\PP}$ of the form~$\EE{\dUpsilon}{\PP}$ on different classes of cylinder functions.
Additionally, we discuss its essential self-adjointness and $L^p$-uniqueness on said classes. 

\subsection{Exponential cylinder functions and semigroup representation}\label{sss:ExpCyl}
In this section, we generalize to our setting part of~\cite[\S4 and \S7]{AlbKonRoe98}, concerned with cylinder functions of exponential type. Recall Notation~\ref{n:FormConfig2}\ref{i:n:FormConfig2:3} and~\ref{i:n:FormConfig2:4}. Closely following~\cite{AlbKonRoe98} and~\cite{KonLytRoe02}, we set
\begin{align}\label{eq:De}
\De\eqdef& \set{f\in \dom{\LL{X}{\mssm}} \cap L^1(\mssm) : \quad \begin{matrix}\LL{X}{\mssm} f\in L^1(\mssm)\cap L^\infty(\mssm) \textrm{ and }\\ -\delta\leq f \leq 0 \textrm{ for some $\delta\in (0,1)$} \end{matrix}} \fstop
\end{align}
It is readily established that~$\Span\De\subset \Daux$ is dense in $L^2(\mssm)$, and that~$\TT{X}{\mssm}_\bullet$ leaves~$\Span\De$ invariant. As a consequence,~$\tparen{\LL{X}{\mssm},\Span\De}$ is essentially self-adjoint by~\cite[Thm.~X.49]{ReeSim75}. 
Further set
\begin{align}\label{eq:ExpCyl}
\Ex{\De}\eqdef& \Span\set{\gamma\longmapsto \exp\tparen{\gamma \log(1+f)} \colon f\in\De}\fstop
\end{align}
Firstly, let us note that, for every~$\delta\in (0,1)$, we may find~$\phi\in \Czinfty(\R)$ so that~$\phi(t)=\log(1+t)$ for all~$t\in [-\delta, 0]$ and~$\phi(0)=0$.
By chain rule for~$\SF{X}{\mssm}$ we have~$\phi\circ f\in\dom{\EE{X}{\mssm}}$ for every~$f\in\De\subset \dom{\EE{X}{\mssm}}$. Thus,~$\log(1+f)^\trid$ is a well-defined element of~$\dom{\EE{\dUpsilon}{\PP}}$ for every~$f\in\De$ by Proposition~\ref{p:ExtDomConfig2}\ref{i:p:ExtDomConfig2:1} for the choice~$F=\phi$. In particular,~$\Ex{\De}\subset L^2(\PP)$. Furthermore, since~$\De$ generates the $\sigma$-algebra $\Bo{\T}$, one has that~$\Ex{\De}$ generates the $\sigma$-algebra~$\Bo{\T_\mrmv(\msE)}$ on~$\dUpsilon$, and that the inclusion~$\Ex{\De}\subset L^2(\PP)$ is a dense one.

\smallskip

For~$t>0$, we define~$\mbfP_t\colon \Ex{\De}\rar L^2(\PP)$ as the linear extension of
\begin{align}\label{eq:SemigroupP}
\mbfP_t\colon \exp\tparen{\log(1+f)^\trid}\longmapsto \exp\tparen{\log(1+\TT{X}{\mssm}_t f)^\trid}\comma \qquad f\in\De\comma \qquad t>0 \fstop
\end{align}
Since~$\TT{X}{\mssm}_\bullet$ leaves~$\De$ invariant, then~$\mbfP_\bullet$ leaves~$\Ex{\De}$ invariant as well.
The goal of this section is to show Theorem~\ref{t:AKR4.1} below, where we identify~$\mbfP_\bullet$ with~$\TT{\dUpsilon}{\PP}_\bullet$ on the set~$\Ex{\De}$.

\begin{thm}[Semigroups' representation]\label{t:AKR4.1}
Let $(\mcX,\cdc)$ be a \TLDS.
Then,
\begin{align*}
\TT{\dUpsilon}{\PP}_t u = \mbfP_t u\comma \qquad u\in \Ex{\De}\comma \qquad t> 0 \fstop
\end{align*}
\end{thm}

\begin{rem}[Comparison with~{\cite{AlbKonRoe98}}] Let us assume that~$\mcX$ is a smooth manifold, and that~$(\mcX,\cdc)$ satisfies all the assumptions in~\cite[\S2.2, p.~452]{AlbKonRoe98}. Then, our Lemma~\ref{l:AKR7.3} and Theorem~\ref{t:AKR4.1} generalize the corresponding statements in~\cite[Lem.~7.3]{AlbKonRoe98} and~\cite[Prop.~4.1]{AlbKonRoe98}, in that we assume neither the essential self-adjointness of the generator~$\LL{X}{\mssm}$ on a suitable core, nor that the semigroup~$\TT{X}{\mssm}_\bullet$ is conservative.
Indeed, in~\cite[Lem.~7.3]{AlbKonRoe98} the authors compute the generator~$\LL{\dUpsilon}{\PP}$ on functions in~$\Ex{\De\cap \Czinfty(X)}$ and subsequently proceed to establish~\eqref{eq:l:AKR7.3:1} on the whole of~$\Ex{\De}$ by approximation.
The chosen approximation technique requires that~$\TT{X}{\mssm}_\bullet$ be conservative.
The essential self-adjointness of~$\tparen{\LL{X}{\mssm},\Czinfty(X)}$ is required in order to show well-posedness of~$\LL{\dUpsilon}{\PP}$, in the sense of independence from the approximating sequence.
On the contrary, here, we directly compute the generator~$\LL{\dUpsilon}{\PP}$ on functions in~$\Ex{\De}$. This generalization is essentially a consequence of Proposition~\ref{p:ExtDomConfig2}, establishing well-posedness of~$\LL{\dUpsilon}{\PP}$ directly on~$\Ex{\De}$.
It allows to address base spaces previously out of reach, including e.g.\ (the configuration spaces over) path/loop spaces or~$\RCD$ spaces, on which essential self-adjointness on spaces of exponential cylinder functions was previously not known.
\end{rem}

Let us now collect all the necessary results for a proof of Theorem~\ref{t:AKR4.1}.

\subsubsection{Explicit form of the generator}
It is not difficult to compute the generator\linebreak $\tparen{\LL{\dUpsilon}{\PP},\dom{\LL{\dUpsilon}{\PP}}}$ of the form $\tparen{\EE{\dUpsilon}{\PP},\dom{\EE{\dUpsilon}{\PP}}}$, by means of the Mecke identity~\eqref{eq:MeckeConfig2}.
Let~$(\mcX,\cdc)$ be a \TLDS, and~$\ell\colon L^\infty(\mssm)\rar\mcL^\infty(\mssm)$ be a strong lifting.
Recall Notation~\ref{n:FormConfig2}\ref{i:n:FormConfig2:3}, and set
\begin{align*}
\Daux\eqdef \set{f\in \dom{\LL{X}{\mssm}}\cap L^1(\mssm)\cap L^\infty(\mssm) : \LL{X}{\mssm}f\in L^\infty(\mssm)}\fstop 
\end{align*}
On the space of representatives~$\rep\Daux\eqdef \ell(\Daux)$ we set~$\repLL{X}{\mssm}_\ell\eqdef \ell\circ \LL{X}{\mssm}$.
For~$\rep v=G\circ \rep\mbfg^\trid\in \Cyl{\rep\Daux}$ of the form~\eqref{eq:d:LiftCdCConfig2} and fixed~$\gamma\in\dUpsilon$ set further
\begin{align}\label{eq:p:Generator:00}
\repLL{\gamma}{\mssm}_\ell \rep v\colon x\longmapsto \sum_{j=1}^m (\partial_j G)\tparen{\rep\mbfg^\trid\gamma} \cdot \tparen{\repLL{X}{\mssm}_\ell g_j}(x) +\sum_{p,q=1}^{m,m} (\partial^2_{pq} G)\tparen{\rep\mbfg^\trid \gamma} \cdot \repSF{X}{\mssm}_\ell\tparen{g_p,g_q}(x)\fstop
\end{align}
Since~$g_j\in\dom{\LL{X}{\mssm}}\cap L^1(\mssm)\subset \co_{1,\mssm}(\Dz)$ for every~$j\leq m$, the function~$\gamma\mapsto \rep g_j^\trid \gamma$ is independent of the choice of the $\mssm$-representative~$\rep g_j$ of~$g_j$ and finite $\PP$-a.e.\ by Proposition~\ref{p:ExtDomConfig2}\ref{i:p:ExtDomConfig2:3}. As a consequence,~$\repLL{\gamma}{\mssm}_\ell \rep v$ is well-defined and finite for every~$x\in X$ for $\PP$-a.e.~$\gamma\in\dUpsilon$.

\begin{prop}[Generators]\label{p:Generator}
One has that
\begin{enumerate}[$(i)$]
\item\label{i:p:Generator:1} $v\eqdef \class[\PP]{G\circ \rep\mbfg^\trid}\in \CylQP{\PP}{\Daux}$ is well-defined in~$\dom{\LL{\dUpsilon}{\PP}}$;
\item\label{i:p:Generator:2} for~$\rep v=F\circ\rep\mbff^\trid\in \Cyl{\rep\Daux}$, it holds that
\begin{align}\label{eq:p:Generator:0}
\tparen{\LL{\dUpsilon}{\PP} v}(\gamma)= \tparen{\repLL{\gamma}{\mssm}_\ell \rep v}^\trid\gamma \qquad \forallae{\PP} \gamma\in\dUpsilon \fstop
\end{align}
As a consequence, $\tparen{\repLL{\gamma}{\mssm}_\ell \rep v}^\trid\gamma$ is well-defined on $\PP$-classes and independent of~$\ell$;

\item\label{i:p:Generator:3} $\tparen{\LL{\dUpsilon}{\PP},\dom{\LL{\dUpsilon}{\PP}}}$ is the Friedrichs extension of~$\tparen{\LL{\dUpsilon}{\PP}, \CylQP{\PP}{\Daux}}$ in~\eqref{eq:p:Generator:0}.
\end{enumerate}

\begin{proof} \ref{i:p:Generator:1} is a consequence of~\ref{i:p:Generator:2} and of the fact that~$\gamma\mapsto v(\gamma)$ is well-defined in~$L^2(\PP)$. By the standard theory of Dirichlet forms, it thus suffices to verify that
\begin{align}\label{eq:p:Generator:1}
\EE{\dUpsilon}{\PP}(u,v)=\tscalar{u}{-\tparen{\repLL{\gamma}{\mssm}_\ell \rep v}^\trid}_{L^2(\PP)}\comma \qquad u,v\in \CylQP{\PP}{\Daux} \fstop
\end{align}
Indeed, the fact that the operator~$\LL{\dUpsilon}{\PP}$ defined in~\eqref{eq:p:Generator:0} is independent of the chosen representatives of~$u$,~$v$ and of the strong lifting~$\ell$ will in turn be a consequence of the above representation~\eqref{eq:p:Generator:1}, since the left-hand side enjoys the same property by Proposition~\ref{p:ExtDomConfig2}\ref{i:p:ExtDomConfig2:2}. \ref{i:p:Generator:3} is a standard consequence of~\eqref{eq:p:Generator:1}, cf.~\cite[Thm.~X.23]{ReeSim75}.

Let us now prove~\eqref{eq:p:Generator:1}. By definition of~$\tparen{\LL{X}{\mssm},\dom{\LL{X}{\mssm}}}$,
\begin{align}\label{eq:TripleIbP}
\int_X \SF{X}{\mssm}(f,g) \, \diff\mssm= -\int_X f \, \LL{X}{\mssm} g \, \diff\mssm \comma \qquad f,g\in\dom{\LL{X}{\mssm}}\fstop
\end{align}

Let further~$\rep u,\rep v\in\Cyl{\rep\Daux}$ be of the form~\eqref{eq:d:LiftCdCConfig2}, and note that~$\mapsto \rep u(\gamma+\delta_x)$, and~$x\mapsto \rep v(\gamma+\delta_x)$ are both ($\mssm$-representatives of) elements of~$\domext{\SF{X}{\mssm}}$ for $\PP$-a.e.~$\gamma\in\dUpsilon$, and
\begin{align}
\nonumber
\SF{X}{\mssm}&\tparen{\rep u(\gamma+\delta_{\emparg}), \rep v(\gamma+\delta_\emparg)}=
\\
\nonumber
=& \sum_{i,j}^{k,m} (\partial_i F)\tparen{\rep\mbff^\trid (\gamma+\delta_{\emparg})}\cdot (\partial_j G)\tparen{\rep\mbfg^\trid(\gamma+\delta_{\emparg})} \cdot 
\\
\nonumber
&\phantom{\sum}\cdot \SF{X}{\mssm}\tparen{\rep f_i^\trid(\gamma+\delta_\emparg),\rep g_j^\trid(\gamma+\delta_\emparg)}
\\
\label{eq:p:Generator:2}
=&\sum_{i,j}^{k,m} (\partial_i F)\tparen{\rep\mbff^\trid (\gamma+\delta_{\emparg})}\cdot (\partial_j G)\tparen{\rep\mbfg^\trid(\gamma+\delta_{\emparg})} \cdot \SF{X}{\mssm}\tparen{f_i,g_j} \qquad \as{\mssm}\fstop
\end{align}

Now, by Proposition~\ref{p:ExtDomConfig2}\ref{i:p:ExtDomConfig2:2},
\begin{align*}
\EE{\dUpsilon}{\PP}(u,v)\overset{\phantom{\eqref{eq:MeckeConfig2}}}{=}&\iint_{X\times \dUpsilon}\, \sum_{i,j}^{k,m} (\partial_i F)(\rep\mbff^\trid \gamma) \cdot (\partial_j G)(\rep\mbfg^\trid \gamma) \cdot \repSF{X}{\mssm}_\ell(f_i,g_j) \, \diff\gamma \, \diff\PP(\gamma)
\\
\overset{\eqref{eq:MeckeConfig2}}{=}&\iint_{X\times \dUpsilon}\, \sum_{i,j}^{k,m} (\partial_i F)\tparen{\rep\mbff^\trid (\gamma+\delta_\emparg)} \cdot (\partial_j G)\tparen{\rep\mbfg^\trid (\gamma+\delta_\emparg)} \cdot \SF{X}{\mssm}(f_i,g_j) \diff\mssm\, \diff\PP(\gamma)
\\
\overset{\eqref{eq:p:Generator:2}}=&\iint_{X\times\dUpsilon}\, \SF{X}{\mssm}\tparen{\rep u(\gamma+\delta_\emparg),\rep v(\gamma+\delta_\emparg)} \cdot \diff\mssm\, \diff\PP(\gamma)
\\
\overset{\eqref{eq:TripleIbP}}=&-\iint_{X\times\dUpsilon}\, \rep u(\gamma+\delta_\emparg)\cdot \LL{X}{\mssm}\tparen{\rep v(\gamma+\delta_\emparg)} \diff\mssm\, \diff\PP(\gamma) 
\\
\overset{\phantom{\eqref{eq:MeckeConfig2}}}{=}&-\iint_{X\times\dUpsilon}\, \rep u(\gamma+\delta_\emparg) \Bigg[ \sum_{j=1}^m (\partial_j G)\tparen{\rep \mbfg^\trid (\gamma+\delta_\emparg)} \cdot \LL{X}{\mssm}\tparen{\rep g_j^\trid(\gamma+\delta_\emparg)} 
\\
&\qquad+\sum_{p,q=1}^{m,m} (\partial^2_{pq} G)\tparen{\rep\mbfg^\trid (\gamma+\delta_\emparg)} \cdot \SF{X}{\mssm}\tparen{\rep g_p^\trid(\gamma+\delta_\emparg),\rep g_q^\trid(\gamma+\delta_\emparg)} \Bigg]
\diff\mssm\, \diff\PP(\gamma)
\\
\overset{\phantom{\eqref{eq:MeckeConfig2}}}{=}&-\iint_{X\times\dUpsilon}\, \rep u(\gamma+\delta_x) \Bigg[\sum_{j=1}^m (\partial_j G)\tparen{\rep\mbfg^\trid (\gamma+\delta_x)} \cdot \tparen{\LL{X}{\mssm} g_j}(x)
\\
&\qquad+\sum_{p,q=1}^{m,m} (\partial^2_{pq} G)\tparen{\rep\mbfg^\trid (\gamma+\delta_x)} \cdot \SF{X}{\mssm}\tparen{g_p,g_q}(x)\Bigg] \diff\mssm(x)\, \diff\PP(\gamma) 
\\
\overset{\phantom{\eqref{eq:MeckeConfig2}}}{=}&-\iint_{X\times\dUpsilon}\, \rep u(\gamma+\delta_x) \Bigg[\sum_{j=1}^m (\partial_j G)\tparen{\rep\mbfg^\trid (\gamma+\delta_x)} \cdot \tparen{\repLL{X}{\mssm}_\ell g_j}(x)
\\
&\qquad+\sum_{p,q=1}^{m,m} (\partial^2_{pq} G)\tparen{\rep\mbfg^\trid (\gamma+\delta_x)} \cdot \repSF{X}{\mssm}_\ell\tparen{g_p,g_q}(x)\Bigg] \diff\mssm(x)\, \diff\PP(\gamma)
\\
\overset{\eqref{eq:MeckeConfig2}}=&-\iint_{X\times\dUpsilon}\, u(\gamma) \Bigg[\sum_{j=1}^m (\partial_j G)(\mbfg^\trid\gamma) \cdot \tparen{\repLL{X}{\mssm}_\ell g_j}(x) 
\\
&\qquad+\sum_{p,q=1}^{m,m} (\partial^2_{pq} G)(\mbfg^\trid \gamma) \cdot \repSF{X}{\mssm}_\ell\tparen{g_p,g_q}(x) \Bigg] \diff\gamma(x)\, \diff\PP(\gamma)\comma
\end{align*}
which concludes the proof.
\end{proof}
\end{prop}

The next lemma extends~{\cite[Lem.~7.3]{AlbKonRoe98}}.

\begin{lem}\label{l:AKR7.3} Let~$(\mcX,\cdc)$ be a \TLDS. Then,
\begin{align}
\nonumber
\exp\tparen{\log(1+f)^\trid}\in& \dom{\LL{\dUpsilon}{\PP}} \comma \qquad f\in \De\comma
\intertext{and}
\label{eq:l:AKR7.3:1}
\LL{\dUpsilon}{\PP} \exp\tparen{\log(1+f)^\trid}=& \paren{\frac{\LL{X}{\mssm} f}{1+f}}^\trid \cdot\exp\tparen{\log(1+f)^\trid}\comma \qquad f\in\De \fstop
\end{align}
\begin{proof}
We start by computing the generator~$\LL{\dUpsilon}{\PP}$ on functions of the form
\begin{align*}
u\colon\gamma\longmapsto e^{g^\trid \gamma}\comma \qquad g\in \Daux\fstop
\end{align*}
Let~$\phi_n(t)\eqdef \sum_{i=0}^n t^i/i!$ and set~$u_n\eqdef \phi_n\circ g^\trid$, $g\in\Daux$. Choosing~$u=(g^\trid)^i$ in Proposition~\ref{p:Generator},
\begin{align}\label{eq:l:AKR7.3:2}
v_n\eqdef \LL{\dUpsilon}{\PP}(\phi_n\circ g^\trid)=\tparen{\LL{X}{\mssm}g}^\trid \cdot (\phi_{n-1}\circ g^\trid) +\tparen{\SF{X}{\mssm}(g)}^\trid \cdot (\phi_{n-2}\circ g^\trid) \fstop
\end{align}
Since~$g\in\Daux$, we have that~$w\eqdef \tparen{\ttabs{\LL{X}{\mssm}g}+\SF{X}{\mssm}(g)}^\trid\cdot e^{g^\trid}$ is in~$L^2(\PP)$. Furthermore, since~$v_n\leq w$ $\PP$-a.e., we may apply Dominated Convergence with dominating function~$w$ to~\eqref{eq:l:AKR7.3:2}, to obtain
\begin{align*}
L^2(\PP)\text{-}\nlim v_n=\tparen{\LL{X}{\mssm} g+\SF{X}{\mssm}(g)}^\trid \gamma \cdot e^{g^\trid \gamma} \defeq v\fstop
\end{align*}

Finally, since~$\tparen{\LL{\dUpsilon}{\PP},\dom{\LL{\dUpsilon}{\PP}}}$ is a closed operator, then~$\LL{\dUpsilon}{\PP} u=v$, viz.\ (cf.\ \cite[Eqn. (4.9)]{AlbKonRoe98})
\begin{align*}
\tparen{\LL{\dUpsilon}{\PP} \, e^{g^\trid}}(\gamma)= \tparen{\LL{X}{\mssm} g+\SF{X}{\mssm}(g)}^\trid \gamma \cdot e^{g^\trid \gamma} \quad \forallae{\PP} \gamma\in\dUpsilon\comma \qquad g\in\Daux \fstop
\end{align*}
For~$f\in\De$ we have~$g\eqdef \log(1+f)\in \Daux$, and~\eqref{eq:l:AKR7.3:1} follows by the chain rule for~$\LL{X}{\mssm}$.
\end{proof}
\end{lem}

\begin{proof}[Proof of Theorem~\ref{t:AKR4.1}]
The proofs of~\cite[Lem.s~7.1 and~7.2]{AlbKonRoe98} carry over verbatim to our setting. The rest of the proof also carries over, having care to substitute~\cite[Lem.~7.3]{AlbKonRoe98} with our Lemma~\ref{l:AKR7.3}.
\end{proof}

\subsection{Essential self-adjointness and \texorpdfstring{$L^p$}{Lp}-uniqueness}\label{sss:ESA}
Let us recall the definition of $L^p$-uniqueness.
We refer the reader to the monograph~\cite{Ebe99} for a complete treatment.

\begin{defs}[$L^p$-uniqueness, e.g.~{\cite[Dfn.~1.1.3]{Ebe99}}]
A densely defined linear operator~$(\LL{X}{\mssm}_p,\Dz)$ on~$L^p(\mssm)$ is \emph{$L^p$-unique} (also: \emph{strongly unique}), if there exists at most one strongly continuous semigroup~$\TT{X}{\mssm}_{p,\bullet}$ the generator of which extends~$(\LL{X}{\mssm}_p,\Dz)$.
\end{defs}

As one further application of the previous section, we show that the essential self-adjointness of the generator~$\LL{\dUpsilon}{\PP}$ (on a suitable core) is inherited from the essential self-adjointness of~$\tparen{\LL{X}{\mssm},\Dz}$ on the base space.
Furthermore, we show the $L^p$-uniqueness of~$\LL{\dUpsilon}{\PP}$.

\begin{cor}[Essential self-adjointness, $L^p$-uniqueness]\label{c:ESA}
Let~$(\mcX,\cdc)$ be a \TLDS, and~$p\in [1,\infty)$.
\begin{enumerate}[$(i)$]
\item\label{i:c:ESA:1} $\tparen{\LL{\dUpsilon}{\PP},\Ex{\De}}$ is essentially self-adjoint on $L^2(\PP)$ and $L^p(\PP)$-unique;

\item\label{i:c:ESA:2} $\tparen{\LL{\dUpsilon}{\PP}, \CylQP{\PP}{\rep\Daux}}$ is essentially self-adjoint on $L^2(\PP)$ and $L^p(\PP)$-unique;

\item\label{i:c:ESA:3} if~$\tparen{\LL{X}{\mssm}, \Dz\cap\Daux}$ is essentially self-adjoint on~$L^2(\mssm)$, resp.\ $L^p(\mssm)$-unique, then the operator $\tparen{\LL{\dUpsilon}{\PP}, \Cyl{\Dz}}$ is essentially self-adjoint on $L^2(\PP)$, resp.\ $L^p(\PP)$-unique.
\end{enumerate}
\end{cor}

\begin{rem}[On essential self-adjointness]
Let us note in passing that the essential self-adjointness of~$\LL{X}{\mssm}$ on some core~$\mcC$ for~$\LL{X}{\mssm}$ implies that of~$\LL{\dUpsilon}{\PP}$ on the space~$\mcF^\circ_\fin(\mcC)$ of \emph{symmetric} $\otimes$-polynomials of~$\mcC$.
This follows from two facts:
\begin{enumerate*}[$(a)$]
\item standard arguments on second-quantization operators, e.g.~\cite[\S{II.6.1.1}, p.~185]{BerKon95};
\item the Hilbert-space isomorphism~$\mcI$ between the Bosonic Fock space of~$L^2(\mssm)$ and~$L^2(\PP_\mssm)$ provided by multiple stochastic integration, e.g.~\cite{Sur84}.
\end{enumerate*}
Whereas independent of any topological structure, this result is not particularly helpful for our analysis in the following sections, in that the image of~$\mcF^\odot_\fin(\mcC)$ via~$\mcI^{-1}$ is \emph{not} well-adapted to computations of the semigroup~$\TT{\dUpsilon}{\PP}_\bullet$.
\end{rem}

\begin{proof}[Proof of Corollary~\ref{c:ESA}]
Let~$\msA$ be either~$\Ex{\De}$,~$\CylQP{\PP}{\rep\Daux}$, or~$\CylQP{\PP}{\rep\Dz}$.
It is readily seen that~$\msA\subset L^p(\PP)$ is dense in~$L^p(\PP)$ and that~$\LL{\dUpsilon}{\PP}=\LL{\dUpsilon}{\PP}_p$ on~$\msA$ for every~$p\in [1,\infty)$.
Furthermore, the semigroup~$\tparen{\mbfP_t,\Ex{\De}}$ in~\eqref{eq:SemigroupP} is contractive in~$L^2(\PP)$, since it coincides with~$\TT{\dUpsilon}{\PP}_t$ by Theorem~\ref{t:AKR4.1}.
As a consequence, it is contractive in~$L^p(\PP)$ by Markovianity and the Riesz--Thorin Interpolation Theorem (see e.g.~\cite[\S2, p.~70]{Shi97}).
Thus, it extends to a strongly continuous contraction semigroup on~$L^p(\PP)$, denoted by~$\TT{\dUpsilon}{\PP}_{p,t}$, for every~$p\in [1,\infty)$.

Now, \ref{i:c:ESA:1} is a consequence of Theorem~\ref{t:AKR4.1} and~\cite[Thm.~X.49]{ReeSim75}.
\ref{i:c:ESA:2} By the proof of Lemma~\ref{l:AKR7.3}, $\Ex{\De}$~is contained in the domain of each self-adjoint extension of $\tparen{\LL{\dUpsilon}{\PP}, \CylQP{\PP}{\Daux}}$, and the conclusion follows from~\ref{i:c:ESA:1}.
\ref{i:c:ESA:3} Applying~\eqref{eq:p:Generator:00} to~$\rep f^\trid$ and combining it with~\eqref{eq:p:Generator:0}, we have that~$\LL{\dUpsilon}{\PP}\tparen{\tclass[\PP]{\rep f^\trid}}=\tclass[\PP]{\tparen{\LL{X}{\mssm}_\ell f}^\trid}$ for every~$f\in\Daux$.
By essential self-adjointness, resp.\ $L^p$-uniqueness, of~$\tparen{\LL{X}{\mssm}, \Dz}$ on~$\Daux$, the latter equality extends to
\begin{equation}\label{eq:c:ESA:1}
\LL{\dUpsilon}{\PP}\tparen{\tclass[\PP]{\rep f^\trid}}=\class[\PP]{\tparen{\LL{X}{\mssm}_\ell f}^\trid} \comma \qquad f\in \dom{\LL{X}{\mssm}_p}\fstop
\end{equation}
In light of~\eqref{eq:p:Generator:0}, Equation~\eqref{eq:c:ESA:1} shows that every~$u=F\circ\mbff^\trid\in\CylQP{\PP}{\rep\Daux}$ can be approximated in the graph-generator norm of~$\tparen{\LL{\dUpsilon}{\PP}_p,\dom{\LL{\dUpsilon}{\PP}_p}}$ by a sequence of functions in~$\Cyl{\Dz\cap\Daux}$ with same outer function~$F$ as~$u$ and inner functions~$f_{n,1},\dotsc, f_{n,k}\in \Dz\cap \Daux$.
Since $\tparen{\LL{\dUpsilon}{\PP}, \CylQP{\PP}{\rep\Daux}}$ is essentially self-adjoint, resp.\ $L^p$-unique, by~\ref{i:c:ESA:2}, a further approximation argument concludes the proof.
\end{proof}

\section{Identification of semigroups and Bakry--\'Emery curvature condition}
In this section we identify the heat kernel measure~$\hh{\dUpsilon}{\PP}_\bullet\eqdef \tseq{\hh{\dUpsilon}{\PP}_t}_{t\geq 0}$ of the form~$\EE{\dUpsilon}{\PP}$.
Throughout this section we assume that the heat kernel measure~$\hh{X}{\mssm}_\bullet\eqdef \tseq{\hh{X}{\mssm}_t}_{t\geq 0}$ of the form~$\EE{X}{\mssm}$ on the base \TLDS~$(\mcX,\cdc)$ satisfies some mild Feller-like properties; for the precise definitions see~\SCF below.
Under this assumption,~$\hh{\dUpsilon}{\PP}_\bullet$ ---~restricted to~$\dUpsilon^\sym{\infty}$~--- inherits stochastic completeness from the base space, and is identifiable with a corresponding heat kernel~$\mssh^\tym{\infty}_\bullet$ on the infinite-product space~$X^\tym{\infty}$.

Profiting the identification of~$\hh{\dUpsilon}{\PP}_\bullet$ with~$\mssh^\tym{\infty}_\bullet$, we show, as a first application, that a Bakry--\'Emery Ricci-curvature lower bound holds for the configuration space~$\tparen{\dUpsilon^\sym{\infty},\SF{\dUpsilon}{\PP}}$ if (and only if) it holds for the base~$(\mcX,\cdc)$.
We postpone a thorough comparison of our results here to the analogous results in~\cite{ErbHue15} and~\cite{KonLytRoe02} to Remarks~\ref{r:ComparisonErbHue} and~\ref{r:ComparisonKonLytRoe} respectively.

\medskip

In the following, we shall need to test heat kernel densities of the form~$\hh{X}{\mssm}_t(\emparg, A)$ against configurations, in the sense of~\eqref{eq:TridConfig2}.
We shall assume the following properties.

\begin{defs}[Stochastic properties of~$\TLDS$'s]\label{d:wFe}
We say that a~\TLDS~$(\mcX,\cdc)$ satisfies \SCF if all of the following conditions hold:
\begin{itemize}
\item \emph{$\mcL^\infty(\mssm)$-to-$\Cb(\T)$-Feller property}
\begin{equation}\tag*{$(\mathsf{F})_{\ref{d:wFe}}$}\label{ass:wFe} 
\hh{X}{\mssm}_t(\emparg, A)\in \Cb(\T)\comma \qquad A\in\A\comma \qquad t>0 \fstop
\end{equation}

\item \emph{$\Sb(\msE)$-to-$\Ccompl(\msE)$-Feller property}
\begin{equation}\tag*{$(\mathsf{F}_\msE)_{\ref{d:wFe}}$}\label{ass:Fe} 
\hh{X}{\mssm}_t(\emparg, E)\in \Ccompl(\msE)\comma \qquad E\in\msE\cap \A\comma \qquad t>0 \fstop
\end{equation}

\item \emph{stochastic completeness}
\begin{equation}\tag*{$(\mathsf{SC})_{\ref{d:wFe}}$}\label{ass:SC} 
\hh{X}{\mssm}_t(x,X)=1\comma \qquad x\in X\comma \quad t>0\fstop
\end{equation}
\end{itemize}
\end{defs}

We turn to the Bakry--\'Emery curvature condition for the configuration space over a \TLDS~$(\mcX,\cdc)$.
\begin{defs}[Bakry--\'Emery gradient estimate]\label{d:BE}
A \TLDS~$(\mcX,\cdc)$ satisfies the \emph{weak Bakry--\'Emery gradient estimate with constants~$c\geq 1$ and~$K\in \R$}, in short:~$\wBE_c(K,\infty)$, if
\begin{align}\tag*{$\wBE_c(K,\infty)_{\ref{d:BE}}$}\label{ass:BE}
\SF{X}{\mssm}\tparen{\TT{X}{\mssm}_t f} \leq c\, e^{-2Kt}\, \TT{X}{\mssm}_t\,\SF{X}{\mssm}(f) \comma \qquad t>0\comma \quad f\in \dom{\EE{X}{\mssm}}\fstop
\end{align}
We write~$\BE(K,\infty)$ for the standard \emph{Bakry--\'Emery gradient estimate}~$\wBE_1(K,\infty)$.
\end{defs}

Our main goal in this section is to show the following results.

\begin{thm}[Bakry--\'Emery]\label{t:BE}
Let~$(\mcX,\cdc)$ be a \TLDS satisfying~\SCF, and fix~$c\geq 1$ and~$K\in\R$. Then, the following are equivalent:
\begin{enumerate}[$(a)$]
\item\label{i:t:BE:1} $(\mcX,\cdc)$ satisfies~\ref{ass:BE},
\item\label{i:t:BE:2} $\tparen{\dUpsilon,\SF{\dUpsilon}{\PP}}$ satisfies~\ref{ass:BE}.
\end{enumerate}
\end{thm}

As a corollary of the Bakry--\'Emery gradient estimate for~$\dUpsilon$, we have the following weak Bochner inequality.
Let~$(\mcE,\mcF)$ be a Dirichlet form with square field operator~$(\cdc,\dom{\cdc})$ and generator~$(\mcL,\dom{\mcL})$, and recall that its iterated square-field operator~$(\cdc_2,\dom{\cdc_2})$ is
\begin{align*}
\dom{\cdc_2}\eqdef \set{f\in\dom{\mcL}: \mcL f\in\dom{\cdc}} \comma \qquad \cdc_2(f,g)\eqdef \tfrac{1}{2}\tbraket{\mcL \cdc(f,g) - \cdc(f,\mcL g)- \cdc(\mcL f,g)}\fstop
\end{align*}
As a consequence of~\cite[Cor.~2.3(vi)$\implies$(i)]{AmbGigSav15}, we also obtain:

\begin{cor}[Weak $L^2$-Bochner inequality]
Le~$(\mcX,\cdc)$ be a \TLDS satisfying~\SCF and~$\BE(K,\infty)$ for some~$K\in\R$.
Then, the following weak $L^2$-Bochner inequality holds for $\ttparen{\dUpsilon,\SF{\dUpsilon}{\PP}}$,
\begin{align*}
\int_\dUpsilon \SF{\dUpsilon}{\PP}_2(u)\, v\diff\PP\geq K \int_\dUpsilon \SF{\dUpsilon}{\PP}(u)\, v\diff\PP\comma \qquad u,v\in \dom{\SF{\dUpsilon}{\PP}_2} \fstop
\end{align*}
\end{cor}

\begin{rem}[Comparison with~\cite{ErbHue15}]\label{r:ComparisonErbHue}
For~$c=1$, i.e.\ for the standard Bakry--\'Emery gradient estimate, Theorem~\ref{t:BE} is stated in~\cite{ErbHue15} in the case when~$X$ is a Riemannian manifold.
In contrast with the proof in~\cite{ErbHue15}, our pur proof of Theorem~\ref{t:BE} does not require the base space to be a metric space ---~nor, in the least, a Riemannian manifold.
In particular, the statement presented here
\begin{itemize}
\item only relies on the \TLDS-space structure of the base space and on~\SCF;
\item is independent of other properties which may or may not be inherited by the Poisson configuration space, as for instance the Rademacher property;
\item disentangles the necessary assumptions from the geometric superstructure in~\cite{ErbHue15};
\item settles the validity of Bakry--\'Emery curvature bounds also in the case of weighted manifolds, which was conjectured in~\cite[Rmk.~1.5]{ErbHue15}.
\end{itemize}

Furthermore, whereas the main ideas behind our proof is already found in~\cite{ErbHue15}, the proofs of several statements in~\cite{ErbHue15} rely on the understanding that cylinder functions of the form~$\Cyl{\mcC^\infty_c(X)}$ are $\mssd_\dUpsilon$-Lipschitz.
This understanding is however \emph{erroneous}, as shown by the explicit counterexample~\cite[Ex.~\ref{e:NonLipCyl}]{LzDSSuz21}.
This affects in particular the proof of Proposition~2.3 (claiming the identification of~$\EE{\dUpsilon}{\PP}$ with~$\Ch[\mssd_{\dUpsilon,\PP}]$), the proof of Theorem~4.7, and the line of reasoning of the entire \S5 in~\cite{ErbHue15}, which relies on Proposition~2.3 in an essential way.
\end{rem}

\subsection{Identifications of heat kernels}
In this section we establish explicit formulas for the heat kernel measure~$\hh{\dUpsilon}{\PP}_\bullet$ representing the Markov semigroup~$\TT{\dUpsilon}{\PP}_\bullet$ associated to the Dirichlet form $\tparen{\EE{\dUpsilon}{\PP},\dom{\EE{\dUpsilon}{\PP}}}$, in terms of the heat kernel measure~$\hh{X}{\mssm}_\bullet$ on the base space.

\medskip

Before indulging into details, let us give here a short summary of the strategy, based upon ideas from~\cite{AlbKonRoe98, KonLytRoe02, ErbHue15}, where the same representation of the heat kernel is shown when~$X$ is a smooth Riemannian manifold with Ricci curvature bounded from below. We aim at establishing the identification
\begin{align}\label{eq:IdentificationHeatPrelim}
\hh{\dUpsilon}{\PP}_t(\gamma,\diff\eta) = \bigotimes_{i=1}^\infty \hh{X}{\mssm}_t(x_i, \diff y_i)\comma \qquad \gamma=\Lb(\mbfx)\comma \eta=\Lb(\mbfy) \fstop
\end{align}
Firstly, we note that such identification may hold only under the assumption that~$\mssm X=\infty$.
Even in this case, one should not expect~\eqref{eq:IdentificationHeatPrelim} to hold on the whole of~$\dUpsilon$. 
Rather, one hopes to establish~\eqref{eq:IdentificationHeatPrelim} on some subset~$\Theta\subset \dUpsilon$ of full $\PP$-measure.

\subsubsection{Heat kernels: definitions}
Assumption~\ref{ass:wFe} grants the existence of a continuous $\mssm$-represen\-ta\-tive of~$\hh{X}{\mssm}_t(\emparg, A)$.
Since $\hh{X}{\mssm}_t(\emparg, A)\geq 0$, we can then make sense of the functionals~$\gamma\mapsto \hh{X}{\mssm}_t(\emparg, A)^\trid \gamma$ in light of Notation~\ref{n:ClassesConfig2}.

\begin{notat}
On~$(X^\tym{\infty},\A^{\hotym\infty})$ we define the heat kernel measure
\begin{align*}
\mssh^\tym{\infty}_t(\mbfx, \diff\mbfy)\eqdef& \bigotimes_{i=1}^\infty \hh{X}{\mssm}_t(x_i,\diff y_i)\comma  \qquad \mbfx\in X^{\times\infty}\comma \qquad t> 0 \fstop
\end{align*}
\end{notat}

\begin{rem}
Note that, for fixed~$\mbfA\in\Bo{\tym\infty}$, the continuity of~$\mssh^\tym{\infty}_\emparg(\emparg, \mbfA)\colon [0,\infty)\times X^\tym{\infty}\to \R$ should not be expected, even in the case when~$(\mcX,\cdc)$ is a smooth compact Riemannian manifold, cf.\ e.g.~\cite[Thm.~1.1(3)]{BenSaC97}.
\end{rem}

\begin{lem}[Stochastic completeness I]\label{l:SCProduct}
Let~$(\mcX,\cdc)$ be a \TLDS satisfying~\SCF. Then,
\begin{align}\label{eq:l:SCConfig1:1}
\mssh^{\tym\infty}_t(\mbfx,\mbfX^\asym{\infty}_\locfin(\msE))=&\ 1 \comma \qquad \mbfx\in\mbfX^\asym{\infty}_\locfin(\msE)\comma\qquad t\geq 0\comma
\end{align}

\begin{proof}
Since~$(\mcX,\cdc)$ satisfies~\ref{ass:SC}, it is clear that~$\mssh^\tym{\infty}_t(\mbfx,X^\tym{\infty})=1$ for every~$\mbfx\in\mbfX^\asym{\infty}_\locfin(\msE)$, i.e.\ ---~informally~--- that cofinitely many particles of~$\Lb(\mbfx)$ may \emph{not} dissipate in finite time.
In order to show~\eqref{eq:l:SCConfig1:1} it thus remains to prove that ---~informally~--- infinitely many particles may not accumulate in a bounded set~$E\in\msE$, viz.~$\mssh^\tym{\infty}_t(\mbfx, E^\tym{\infty})=0$ for every~$\mbfx\in \mbfX^\asym{\infty}_\locfin(\msE)$ and every~$t\geq 0$.
By~\ref{ass:Fe}, there exists~$F\in\msE$ such that~$\tabs{\hh{X}{\mssm}_t(\emparg,E)}<\tfrac{1}{2}$ for every~$x\in F^\complement$.
Since~$\mbfx\in \mbfX^\asym{\infty}_\locfin(\msE)$, there exists~$k\in \N$ so that~$x_i\in F^\complement$ for all~$i\geq k$.
Thus,
\begin{align*}
\mssh^\tym{\infty}_t(\mbfx,E^\tym{\infty})=&\ \prod_{i=1}^\infty \hh{X}{\mssm}_t(x_i,E) = \prod_{i=1}^k \hh{X}{\mssm}_t(x_i,E) \cdot \prod_{i=k+1}^\infty \hh{X}{\mssm}_t(x_i,E)
\\
\leq&\ \prod_{i=1}^k \hh{X}{\mssm}_t(x_i,E) \cdot \tparen{\tfrac{1}{2}}^\infty=0\comma
\end{align*}
which concludes the proof.
\end{proof}
\end{lem}

\begin{lem}\label{l:CoincidenceHK} Let~$(\mcX,\cdc)$ be a \TLDS satisfying~\ref{ass:wFe} and~\ref{ass:SC}. Then,
\begin{align}\label{eq:l:CoincidenceHK:0}
\Lb_\pfwd \mssh_t^\tym{\infty}\tparen{\mbfx_1,\diff\emparg}=\Lb_\pfwd \mssh_t^\tym{\infty}\tparen{\mbfx_2,\diff\emparg}\comma \qquad \gamma\in \dUpsilon^\sym{\infty}\comma\quad \mbfx_1,\mbfx_2\in \Lb^{-1}(\gamma)\comma \quad t\geq 0 \comma
\end{align}
as measures on~$\A_\mrmv(\msE)$.
\begin{proof}
Fix~$\gamma\in\dUpsilon^\sym{\infty}$ and~$t>0$.
By Lemma~\ref{l:SCProduct} both~$\Lb_\pfwd\mssh_t^\tym{\infty}\tparen{\mbfx_1,\diff\emparg}$ and~$\Lb_\pfwd\mssh_t^\tym{\infty}\tparen{\mbfx_2,\diff\emparg}$ are probability measures.
Thus, it is enough to show that they coincide on a family of sets generating the $\sigma$-algebra~$\A_\mrmv(\msE)$ (restricted to~$\dUpsilon^\sym{\infty}$) and closed w.r.t.\ finite intersections, e.g.~\cite[Lem.~I.1.9.4]{Bog07}.
In particular, since~$\mssm$ has full $\T$-support, the family of concentration sets as in~\eqref{eq:ConcentrationSetConfig2} with~$E\in\msE\cap \T$ and~$\mssm E>0$ generates~$\A_\mrmv(\msE)$.
Thus, it suffices to show~\eqref{eq:l:CoincidenceHK:0} for the family of sets in~\eqref{eq:p:FundamentalSetsConfig2:0}.
In turn, in light of Proposition~\ref{p:FundamentalSetsConfig2}, it suffices to show that~$\mssh_t^\tym{\infty}\tparen{\mbfx_1,\diff\emparg}$ and~$\mssh_t^\tym{\infty}\tparen{\mbfx_2,\diff\emparg}$ coincide on cylinder sets of the form~\eqref{eq:CylinderSet}.
To this end, let~$\mbfA$ be any such cylinder set, and note that, by~\ref{ass:SC}, the sequence~$\tseq{\log\hh{X}{\mssm}_t(x_i, A_i)}_i$, with~$A_i=X$ for all large enough~$i$, is eventually vanishing, and thus unconditionally summable.
As a consequence, and since~$\mssh_t^\tym{\infty}(\mbfx,\mbfA)=\exp \tbraket{\sum_{i=1}^\infty \log\hh{X}{\mssm}_t(x_i,A_i)}$, we have that~$\mssh_t^\tym{\infty}(\mbfx_\sigma,\mbfA_\sigma)=\mssh_t^\tym{\infty}(\mbfx,\mbfA)$ for every bijection~$\sigma\colon \N_1\to\N_1$, which concludes the assertion.
\end{proof}
\end{lem}

As a direct consequence of Lemma~\ref{l:CoincidenceHK}, for each~$t>0$ and~$\gamma\in\dUpsilon^\sym{\infty}$, the kernel~$\mssh^\tym{\infty}_t(\mbfx,\emparg)$, with~$\mbfx\in\mbfX^\asym{\infty}_\locfin(\msE)$, induces a probability measure~$\mssh_t^\dUpsilon(\gamma,\emparg)$ on~$\tparen{\dUpsilon,\A_\mrmv(\msE)}$, as follows.

\begin{defs}[Pointwise-defined heat kernel]\label{d:HKExtension}
Let~$(\mcX,\cdc)$ be a \TLDS satisfying~\ref{ass:wFe} and \ref{ass:SC}.
The \emph{pointwise-defined heat kernel measure} on~$(\dUpsilon,\A_\mrmv(\msE))$ is the measure
\begin{align}\label{eq:d:Heat:0}
\mssh^\dUpsilon_t(\gamma,\Lambda)\eqdef&\ \mssh^\tym{\infty}_t\tparen{\mbfx, \Lb^{-1}(\Lambda)}\comma \qquad \gamma\in\dUpsilon^\sym{\infty}\comma \mbfx\in \Lb^{-1}(\gamma) \comma t\geq 0\comma \quad \Lambda\in \A_\mrmv(\msE)\fstop
\end{align}
We extend the definition of~$\mssh^\dUpsilon_\bullet(\emparg,\diff\emparg)$ on~$\dUpsilon^\sym{\infty}$ to a non-relabeled kernel on~$\dUpsilon$, by extending the right-hand side of~\eqref{eq:d:Heat:0} for every~$\gamma\in\dUpsilon^\sym{n}$ and~$\mbfx\in \Lb^{-1}(\gamma)\subset X^\tym{n}$, for every~$n\in \N$, viz.
\begin{align}\label{eq:d:Heat:1}
\mssh^\dUpsilon_t(\gamma,\Lambda)\eqdef&\ \mssh^\tym{N}_t\tparen{\mbfx, \Lb^{-1}(\Lambda)}\comma \qquad \gamma\in\dUpsilon^\sym{N}\comma \mbfx\in \Lb^{-1}(\gamma) \comma t\geq 0\comma \Lambda\in \A_\mrmv(\msE)\comma N\in \overline\N_1\fstop
\end{align}
We also set, conventionally~$\mssh^\dUpsilon_\bullet(\emp,\set{\emp})\equiv 1$, where~$\emp$ denotes the empty configuration.
\end{defs}

The notation~$\mssh_t^\dUpsilon$ will be justified by Proposition~\ref{p:KonLytRoe} below, where it is shown that the measure~$\mssh_t^\dUpsilon$ coincides with the heat-kernel measure~$\hh{\dUpsilon}{\PP}_t$ for every~$t>0$.
As a consequence~\eqref{eq:d:Heat:0} and Lemma~\ref{l:SCProduct} we have that~$\mssh^\dUpsilon_t(\gamma,\diff\emparg)$ is a \emph{probability} measure on~$\tparen{\dUpsilon,\A_\mrmv(\msE)}$ for sufficiently many~$\gamma$'s.

\begin{cor}[Stochastic completeness II]\label{c:SCConfig1}
Let~$(\mcX,\cdc)$ be a \TLDS satisfying~\SCF. Then,
\begin{align*}
\mssh^\dUpsilon_t(\gamma,\dUpsilon^\sym{\infty})=&\ 1 \comma \qquad \gamma\in\dUpsilon^\sym{\infty}\comma\qquad t\geq 0\fstop
\end{align*}
\end{cor}

\subsubsection{Heat kernels: identifications}\label{sss:HeatIdentification}
Let~$(\mcX,\cdc)$ be a \TLDS, additionally satisfying~\SCF.
We are now ready to prove the identification of~$\mssh^\dUpsilon_\bullet$ with the heat-kernel measure~$\hh{\dUpsilon}{\PP}_\bullet$ representing the semigroup~$\TT{\dUpsilon}{\PP}_\bullet$ corresponding to~$\tparen{\EE{\dUpsilon}{\PP},\dom{\EE{\dUpsilon}{\PP}}}$.

\begin{prop}[Kernel's identification]\label{p:KonLytRoe}
Let~$(\mcX,\cdc)$ be a \TLDS satisfying~\SCF. Then, for every~$u\in L^2(\PP)$ and every~$t>0$ there exists a set~$\Theta_t[u]\in\A_\mrmv(\msE)$ of full~$\PP$-measure such that
\begin{align}\label{eq:HeatTheta}
\Theta_t[u]\ni\gamma\longmapsto \tparen{\widerep{\TT{\dUpsilon}{\PP}_t u}}(\gamma)\eqdef \int_{\dUpsilon^\sym{\infty}} u(\eta) \,\mssh^\dUpsilon_t(\gamma,\diff\eta)
\end{align}
is a $\PP$-version of~$\TT{\dUpsilon}{\PP}_t u\in L^2(\PP)$.

\begin{rem}[Comparison with~{\cite{KonLytRoe02}}]\label{r:ComparisonKonLytRoe}
Proposition~\ref{p:KonLytRoe} extends to the case of \TLDS's satisfying \SCF the analogous identification results previously established on manifolds by Yu.~G.~Kondratiev, E.~W.~Lytvynov, and M.~R\"ockner in~\cite{KonLytRoe02}, and by M.~Erbar and M.~Huesmann in~\cite{ErbHue15}.
Here, we reduce the necessary assumptions to a bare minimum, relying on the Feller property to show that the objects are well-defined, and on stochastic completeness to show that the construction is always non-trivial.

Furthermore, our Proposition~\ref{p:KonLytRoe} is different from the analogous statements in~\cite[Thm.~5.1]{KonLytRoe02} and~\cite[Thm.~2.4]{ErbHue15}.
Indeed, the idea underlying the aforementioned results is that to define a set~$\Theta\subset \dUpsilon$, with~$\PP\Theta=1$, of \emph{good configurations}, and then to perform the analysis of the heat kernel measures~$\mssh^\dUpsilon_\bullet(\gamma,\diff\emparg)$, with~$\gamma\in \Theta$.
The existence of such a set~$\Theta$ is non-trivial, even in the case when~$X$ is Riemannian manifold. In~\cite{ErbHue15}, the existence of~$\Theta$ is a consequence of Ricci-curvature lower-bounds for the base manifold~$X$.
In our setting we allow for the relevant set of configurations, of full $\PP$-measure, to depend on the function~$u$ the heat semigroup of which we are computing, as well as on the time~$t$.
In this sense, our Proposition~\ref{p:KonLytRoe} is weaker than either of the aforementioned results, in that our description is not uniform either in the time parameter or in the function~$u$.
This description however will be sufficient for our purposes below, and in particular to the study of several notions of curvature lower bounds for~$\dUpsilon$.

Finally, our results do not depend upon any geometric assumption, such as volume-growth estimates and Gaussian heat-kernel estimates as in~\cite{KonLytRoe02}, nor lower Ricci curvature bounds as in~\cite{ErbHue15}, which allows for a much wider range of application.
\end{rem}

\begin{lem}\label{l:Isometry}
The map~$\rep\emparg^\trid\colon \rep f\mapsto \rep f^\trid$ descends to an isometric order-preserving linear embedding
\begin{align*}
\emparg^\trid\colon L^1(\mssm)\longrar L^1(\PP)\fstop
\end{align*}
\begin{proof}
Let~$\rep f\in\mcL^1(\mssm)^+$. By a simple application of~\eqref{eq:MeckeConfig2}, we have that~$\ttnorm{\rep f^\trid}_{\mcL^1(\PP)}= \ttnorm{\rep f}_{\mcL^1(\mssm)}$, i.e.~$\rep\emparg^\trid\colon \rep f\mapsto \rep f^\trid$ is an isometry~$\mcL^1(\mssm)\to\mcL^1(\PP)$.
By standard arguments with the positive and negative parts of~$\rep f$, the above isometry extends to the whole of~$\mcL^1(\mssm)$.
In order to show that it descends to~$L^1(\mssm)$ it is enough to compute~$\ttnorm{\rep f^\trid-\reptwo f^\trid}_{\mcL^1(\PP)}=\ttnorm{\rep f-\reptwo f}_{\mcL^1(\mssm)}$ on different $\mssm$-representatives~$\rep f$, and~$\reptwo f$ of the same $\mssm$-class~$f\in L^1(\mssm)$.
The order-preserving property is straightforward.
\end{proof}
\end{lem}

\begin{proof}[Proof of Proposition~\ref{p:KonLytRoe}] We adapt the proofs of~\cite{KonLytRoe02}. Recall the definition of~$\De\subset \dom{\LL{X}{\mssm}}\cap L^1(\mssm)$ in~\eqref{eq:De}.

\paragraph{Step~1, cf.~{\cite[Lem.~5.1]{KonLytRoe02}}} We claim that, for every~$f\in \De$ and every~$t>0$ there exists a set~$\Theta_t[f]\in\A_\mrmv(\msE)$ of full~$\PP$-measure such that
\begin{equation}\label{eq:p:KonLytRoe:Step1}
\begin{aligned}
\int_{\dUpsilon^\sym{\infty}} \exp&\tparen{\log(1+f)^\trid \eta}\, \mssh^\dUpsilon_t(\gamma, \diff\eta)=
\\
=&\ \exp \quadre{\paren{\log\paren{1+\int_X f(y) \diff \hh{X}{\mssm}_t(\emparg, \diff y)}}^\trid\gamma}\comma 
\end{aligned}\qquad \gamma\in \Theta_t[f] \fstop
\end{equation}
Indeed, fix an $\mssm$-representative~$\rep f\in \mcL^1(\mssm)$ of~$f\in \De \subset L^1(\mssm)$.
Since~$\rep f\in\mcL^1(\mssm)$, and since $\abs{\log(1+t)}\leq \log(1-\delta)\abs{t}$ on~$[-\delta,\infty)$, we see that~$\log(1+\rep f)\in \mcL^1(\mssm)$.
By Lemma~\ref{l:Isometry}, for fixed~$t>0$ there exists~$\Theta_t[f]\in \A_\mrmv(\msE)$ of full $\PP$-measure and so that $\tabs{\tparen{\log\tparen{1+(\TT{X}{\mssm}_t \rep f)}}^\trid\gamma}<\infty$ for every~$\gamma\in \Theta_t[f]$, which makes the right-hand side in~\eqref{eq:p:KonLytRoe:Step1} well-defined and finite.
Thus, again for every~$\gamma\in \Theta_t[f]$, for some~$\mbfx\in\Lb^{-1}(\gamma)$, hence for every such~$\mbfx$ by~\eqref{eq:d:Heat:0},
\begin{align*}
\int_{\dUpsilon^\sym{\infty}} \exp\tparen{\log(1+\rep f)^\trid \eta}\, \mssh^\dUpsilon_t(\gamma, \diff\eta)=&\ 
\int_{\dUpsilon^\sym{\infty}} \exp\tparen{\log(1+\rep f)^\trid \eta}\, \Lb_\pfwd \mssh^\tym{\infty}_t(\gamma, \diff\eta)\fstop
\end{align*}
By Lemma~\ref{l:CoincidenceHK}, we therefore have that
\begin{align*}
\int_{\dUpsilon^\sym{\infty}} \exp\tparen{\log(1+\rep f)^\trid \eta}\, \mssh^\dUpsilon_t(\gamma, \diff\eta)=&\ \int_{\mbfX^\asym{\infty}_\locfin(\msE)} \exp\paren{\sum_{i=1}^\infty \log \tparen{1+\rep f(y_i)}} \bigotimes_{i=1}^\infty \hh{X}{\mssm}_t(x_i,\diff y_i)\fstop
\end{align*}
By Corollary~\ref{c:SCConfig1}, the measure~$\mssh^\tym{\infty}_t(\mbfx,\diff\emparg)$ is concentrated on~$\mbfX^\asym{\infty}_\locfin(\msE)$, and thus
\begin{align*}
\int_{\dUpsilon^\sym{\infty}} \exp\tparen{\log(1+\rep f)^\trid \eta}\, \mssh^\dUpsilon_t(\gamma, \diff\eta)=&\ \int_{X^\tym{\infty}} \exp\paren{\sum_{i=1}^\infty \log \tparen{1+\rep f(y_i)}} \bigotimes_{i=1}^\infty \hh{X}{\mssm}_t(x_i,\diff y_i)\comma
\intertext{whence, by~\ref{ass:SC},}
\int_{\dUpsilon^\sym{\infty}} \exp\tparen{\log(1+\rep f)^\trid \eta}\, \mssh^\dUpsilon_t(\gamma, \diff\eta)=&\ \prod_{i=1}^\infty \int_X \tparen{1+\rep f(y)} \, \hh{X}{\mssm}_t(x_i,\diff y)
\\
=&\ \exp \quadre{\paren{\log\paren{1+\int_X \rep f(y)\, \hh{X}{\mssm}_t(\emparg, \diff y)}}^\trid\gamma}\comma
\end{align*}
which proves the claim.

\paragraph{Step~2, cf.~{\cite[Lem.~5.2]{KonLytRoe02}}} We claim that, for every $\A_\mrmv(\msE)$-measurable~$u\colon \dUpsilon\rar [0,\infty)$,
\begin{align}\label{eq:p:KonLytRoe:Step2}
\iint_{{\dUpsilon^\sym{\infty}}^\tym{2}}  u(\eta)\, \mssh^\dUpsilon_t(\gamma, \diff\eta) \diff\PP(\gamma)= \int_{\dUpsilon^\sym{\infty}} u(\gamma) \diff\PP(\gamma)\comma \qquad t>0\fstop
\end{align}

It suffices to verify~\eqref{eq:p:KonLytRoe:Step2} for every~$u\in \Ex{\De}$.
Combining the $\mssm$-invariance of~$\TT{X}{\mssm}_t$ with~\eqref{eq:LaplacePoissonConfig2}, Equation~\eqref{eq:p:KonLytRoe:Step2} follows from~\eqref{eq:p:KonLytRoe:Step1} with~$\Theta_t[u]\eqdef \Theta_t[f]$.

\paragraph{Step~3, cf.~{\cite[Thm.~5.1, p.~18]{KonLytRoe02}}} The rest of the proof follows as in~\cite{KonLytRoe02}, having care to substitute~\cite[Prop.~2.1]{KonLytRoe02}, i.e.~\cite[Prop.~4.1]{AlbKonRoe98}, with Theorem~\ref{t:AKR4.1}, and noting that~$L^2(\PP)$ is separable, hence for every sequence~$\seq{u_k}_k\subset \Ex{\De}$ approximating~$u\in L^2(\PP)$, we have that~$\PP\tparen{\cap_{k\in \N} \Theta_t[u_k]}=1$.
\end{proof}
\end{prop}

Before commenting on Proposition~\ref{p:KonLytRoe} let us establish some further notation.

\begin{notat} We let~$\TTt^\asym{\infty}_\bullet$ denote the semigroup corresponding to~$\mssh^\tym{\infty}_\bullet$, viz.
\begin{align*}
\tparen{\TTt^\asym{\infty}_t \rep U}(\mbfx)\eqdef& \int_{X^\tym{\infty}} \rep U(\mbfy) \, \mssh_t^\tym{\infty}(\mbfx, \diff \mbfy) \comma & \mbfx\in X^\tym{\infty}\comma &t> 0 \comma \qquad \rep U\in \preW(\msE)\fstop
\end{align*}
Recalling the notation~$\rep U_{\mbfx,p}$ in~\eqref{eq:d:DiConfig2:0}, further let
\begin{align}\label{eq:n:SemigroupP}
(\TTt^p_t \rep U)(\mbfx)\eqdef& \int_X \rep U_{\mbfx,{p}}(y)\, \hh{X}{\mssm}_t(x,\diff y)\comma & \mbfx\in X^\tym{\infty}\comma &t> 0 \comma \qquad \rep U\in \preW(\msE)\comma
\end{align}
and define a corresponding infinite-product operator on~$\preW(\msE)$ by
\begin{align*}
\TTt^{\asym{\infty}, {p}}_t \eqdef& \bigotimes_{\substack{q=1\\q\neq p}}^\infty \TTt^q_t \comma \qquad t> 0\comma \qquad p\in\N_1 \fstop
\end{align*}
\end{notat}

We note here that Proposition~\ref{p:KonLytRoe} completes the identification of forms provided in~\cite[\S{3.5}]{LzDSSuz21} with the identification of the corresponding heat kernel measures. Indeed, as a consequence of~\eqref{eq:d:Heat:0} and Proposition~\ref{p:KonLytRoe}, we may write, for~$u\in L^2(\PP)$,
\begin{align}\label{eq:IdentificationSemigroup}
\Lb^*\tparen{
\widerep{\TT{\dUpsilon}{\PP}_t u}}(\mbfx)= \TTt^\asym{\infty}_t (\Lb^*\rep u)(\mbfx)\comma  \qquad \forallae{\lb_\pfwd\PP} \mbfx\in\Lb^{-1}(\Theta_t[u])\comma \quad t\geq 0 \comma
\end{align}
where~$\Theta_t[u]\in\A_\mrmv(\msE)$ is given by Proposition~\ref{p:KonLytRoe}.

Together with Proposition~\ref{p:ExtDomConfig2} and \cite[Cor.~3.61]{LzDSSuz21}, Proposition~\ref{p:KonLytRoe} shows that we may (and henceforth occasionally will) neglect the distinction among $\PP$-, resp.~$\lb_\pfwd\PP$-, representatives, both on~$\dUpsilon$ and on $\mbfX^\asym{\infty}_\locfin(\msE)$.
Finally, Proposition~\ref{p:KonLytRoe} also complements the statement of Proposition~\ref{p:ExtDomConfig2}\ref{i:p:ExtDomConfig2:3}, in view of the next Corollary.

\begin{cor}\label{c:ExchangeHeatSemigroupTrid}
Let~$(\mcX,\cdc)$ be a \TLDS satisfying \SCF. Then,
\begin{align}\label{eq:c:ExchangeHeatSemigroupTrid:0}
\TT{\dUpsilon}{\PP}_t\tclass[\PP]{\rep f^\trid}=\class[\PP]{(\TT{X}{\mssm}_t\rep f)^\trid} \comma \qquad f\in L^1(\mssm)\comma \quad t>0\fstop
\end{align}
\begin{proof}
Fix~$f\in\Dz$, and~$t>0$. In light of~\eqref{eq:HeatTheta}, we have that, for some~$\Theta_t[f^\trid]\in\A_\mrmv(\msE)$ of full $\PP$-measure, and for every~$\mbfx\in \Lb^{-1}(\Theta_t[f^\trid])$,
\begin{align*}
\tparen{\TT{\dUpsilon}{\PP}_t(f^\trid)}(\Lb(\mbfx))= \int_{\dUpsilon^\sym{\infty}} \rep f^\trid(\eta) \,\mssh^\dUpsilon_t(\Lb(\mbfx),\diff\eta)
\overset{\eqref{eq:d:Heat:0}}{=} \int_{\dUpsilon^\sym{\infty}} \sum_{y\in \eta} \rep f(y) \, \Lb_\pfwd\mssh^\tym{\infty}_t(\Lb(\mbfx),\diff \eta)\fstop
\end{align*}
By Corollary~\ref{c:SCConfig1}, the measure~$\mssh^\tym{\infty}_t(\mbfx,\diff\emparg)$ is concentrated on~$\mbfX^\asym{\infty}_\locfin(\msE)$, and thus
\begin{align*}
\tparen{\TT{\dUpsilon}{\PP}_t(f^\trid)}(\Lb(\mbfx))= \int_{\mbfX^\asym{\infty}_\locfin(\msE)} \sum_{i=1}^\infty \rep f(y_i) \, \bigotimes_{i=1}^\infty \hh{X}{\mssm}_t(x_i,\diff y_i) = \int_{X^\tym{\infty}} \sum_{i=1}^\infty \rep f(y_i) \, \bigotimes_{i=1}^\infty \hh{X}{\mssm}_t(x_i,\diff y_i) \fstop
\end{align*}
Since~$f\in\Dz$, we may exchange the series and integral signs in the previous expression, and we obtain that
\begin{align*}
\tparen{\TT{\dUpsilon}{\PP}_t(f^\trid)}(\Lb(\mbfx))= \sum_{i=1}^\infty \int_{X^\tym{\infty}} \rep f(y_i) \, \bigotimes_{i=1}^\infty \hh{X}{\mssm}_t(x_i,\diff y_i)\comma
\end{align*}
hence, by~\ref{ass:SC},
\begin{align}\label{eq:ExchangeHeatSemigroupTrid}
\tparen{\TT{\dUpsilon}{\PP}_t(f^\trid)}(\Lb(\mbfx))=&\ \sum_{i=1}^\infty \int_X \rep f(y_i) \, \hh{X}{\mssm}_t(x_i,\diff y_i)= \sum_{i=1}^\infty \tparen{\TT{X}{\mssm}_t \rep f}(x_i)= (\TT{X}{\mssm}_t\rep f)^\trid \Lb(\mbfx) \fstop
\end{align}
As a consequence of Lemma~\ref{l:Isometry}, together with the density of~$\Dz$ in~$L^1(\mssm)$ and the fact that~$\TT{X}{\mssm}_t$, resp.~$\TT{\dUpsilon}{\PP}_t$ is a bounded operator on~$L^1(\mssm)$, resp.~$L^1(\PP)$, we see that~\eqref{eq:ExchangeHeatSemigroupTrid} holds in fact for every~$f\in L^1(\mssm)$, which concludes the proof by letting~$\gamma\eqdef \Lb(\mbfx)\in \Theta_t[f^\trid]$ and recalling that the latter has full $\PP$-measure.
\end{proof}
\end{cor}

\subsection{Bakry--\'Emery curvature condition}
We present here a proof of the Bakry--\'Emery gradient estimate for the Poisson configuration space~$\tparen{\dUpsilon, \SF{\dUpsilon}{\PP}, \PP}$.

\begin{proof}[Proof of Theorem~\ref{t:BE}]
Throughout the proof, let~$t>0$,~$c\geq 1$, and~$K\in\R$ be fixed.

\noindent\ref{i:t:BE:1} $\Longrightarrow$ \ref{i:t:BE:2}.
Firstly, we show~\ref{ass:BE} for~$\dUpsilon$ for a function~$u\in \Cyl{\Dz}$. Set~$\rep U\eqdef \Lb^*u$.
Recall the notation in~\eqref{eq:d:DiConfig2:1} and~\eqref{eq:n:SemigroupP}. For all~$p\in \N_1$, by definition,
\begin{align*}
\tparen{\rep\cdc^p(\TTt^p_t \rep U)}(\mbfx)=\tparen{\repSF{X}{\mssm}(\TT{X}{\mssm}_t \rep U_{\mbfx,{p}})}(x_p)\comma \qquad \mbfx\in X^\tym{\infty} \fstop
\end{align*}
In view of Lemma~\ref{l:AbsContPoissonProj}, it follows from~\ref{ass:BE} for~$X$, that 
\begin{align*}
\begin{aligned}
\tparen{\rep\cdc^p(\TTt^p_t \rep U)}(\mbfx)=&\ \tparen{\repSF{X}{\mssm}(\TT{X}{\mssm}_t \rep U_{\mbfx,{p}})}(x_p)
\\
\leq&\ c\, e^{-2Kt}\tparen{\TT{X}{\mssm}_t\,\repSF{X}{\mssm}(\rep U_{\mbfx,{p}})}(x_p)
\\
=&\ c\, e^{-2Kt}\, \tparen{\TTt^p_t \rep\cdc^p(\rep U)}(\mbfx)
\end{aligned}
\qquad \forallae{\lb_\pfwd\PP} \mbfx\in X^\tym{\infty}\fstop
\end{align*}
By Jensen's inequality and sub-Markovianity of~$\TTt^{\asym{\infty}, {p}}_t$ (see Lemma~\ref{l:SubmarkovianityTp} below) we thus have that
\begin{align*}
\tparen{\rep\cdc^p(\TTt^\asym{\infty}_t \rep U)}(\mbfx)\leq \TTt^{\asym{\infty}, {p}}_t\tparen{\rep\cdc^p(\TTt^p_t \rep U)}(\mbfx)\leq c\, e^{-2Kt}\tparen{\TTt^{\asym{\infty}}_t \rep\cdc^p(\rep U)}(\mbfx) \quad \forallae{\lb_\pfwd \PP} \mbfx\in X^{\tym{\infty}}\comma
\end{align*}
and, summing over~$p\in \N_1$,
\begin{align}\label{eq:t:BE:2b}
\tparen{\rep\cdc^{\asym{\infty}}(\TTt^\asym{\infty}_t \rep U)}(\mbfx)\leq c\, e^{-2Kt}\tparen{\TTt^{\asym{\infty}}_t \rep\cdc^{\asym{\infty}}(\rep U)}(\mbfx) \quad \forallae{\lb_\pfwd\PP} \mbfx\in X^{\tym{\infty}}\fstop
\end{align}

In order to pass this estimate to an estimate on~$\Upsilon(\msE)$, we use~\eqref{eq:IdentificationSemigroup}.
Namely, applying~$\cdc^\asym{\infty}$ on both sides of~\eqref{eq:IdentificationSemigroup}, it follows from~\cite[Cor.~3.61]{LzDSSuz21} that
\begin{equation}\label{eq:t:BE:2a}
\begin{aligned}
\Lb^*\tparen{\SF{\Upsilon}{\PP} (\TT{\Upsilon}{\PP}_t u)}(\mbfx)=&\ \cdc^\asym{\infty} \tparen{\Lb^* (\TT{\Upsilon}{\PP}_t u)}(\mbfx)
\\
=&\ \cdc^\asym{\infty}\tparen{\TTt^\asym{\infty}_t (\Lb^*u)}(\mbfx) 
\end{aligned}
\quad \forallae{\lb_\pfwd\PP} \mbfx\in X^{\tym{\infty}}\comma\qquad u\in L^2(\PP)\fstop
\end{equation}
Analogously, again by~\cite[Cor.~3.61]{LzDSSuz21} and~\eqref{eq:IdentificationSemigroup},
\begin{align}\label{eq:t:BE:2c}
\TTt^\asym{\infty}_t \cdc^\asym{\infty}(\Lb^*u)=\Lb^*\tparen{\TT{\Upsilon}{\PP}_t \SF{\Upsilon}{\PP}(u)}\quad \forallae{\lb_\pfwd\PP} \mbfx\in X^{\tym{\infty}}\comma\qquad u\in \dom{\SF{\Upsilon}{\PP}}\fstop
\end{align}

Combining~\eqref{eq:t:BE:2b} with~\eqref{eq:t:BE:2a} and~\eqref{eq:t:BE:2c} proves
\begin{align}\label{eq:t:BE:0}
\SF{\Upsilon}{\PP}\tparen{\TT{\Upsilon}{\PP}_t u} \leq c\, e^{-2Kt}\,\TT{\Upsilon}{\PP}_t \SF{\Upsilon}{\PP}(u) \as{\PP} \comma \qquad u\in \Cyl{\Dz} \fstop
\end{align}
The extension of~\eqref{eq:t:BE:0} to~$u\in\dom{\EE{\Upsilon}{\PP}}$ follows by $L^2(\pi)$-lower-semicontinuity of~$\SF{\dUpsilon}{\PP}(\emparg)$ and density of~$\CylQP{\PP}{\Dz}$ in~$\dom{\EE{\dUpsilon}{\PP}}$.

\medskip

\noindent\ref{i:t:BE:2} $\Longrightarrow$ \ref{i:t:BE:1}.
Fix~$f\in \dom{\EE{X}{\mssm}}$. By the assumption of~\ref{ass:BE} for~$\dUpsilon$ with function~$f^\trid$, we have that
\begin{align*}
\SF{\dUpsilon}{\PP}\tparen{\TT{\dUpsilon}{\PP}_t f^\trid}\leq c\, e^{-Kt}\, \TT{\dUpsilon}{\PP}_t\, \SF{\dUpsilon}{\PP} (f^\trid) \as{\PP} \fstop
\end{align*}
Applying~\eqref{eq:c:ExchangeHeatSemigroupTrid:0} and Proposition~\ref{p:ExtDomConfig2}\ref{i:p:ExtDomConfig2:3} on both sides (in reverse order) yields
\begin{align*}
\tparen{\SF{X}{\mssm}\tparen{\TT{X}{\mssm}_t f}}^\trid\leq \tparen{c\, e^{-Kt}\,\TT{X}{\mssm}_t\, \SF{X}{\mssm} (f)}^\trid \as{\PP}\comma
\end{align*}
and from the order-preserving property of~$\emparg^\trid$ in Lemma~\ref{l:Isometry} we conclude~\ref{ass:BE} for~$X$ with function~$f\in\dom{\EE{X}{\mssm}}$.
Since~$f$ was arbitrary, this proves the assertion.
\end{proof}

\begin{lem}[Sub-Markovianity of~$\TTt^{\asym{\infty},p}_\bullet$]\label{l:SubmarkovianityTp}
Let~$(\mcX,\cdc)$ be a \TLDS satisfying~\SCF, and fix~$p\in\N_1$ and~$t>0$. Then, the operator~$\TTt^{\asym{\infty},p}_t$ is pointwise sub-Markovian, viz., for every bounded $\boldSigma$-measurable~$\rep U \colon X^\tym{\infty}\to\R$,
\begin{align*}
0\leq \rep U(\mbfx) \leq 1 \implies 0\leq \tparen{\TTt^{\asym{\infty},p}_t \rep U}(\mbfx) \leq 1\comma \qquad \mbfx\in X^\tym{\infty}\fstop
\end{align*}
\begin{proof}
Since~$\TT{X}{\mssm}_t\colon \mcL^\infty(\mssm)\to \mcL^\infty(\mssm)$ is (pointwise) sub-Markovian, it suffices to show that sub-Markovianity tensorizes over infinite products.
Indeed let~$0\leq \rep U \leq 1$ everywhere on~$X^\tym{\infty}$.
We have that
\begin{align*}
\TTt^{\asym{\infty},p}_t\rep U(\mbfx)=&\ \int_{X^\tym{\infty}} \bigotimes_{\substack{q=1\\ q\neq p}}^\infty \hh{X}{\mssm}_t(x_q,\diff y_q)\, \rep U(\mbfy)
\\
=&\ \int_{X^\tym{\infty}} \bigotimes_{\substack{q=2\\ q\neq p}} \hh{X}{\mssm}(x_q,\diff y_q) \int_X \hh{X}{\mssm}_t(x_1,y)\, \rep U_{\mbfy,1}(y) \diff y \fstop
\end{align*}
Since~$\rep U\leq 1$ everywhere on~$X^\tym{\infty}$, we conclude that~$\rep U_{\mbfy,1}\leq 1$ everywhere on~$X$ for every~$\mbfy\in X^\tym{\infty}$, hence, by sub-Markovianity of~$\TT{X}{\mssm}_t\colon \mcL^\infty(\mssm)\to \mcL^\infty(\mssm)$,
\begin{align*}
\TTt^{\asym{\infty},p}_t\rep U(\mbfx) \leq \int_{X^\tym{\infty}} \bigotimes_{\substack{q=2\\ q\neq p}} \hh{X}{\mssm}_t(x_q,\diff y_q) \int_X \hh{X}{\mssm}_t(x_1,\diff y) \leq 1 \fstop
\end{align*}
Since, trivially,~$\TTt^{\asym{\infty},p}_t\rep U\geq 0$ everywhere on~$X^\tym{\infty}$, the proof is concluded.
\end{proof}
\end{lem}

\section{Synthetic Ricci-curvature lower bounds on configuration spaces}\label{s:RicciBounds}
In this section we study the validity of synthetic Ricci-curvature lower bounds for configuration spaces.
In particular, we show that the base \MLDS~$(\mcX,\cdc,\mssd)$ satisfies either one of the following properties if and only if the same property holds for the configuration space~$\tparen{\dUpsilon,\SF{\dUpsilon}{\PP},\mssd_\dUpsilon}$
\begin{itemize}
\item the \emph{logarithmic Harnack inequality} (Dfn.~\ref{d:LogH});
\item the \emph{Wasserstein contractivity estimate} (Dfn.~\ref{d:WC});
\item the \emph{Evolution Variation Inequality} (Dfn.~\ref{d:EVI}).
\end{itemize}

The validity of these properties and their mutual implications on infinitesimally Hilbertian metric measure spaces is well-studied, see e.g.~\cite{AmbGigSav14b,AmbGigSav15,KopStu21}.

Throughout this section we let
\begin{align*}
I_K(t)\eqdef \int_0^t e^{Kr} \diff r= \frac{e^{K t}-1}{K} \fstop
\end{align*}

\subsection{Preliminaries on base spaces}
Let us start by recalling the necessary definitions in the setting of~$\EMLDS$'s.
As usual, the definitions below are given for base spaces, but the generality is sufficient to include as well configuration spaces, resulting in a unified treatment of all properties of interest.

\subsubsection{Wasserstein distance}
Let~$(\mcX,\T,\mssd)$ be an \emph{extended}-metric topological space in the sense of~\cite[Dfn.~\ref{d:AES}]{LzDSSuz21}, and denote by~$\msP(X)$ the space of all Radon probability measures on~$(X,\A)$.
By uniqueness of the Charath\'eodory completion,~$\msP(X)$ is in natural correspondence with the space of all probability measures on either~$(X,\A^*)$ or~$(X,\Bo{\T})$, so that no confusion may arise in omitting the $\sigma$-algebra from the notation, and we may regard elements in~$\msP(X)$ as defined on either~$\Bo{\T}$,~$\A$, or~$\A^*$.

On~$\msP(X)$ we define the \emph{extended $L^2$-Wasserstein distance}
\begin{align*}
W_{2,\mssd}(\mu_1,\mu_2)\eqdef \inf_{\kappa\in\Cpl(\mu_1,\mu_2)}\paren{\int_{X^\tym{2}}\mssd(x,y)^2\diff\kappa(x,y)}^{1/2}\comma
\end{align*}
where~$\Cpl(\mu,\nu)$ denotes the set of all coupling~$\kappa\in\msP(X^\tym{2},\A^\otym{2})$ so that~$\pr^i_\pfwd \kappa=\mu_i$ for~$i=1,2$.
The functional~$W_{2,\mssd}$ is an extended distance on~$\msP(X)$.

\paragraph{The case of distances} When~$\mssd$ is a distance inducing the topology~$\T$, we define for some fixed~$x_0\in X$, the \emph{$L^2$-Wasserstein space}
\begin{align*}
\msP_{2,\mssd}(X)\eqdef \set{\mu\in\msP(X): \int_X \mssd(x_0,x)^2\diff\mu(x)<\infty}\eqdef B^{W_{2,\mssd}}_\infty(\delta_{x_0})\fstop
\end{align*}
By triangle inequality for~$W_{2,\mssd}$ it is readily seen that~$\msP_{2,\mssd}$ is in fact independent of the choice of~$x_0$, and contains all measures in~$\msP(X)$ with finite $L^2$ $\mssd$-moment.
It is well-known that $W_{2,\mssd}$ inherits the metric properties of~$\mssd$, see e.g.~\cite{AmbGig11}.
In particular, if~$(X,\mssd)$ is either a complete, length, or geodesic extended metric space, then~$\msP_{2,\mssd}(X)$ is so as well.

\paragraph{The case of extended distances}
When~$(X,\T,\mssd)$ is a general extended-metric topological space, a thorough discussion of the properties of~$\tparen{\msP(X),W_{2,\mssd}}$ is found in~\cite[\S5]{AmbErbSav16}.
Also in this case, $\msP(X)$ inherits properties of the base space: endowed with the narrow topology~$\T_\mrmn$, the space~$\tparen{\msP(X),\T_\mrmn,W_{2,\mssd}}$ is again an extended-metric topological space in the sense of~\cite[Dfn.~\ref{d:AES}]{LzDSSuz21}, see~\cite[Prop.~5.13]{AmbErbSav16}; the extended metric space~$\tparen{\msP(X),W_{2,\mssd}}$ inherits the completeness of the base space~$(X,\mssd)$, see~\cite[Prop.~5.4]{AmbErbSav16}.
We refer the reader to~\cite{AmbErbSav16} for further results concerning the space~$\tparen{\msP(X),W_{2,\mssd}}$ on extended-metric topological spaces.

\subsubsection{Boltzmann--Shannon entropy and Fisher information}
Let~$(\mcX,\mssd)$ be a metric local structure, and denote by~$\msP^\mssm(X)$ the space of all probability measures on~$(X,\A)$ absolutely continuous w.r.t.~$\mssm$.

\begin{defs}[Boltzmann--Shannon entropy]
The \emph{Boltzmann--Shannon entropy} of~$\mu\in\msP(X)$ is the functional~$\Ent_\mssm$ defined by
\begin{align*}
\Ent_\mssm(\mu)\eqdef \int_X \rho\log \rho \diff\mssm \quad \text{if} \quad \mu=\rho\mssm\in\msP^\mssm(X)\fstop
\end{align*}
If otherwise~$\mu\notin\msP^\mssm(X)$, we set~$\Ent_\mssm(\mu)\eqdef +\infty$.
Note that~$\Ent_\mssm\geq 0$.
\end{defs}

\begin{defs}[Fisher information]
Let~$(\mcX,\cdc)$ be a \TLDS.
The \emph{Fisher information} of~$(\mcX,\cdc)$ is the functional~$\FF{X}{\mssm}$ defined by
\begin{align*}
\dom{\FF{X}{\mssm}}\eqdef \set{\mu\in\msP(X): \mu=f\mssm\comma \sqrt{f}\in\dom{\EE{X}{\mssm}}} \comma \qquad \FF{X}{\mssm}(f)\eqdef 4\EE{X}{\mssm}\tparen{\sqrt{f}}\fstop
\end{align*}
One has that, e.g.~\cite[Prop.~4.1]{AmbGigSav15},
\begin{align*}
\FF{X}{\mssm}(f)=\int_{\set{f>0}} \frac{\SF{X}{\mssm}(f)}{f} \diff \mssm \fstop
\end{align*}
\end{defs}

\paragraph{Fisher information}
Let us recall some standard facts connecting the entropy functional~$\Ent_\mssm$ and the Fisher information~$\FF{X}{\mssm}$ relative to the Cheeger energy of an infinitesimally Hilbertian extended-metric measure space.
For a treatment of Cheeger energies in this generality see~\cite[\S\ref{sss:CheegerE}]{LzDSSuz21},~\cite{Sav19}, or~\cite{AmbErbSav16}.

\begin{lem}\label{l:EntropyFisher}
Let~$(X,\mssd,\mssm)$ be an infitesimally Hilbertian extended metric measure space.
Further let~$\mu=f\mssm\in \dom{\Ent_\mssm}$ and set~$\mu_t\eqdef (\TT{X}{\mssm}_t f)\mssm$.
Then,~$t\mapsto \Ent_\mssm(\mu_t)$ is non-increasing and locally absolutely continuous.
Furthermore,
\begin{align*}
\int_0^T \FF{X}{\mssm}(\mu_t) \diff t \leq 2\, \Ent_\mssm(\mu_0) \fstop
\end{align*}
The curve~$t\mapsto \mu_t$ is absolutely continuous w.r.t.~$W_{2,\mssd}$ and
\begin{align*}
\abs{\dot\mu_t}^2\leq \FF{X}{\mssm}(\mu_t) \quad \as{\diff t} \fstop
\end{align*}
\begin{proof}
See~\cite[proof of Lem.~5.2]{ErbHue15} and references therein.
\end{proof}
\end{lem}

\subsubsection{Synthetic Ricci-curvature lower bounds on base spaces}
In this section we recall several different notions of synthetic Ricci-curvature lower bounds.
Such synthetic formulations, and in particular the curvature-dimension condition~$\CD$ introduced by J.~Lott and C.~Villani in~\cite{LotVil09}, and by K.-T.~Sturm in~\cite{Stu06a,Stu06b} have played a major role in recent developments in non-smooth geometry.
Here, we focus on the Riemannian curvature dimension condition~$\RCD$ and on equivalent formulations of the latter.
We adapt to \EMLDS's the original definitions for metric measure spaces due to L.~Ambrosio, N.~Gigli, and G.~Savar\'e,~\cite{AmbGigSav14b, AmbGigSav15,AmbGigMonRaj12}.
In addition to these, we also recall here one \emph{a priori} weaker condition: the logarithmic Harnack inequality, introduced on manifolds by F.-Y.~Wang in~\cite{Wan97b}.

\paragraph{Logarithmic Harnack inequality}
One first incarnation of synthetic Ricci-curvature lower bounds is the following logarithmic Harnack inequality for the heat semigroup.

\begin{defs}[Log-Harnack inequality]\label{d:LogH}
Let~$(\mcX,\cdc,\mssd)$ be an~\EMLDS satisfying~\SCF, and~$K\in\R$. We say that~$(\mcX,\cdc,\mssd)$ satisfies a \emph{logarithmic Harnack inequality} with rate~$K$ if, for every~$f\in L^\infty(\mssm)$, and every~$t>0$,
\begin{align}\tag*{$(\mathsf{logH})_{\ref{d:LogH}}$}\label{ass:LogH}
\tparen{\TT{X}{\mssm}_t\log f}(x)\leq \log\tparen{\TT{X}{\mssm}_tf}(y)+\frac{\mssd(x,y)^2 }{4I_{2K}(t)}\quad \forallae{\mssm^\otym{2}} (x,y)\in X^\tym{2}\fstop
\end{align}
\end{defs}

For the interpretation of~\ref{ass:LogH} in the sense of Ricci-curvature lower bounds, see Remark~\ref{r:LogH} below.

\paragraph{Wasserstein contractivity of the heat flow}
Let~$(\mcX,\cdc,\mssd)$ be an~\EMLDS.
Further suppose that~$(\mcX,\cdc)$ satisfies~\ref{ass:SC} and that~$\hh{X}{\mssm}_\bullet(x,\diff\emparg)$, defined for \emph{every}~$x\in X$, is a probability kernel on~$\Bo{\T}$ representing~$\TT{X}{\mssm}_\bullet\colon L^2(\mssm)\to L^2(\mssm)$ as in~\eqref{eq:n:FormConfig2:1}.
We define the following semigroups on measures:
\begin{subequations}
\begin{align}
\label{eq:SemigroupHK}
(\TT{X}{\mssm}_t \mu)A\eqdef& \int_A \TT{X}{\mssm}_t \rho \diff\mssm \comma \qquad \mu=\rho\mssm\in\msP^\mssm(X)\comma \qquad A\in\Bo{\T}\comma \quad t>0\comma
\\
\label{eq:KernelHK}
(\hh{X}{\mssm}_t\mu) A\eqdef& \int_X \hh{X}{\mssm}_t(x,A) \diff\mu(x)\comma \qquad \mu\in\msP(X)\comma
\qquad A\in\Bo{\T}\comma \quad t>0\fstop
\end{align}
\end{subequations}
Since~$\TT{X}{\mssm}_\bullet\colon L^2(\mssm)\to L^2(\mssm)$ is represented by~$\hh{X}{\mssm}_\bullet$, it is clear that
\begin{align}\label{eq:TtoHExtension}
\hh{X}{\mssm}_\bullet\colon \msP(X)\to \msP(X)\quad \text{extends} \quad \TT{X}{\mssm}_\bullet\colon \msP^\mssm(X)\to \msP^\mssm(X) \fstop
\end{align}
However, it is convenient to distinguish the two, for the purpose of stating the following definition.

\begin{defs}[Wasserstein contractivity estimates]\label{d:WC}
Fix
\begin{equation}\label{eq:WCConstant}
c\colon [0,\infty)\to [1,\infty) \quad \text{with}\quad \lim_{t\to 0} c(t)=c(0)=1\fstop
\end{equation}
Further let~$(\mcX,\cdc,\mssd)$ be an~\EMLDS satisfying~\ref{ass:SC}, and define
\begin{itemize}
\item the \emph{semigroup Wasserstein contractivity} estimate for~$(\mcX,\cdc,\mssd)$ with rate~$c(t)$
\begin{align}\tag*{$(\sgWC{\mssd})$}\label{ass:sgWC}
W_{2,\mssd}(\TT{X}{\mssm}_t\mu_0, \TT{X}{\mssm}_t\mu_1)\leq c(t) W_{2,\mssd}(\mu_0,\mu_1)\comma \qquad \mu_0,\mu_1\in\msP^\mssm(X)\semicolon
\end{align}
\end{itemize}
If~$(\mcX,\cdc,\mssd)$ is additionally endowed with~$\hh{X}{\mssm}_\bullet(x,\diff\emparg)$ as above (defined for every~$x\in X$), define
\begin{itemize}
\item the \emph{kernel Wasserstein contractivity} estimate for~$(\mcX,\cdc,\mssd,\hh{X}{\mssm}_\bullet)$ with rate~$c(t)$
\begin{align}\tag*{$(\kWC{\mssd})$}\label{ass:kWC}
W_{2,\mssd}(\hh{X}{\mssm}_t\mu_0, \hh{X}{\mssm}_t\mu_1)\leq c(t) W_{2,\mssd}(\mu_0,\mu_1)\comma \qquad \mu_0,\mu_1\in\msP(X)\fstop
\end{align}
\end{itemize}
\end{defs}

\begin{rem}\label{r:kWCsgWC}
In light of~\eqref{eq:TtoHExtension} it is clear that
\begin{align*}
(\kWC{\mssd}) \Longrightarrow (\sgWC{\mssd})\fstop
\end{align*}
The reverse implication holds if~$\hh{X}{\mssm}_\bullet\colon \msP(X)\to \msP(X)$ is the unique extension of
\[
\TT{X}{\mssm}_\bullet\colon \msP^\mssm(X)\to \msP^\mssm(X)\fstop
\]
Typically, this requires some additional assumptions on the $\T$-continuity of~$\hh{X}{\mssm}_\bullet(\emparg, A)$ for fixed~$A\in\Bo{\T}$.
For instance, it is shown in~\cite[Prop.~3.2(i)--(iii)]{AmbGigSav15} when~$(\mcX,\cdc,\mssd)$ is an \MLDS (\emph{not}: an \EMLDS) under the additional assumption that~$\TT{X}{\mssm}_tf\in \bLip(\mssd)$ and~$\Li[\mssd]{\TT{X}{\mssm}_t f}\leq a(t)\Li[\mssd]{f}$ for all~$f\in \bLip(\mssd)\cap L^2(\mssm)$ and every~$t>0$ for some constant~$a(t)>0$ independent of~$f$.
\end{rem}

\paragraph{Evolution Variation Inequality}
Let~$(\mcX,\mssd)$ be an~\EMLDS. Let~$F\colon X\to\R\cup\set{+\infty}$ be a function with non-empty domain~$\dom{F}\eqdef \set{x\in X: F(x)\in\R}$.
For a curve~$x_\bullet\colon (0,\infty)\to\R\cup\set{+\infty}$, denote by~$\diff_t^+$ its upper right derivative at time~$t$.
We recall the following definition of $\EVI_K$ gradient flow, $K\in\R$.

\begin{defs}[\emph{$\EVI_K$ gradient flow},~{\cite[Dfn.~3.2]{AmbErbSav16}}]\label{d:EVIGradF}
Let~$K\in\R$.
A curve
\begin{equation*}
x_\bullet\in \AC^2_\loc\tparen{(0,\infty);(\dom{F},\mssd)}
\end{equation*}
is an~$\EVI_K$ gradient curve of~$F$ if~$t\mapsto F(x_t)$ is lower semi-continuous on~$(0,\infty)$ and for every~$y\in\dom{F}$ with~$\mssd(y,x_t)<\infty$ for some (hence all)~$t\in (0,\infty)$ it holds that
\begin{align}\label{eq:d:EVI:0}
\diff_t^+ \tfrac{1}{2}\mssd(x_t,y)^2+\tfrac{K}{2} \mssd(x_t,y)^2\leq F(y)-F(x_t) \comma \qquad t>0\fstop
\end{align}
We extend the definition to~$x_*\in B^\mssd_\infty\tparen{\dom{F}}$ by saying that~$x_\bullet$ as above starts from~$x_*$ if $\liminf_{t\downarrow 0} F(x_t)\geq F(x_*)$ and~$\lim_{t\downarrow 0}\mssd(x_t,y)=\mssd(x_*,y)$ for every~$y\in \cl_\mssd\tparen{\dom{F}}$.
\end{defs}

\begin{defs}[$\EVI_K$ property]\label{d:EVI}
We say that an~\EMLDS~$(\mcX,\cdc,\mssd)$ endowed with a pointwise defined heat-kernel measure~$\hh{X}{\mssm}_\bullet$ satisfies~$(\EVI_{K,\mssd,\mssm})$ if for every~$\mu\in B^{W_{2,\mssd}}_\infty\tparen{\dom{\Ent_\mssm}}$ the kernel heat flow~$\tseq{\hh{X}{\mssm}_t\mu}_{t\geq 0}$ in~\eqref{eq:KernelHK} is an $\EVI_K$ gradient flow of~$\Ent_\mssm$ in the sense of Definition~\ref{d:EVIGradF}.
\end{defs}

\paragraph{Riemannian curvature-dimension condition}
Let us now summarize an equivalent result for some of the properties listed above, in the case of \MLDS's (\emph{not}: \EMLDS).

It is shown in~\cite[Prop.s~7.3,~7.4]{LzDSSuz21} that every infinitesimally Hilbertian metric local structure~$(\mcX,\mssd)$ gives rise to an~\MLDS in the sense of Definition~\ref{d:EMLDSConfig2}.
As a consequence, the following definition of~$\RCD(K,\infty)$ space coincides with the original one in~\cite{AmbGigSav14b}.

\begin{defs}[$\RCD(K,\infty)$-spaces,~{\cite{AmbGigSav14b}}]
Let~$(X,\mssd,\mssm)$ be an infinitesimally Hilbertian metric measure space, and let~$K\in\R$.
We say that~$(X,\mssd,\mssm)$ is an~$\RCD(K,\infty)$ space if the associated~\MLDS satisfies~$(\EVI_K)$.
\end{defs}

For \MLDS's, i.e.\ when~$\mssd$ is a distance, we have the following characterization.

\begin{thm}[Characterization of $\RCD(K,\infty)$ spaces~{\cite{AmbGigSav15}}]\label{t:RCD}
Fix~$K\in\R$.
Let~$(\mcX,\cdc,\mssd)$ be an~\MLDS satisfying~$(\SL{\mssd}{\mssm})$ and assume that~$\tparen{\EE{X}{\mssm},\dom{\EE{X}{\mssm}}}=\tparen{\Ch[\mssd,\mssm],\dom{\Ch[\mssd,\mssm]}}$.
Then, the following are equivalent:
\begin{enumerate}[$(a)$]
\item $(\mcX,\cdc)$ satisfies~$\BE(K,\infty)$;
\item $(\mcX,\cdc,\mssd)$ satisfies~$(\kWC{\mssd})$ with rate~$e^{-Kt}$;
\item[$(b')$] $(\mcX,\cdc,\mssd)$ satisfies~$(\sgWC{\mssd})$ with rate~$e^{-Kt}$;
\item $(X,\mssd,\mssm)$ is an $\RCD(K,\infty)$ space.
\end{enumerate}
If any of the above holds, then
\begin{enumerate}[$(a)$]\setcounter{enumi}{3}
\item $(\mcX,\cdc,\mssd)$ satisfies~\ref{ass:LogH}.
\end{enumerate}
\end{thm}

\begin{rem}[Log-Harnack and $\RCD$]\label{r:LogH}
It is shown in~\cite[Cor.~1.5]{KopStu21} that, whenever~$(X,\mssd,\mssm)$ is an $\RCD(-L,\infty)$ space for some~$L>0$, then it is an~$\RCD(K,\infty)$ space, for any $K\in \R$, if and only if it satisfies~\ref{ass:LogH}.
\end{rem}

\begin{prop}[Properties of $\RCD(K,\infty)$ spaces]\label{p:PropertiesRCD(Kinfty)Config2}
Fix~$K\in\R$. Let~$(X,\mssd,\mssm)$ be an~$\RCD(K,\infty)$ spacewith Cheeger energy~$\tparen{\Ch[\mssd,\mssm],\dom{\Ch[\mssd,\mssm]}}$.
Then,
\begin{enumerate}[$(i)$]
\item $\tparen{\Ch[\mssd,\mssm],\dom{\Ch[\mssd,\mssm]}}$ is quadratic, and thus a Dirichlet form, additionally admitting square field operator~$\SF{X}{\mssm}=\slo[w,\mssd]{\emparg}$, see \cite[Thm.~4.18(iv)]{AmbGigSav14b}, and satisfying~$(\Rad{\mssd}{\mssm})$ by definition;

\item\label{i:p:PropertiesRCD(Kinfty)Config2:2} $\tparen{\Ch[\mssd,\mssm],\dom{\Ch[\mssd,\mssm]}}$ is quasi-regular, see~\cite[Lem.~6.7]{AmbGigSav14b} or~\cite[Thm.~4.1]{Sav14};

\item\label{i:p:PropertiesRCD(Kinfty)Config2:3} $\tparen{\Ch[\mssd,\mssm],\dom{\Ch[\mssd,\mssm]}}$ is irreducible (consequence of~\ref{i:p:PropertiesRCD(Kinfty)Config2:7} below together with Prop.~\ref{p:IrreducibilityBase});

\item\label{i:p:PropertiesRCD(Kinfty)Config2:7} $(\mcX,\SF{X}{\mssm},\mssd)$ satisfies~$(\SL{\mssm}{\mssd})$, see~\cite[Thm.~7.2]{AmbGigMonRaj12} after~\cite[Thm.~6.2]{AmbGigSav14b};

\item\label{i:p:PropertiesRCD(Kinfty)Config2:8} the intrinsic distance~$\mssd_\mssm$ of~$\tparen{\Ch[\mssd,\mssm],\dom{\Ch[\mssd,\mssm]}}$ coincides with~$\mssd$, see~\cite[Thm.~7.4]{AmbGigMonRaj12} after~\cite[Thm.~6.10]{AmbGigSav14b};

\item\label{i:p:PropertiesRCD(Kinfty)Config2:4} $\tparen{\Ch[\mssd,\mssm],\dom{\Ch[\mssd,\mssm]}}$ is conservative, see~\cite[Thm.~4]{Stu94}, applicable by~\ref{i:p:PropertiesRCD(Kinfty)Config2:8}.
In particular,~$\hh{X}{\mssm}_\bullet$ satisfies~\ref{ass:SC}.

\item\label{i:p:PropertiesRCD(Kinfty)Config2:5} $\TT{X}{\mssm}_t\colon L^\infty(\mssm)\to \bLip(\mssd)$ for~$t>0$, see~\cite[Thm.~7.3]{AmbGigMonRaj12}\footnote{In \cite{AmbGigMonRaj12} it is proved that~$\TT{X}{\mssm}_t\colon L^2(\mssm)\cap L^\infty(\mssm)\to \bLip(\mssd)$.
The case of $L^\infty(\mssm)$ follows by a standard localization argument.}. 
In particular,~$\hh{X}{\mssm}_\bullet$ satisfies~\ref{ass:wFe}.
Furthermore, for every~$f\in L^2(\mssm)\cap L^\infty(\mssm)$, the $\mssd$-Lipschitz $\mssm$-representative of~$\widerep{\TT{X}{\mssm}_t f}$ of~$\TT{X}{\mssm}_t f$ satisfies
\begin{align}\label{eq:p:PropertiesRCD(Kinfty)Config2:5}
\sqrt{2\, I_{2K}(t)}\,\, \bLi[\mssd]{\widerep{\TT{X}{\mssm}_t f}}\leq \norm{f}_{L^\infty(\mssm)} \fstop
\end{align}
In particular,~$\hh{X}{\mssm}_\bullet$ satisfies~\ref{ass:Fe};
\item\label{i:p:PropertiesRCD(Kinfty)Config2:6} $(\kWC{\mssd})\iff (\sgWC{\mssd})$, consequence of~\eqref{eq:p:PropertiesRCD(Kinfty)Config2:5} and~\cite[Prop.~3.2(i)--(iii)]{AmbGigSav15}, see Remark~\ref{r:kWCsgWC}.
\end{enumerate}
\end{prop}

\subsection{Configuration spaces}\label{ss:ConfigCurvature}
In this section we prove the characterization of synthetic Ricci-curvature lower bounds for~$\dUpsilon$ anticipated in the beginning of~\S\ref{s:RicciBounds}.

\subsubsection{The distance-Sobolev--to--Lipschitz property via maximal functions}\label{sss:StoL}
Let us start by proving the distance-Sobolev-to-Lipschitz property~$(\dSL{\mssd_\dUpsilon}{\PP})$ for $\tparen{\dUpsilon,\SF{\dUpsilon}{\PP}, \mssd_\dUpsilon}$.

\begin{thm}[Distance-Sobolev-to-Lipschitz]\label{t:dSLConfig2}
Let~$(\mcX,\cdc,\mssd)$ be an \MLDS satisfying \SCF and $(\SL{\mssm}{\mssd})$.
Then,
\begin{align}\label{eq:dSLConfig2}
\mssd_\dUpsilon\geq \mssd_{\PP} \fstop
\end{align}
\end{thm}

As a consequence of Theorem~\ref{t:dSLConfig2} and of~\eqref{eq:EquivalenceRadStoLConfig2}, we immediately obtain the continuous-Sobolev-to-Lipschitz property.

\begin{cor}[Continuous-Sobolev-to-Lipschitz]\label{c:cSLUpsilon-dSL}
Let~$(\mcX,\cdc,\mssd)$ be an \MLDS satisfying \SCF and~$(\SL{\mssm}{\mssd})$.
Then, the \EMLDS $\tparen{\dUpsilon,\SF{\dUpsilon}{\PP},\mssd_\dUpsilon}$ satisfies~$(\cSL{\T_\mrmv(\Ed)}{\PP}{\mssd_\dUpsilon})$.
\end{cor}

Together with the reverse inequality, (consequence of Theorem~\ref{t:GeometricProperties}\ref{i:t:GeometricProperties:1} below and~\eqref{eq:EquivalenceRadStoLConfig2}), Theorem~\ref{t:dSLConfig2} also implies the following identification of the intrinsic distance~$\mssd_{\PP}$ with the $L^2$-transport\-ation distance~$\mssd_\dUpsilon$.
The irreducibility assumption is verified in light of Proposition~\ref{p:IrreducibilityBase}.

\begin{cor}[Identification of distances]\label{c:IdentifDist}
Let~$(\mcX,\cdc,\mssd)$ be an~\MLDS satisfying~\SCF, $(\Rad{\mssd}{\mssm})$, and~$(\SL{\mssm}{\mssd})$. Then,
\begin{align}\label{eq:IdentificationDistanceConfig2}
\mssd_\dUpsilon=\mssd_\PP\fstop
\end{align}
\end{cor}

\begin{rem}[Comparison with~{\cite{LzDSSuz21}}]
The inequality~\eqref{eq:dSLConfig2} and the identification~\eqref{eq:IdentificationDistanceConfig2} were obtained for a far more general class of reference measures~$\QP$ in~\cite[Cor.~\ref{c:dSLUpsilon-cSL}, Thm.~\ref{t:StoL2}]{LzDSSuz21}
under a different set of assumptions for the base~\MLDS.
In particular,~\cite[\S\ref{sss:cSL}]{LzDSSuz21} relies on~\cite[Ass.~\ref{ass:cSLTensor}]{LzDSSuz21}, a local form of the continuous-Sobolev-to-Lipschitz property on all product spaces~$(\mcX^\otym{n}, \cdc^\tym{n})$.
This assumption is verified on every Riemannian manifold, and on $\MLDS$'s satisfying both~$\RCD^*(K,N)$ \emph{and}~$\CAT(0)$.

Here, we provide a different proof ---~only for the case~$\QP=\PP$~--- under much less restrictive assumption for the base space, which in particular allows us to prove~\eqref{eq:dSLConfig2} for~$\RCD^*(K,N)$ spaces (i.e., removing the $\CAT(0)$ assumption).
This is essential to our results on Ricci-curvature lower bounds in the rest of \S\ref{ss:ConfigCurvature}.
\end{rem}

We now turn to the continuous-Sobolev-to-Lipschitz property on Poisson configuration spaces.

\paragraph{Some preliminaries}
We recall the main results in~\cite[\S\S5,6]{LzDSSuz21}, specialized to Poisson measures.

\begin{thm}[Geometric properties]\label{t:GeometricProperties}
Let~$(\mcX,\cdc,\mssd)$ be an \MLDS.
Then, $(\dUpsilon,\SF{\dUpsilon}{\PP},\mssd_\dUpsilon)$ is an \EMLDS.
If, additionally,~$(\mcX,\cdc,\mssd)$ satisfies the Rademacher property~$(\Rad{\mssd}{\mssm})$ and~$\Dz\subset \bLip(\T,\mssd)$, then the following holds:
\begin{enumerate}[$(i)$]
\item\label{i:t:GeometricProperties:1}~$(\dUpsilon, \SF{\dUpsilon}{\PP}, \mssd_\dUpsilon)$ satisfies~$(\Rad{\mssd_\dUpsilon}{\PP})$,~\cite[Thm.~5.2]{LzDSSuz21};
\item\label{i:t:GeometricProperties:2}~$\tparen{\EE{\dUpsilon}{\PP},\dom{\EE{\dUpsilon}{\PP}}}$ is a $\T_\mrmv(\msE)$-quasi-regular strongly local recurrent (conservative) Dirichlet form on~$L^2(\PP)$,~\cite[Cor.~6.3]{LzDSSuz21};

\item\label{i:t:GeometricProperties:3}~$\tparen{\EE{\dUpsilon}{\PP},\dom{\EE{\dUpsilon}{\PP}}}$ is properly associated with a Markov diffusion process with state space~$\dUpsilon$ and invariant measure~$\PP$ (consequence of the standard theory of Dirichlet forms).
\end{enumerate}

\noindent If additionally~$(\mcX,\cdc,\mssd)$ satisfies the \emph{tensorization assumption}~\cite[Ass.~4.22]{LzDSSuz21}, then, additionally,
\begin{enumerate}[$(i)$]\setcounter{enumi}{3}
\item\label{i:t:GeometricProperties:4} the Cheeger energy~$\Ch[\mssd_\dUpsilon,\PP]$ of the extended-metric measure space~$(\dUpsilon,\mssd_\dUpsilon,\PP)$ is quadratic and it coincides with the Dirichlet form~$\EE{\dUpsilon}{\PP}$, \cite[Thm.~5.8]{LzDSSuz21}, viz.
\begin{align*}
\tparen{\Ch[\mssd_\dUpsilon,\PP],\dom{\Ch[\mssd_\dUpsilon,\PP]}}=\tparen{\EE{\dUpsilon}{\PP},\dom{\EE{\dUpsilon}{\PP}}} \fstop
\end{align*}
\end{enumerate}
\end{thm}

Our proof of Theorem~\ref{t:dSLConfig2} will be a byproduct of the study of maximal functions on~$\dUpsilon$ by means of the infinite-product heat semigroup.

\begin{lem}\label{l:AuxStoL} Let~$(\mcX,\mssd)$ be a metric local structure in the sense of Definition~\ref{d:EMLSConfig2}. Let~$\mbfx$, $\mbfy\in X^{\tym\infty}$ be so that~$\mssd_\tym{\infty}(\mbfx,\mbfy)<\infty$, and set~$A_{n,p}\eqdef \overline{B}^\mssd_{2^{-n}}(y_p)$ and
\begin{align*}
\mbfA_n\eqdef A_{n,1}\times \cdots \times A_{n,n}\times X\times X\times  \cdots \in \CylSet{\Ed} \fstop
\end{align*}
Then, there exists~$\nlim \mssd_{\tym\infty}(\mbfx,\mbfA_n)=\mssd_{\tym\infty}(\mbfx,\mbfy)$.
\begin{proof} The sought limit exists since the sequence~$\mbfA_n$ is monotone decreasing.
The inequality~`$\leq$' is trivial. Since~$\mssd_{\tym{\infty}}(\mbfx,\mbfy)<\infty$, for every~$\eps>0$ there exists~$m=m_\eps$ so that
\begin{equation*}
\abs{\mssd_{\tym{\infty}}(\mbfx,\mbfy)-\mssd_{\tym{m}}(\mbfx^\asym{m}, \mbfy^\asym{m})}<\eps\fstop
\end{equation*}
Therefore, since
\begin{equation*}
\nlim \mssd_{\tym\infty}(\mbfx,\mbfA_n)\geq \nlim \mssd_{\tym{m}}(\mbfx^\asym{m},\mbfA_n^\asym{m})\comma \qquad m\in\N_1\comma
\end{equation*}
it suffices to show that
\begin{align*}
\nlim \mssd_{\tym{m}}(\mbfx^\asym{m}, \mbfA_n^{\asym{m}})\geq \mssd_{\tym{m}}(\mbfx^\asym{m},\mbfy^\asym{m}) \comma \qquad m\in \N_1 \fstop
\end{align*}
In particular, it suffices to show that $\lim_n \mssd(x, A_n)=\mssd(x,y)$ whenever~$y\in A_n$ and $\nlim \diam A_n=0$. The latter follows by the reverse triangle inequality.
\end{proof}
\end{lem}

\begin{lem}\label{l:Sobolev1}
Let~$(\mcX,\cdc,\mssd)$ be an \MLDS satisfying~\SCF and~$(\SL{\mssm}{\mssd})$.
Then, for every~$\mbfA\in \CylSet{\Ed}\cap \T^\tym{\infty}$, for~$\PP$-a.e.~$\gamma$ and some~$\mbfx \in \Lb^{-1}(\gamma)$,
\begin{align}\label{eq:l:Sobolev1:0}
\mssd_{\tym{\infty}}(\mbfx, \mbfA)\geq \hr{\PP,\Lb(\tilde\mbfA)}(\gamma)\geq \mssd^\reg_{\PP,\Lb(\tilde\mbfA)}(\gamma) \comma
\end{align}
where~$\hr{}$ denotes the maximal function in Definition~\ref{d:MaximalFunctionConfig2} and~$\mssd^\reg$ denotes the regularized intrinsic point-to-set distance in Definition~\ref{d:RegIntrinsicD}.

\begin{proof}
The second inequality is Proposition~\ref{p:RegIntrinsic}\ref{i:p:RegIntrinsic:4}, thus it suffices to show the first one.

\smallskip

As a consequence of Proposition~\ref{p:IrreducibilityBase}, the space~$(X,\cdc,\mssd)$ is irreducible.

Firstly, let us note that since~$\mbfA$ is $\T^\tym{\infty}$-open,~$A_i$ is $\T$-open for each~$i \leq n$, and that~$0<\mssm A_i<\infty$ for each~$i \leq n$.
Secondly, $(\SL{\mssm}{\mssd})$ implies that~$\hr{\mssm, A_i}\wedge r$ has a (non-relabeled) representative
\begin{align*}
\hr{\mssm, A_i}\wedge r\in \bLip(X,\mssd) = \bLip(X,\mssd,\T)\comma \qquad r\geq 0\comma\qquad i\leq n\fstop
\end{align*}
By continuity of~$\hr{\mssm, A_i}\wedge r$ we thus have~$\hr{\mssm,A_i}\equiv 0$ everywhere (rather than only $\mssm$-a.e.) on~$A_i$.
Since~$\hr{\mssm, A_i}\wedge r$ is $1$-Lipschitz, we conclude that~$\hr{\mssm, A_i}\wedge r \leq \mssd(\emparg, A_i) \wedge r$, hence, letting~$r$ to infinity,
\begin{align}\label{eq:l:Sobolev1:0.5}
\hr{\mssm, A_i} \leq \mssd(\emparg, A_i) \comma \qquad i\leq n \fstop
\end{align}

Since $\mssd(\emparg, A_i)$ is everywhere finite by finiteness of~$\mssd$, we conclude by~\cite[Thm.~5.2]{AriHin05} that, for every probability measure~$\nu\sim \mssm$,
\begin{align}\label{eq:l:Sobolev1:1}
\nu\text{-}\lim_{t\rar 0} \tparen{-2t \log\hh{X}{\mssm}_t(\emparg, A_i) }= \hr{\mssm, A_i}^2\comma \qquad i\leq n \fstop
\end{align}

Now, by~\ref{ass:SC} we have that~$\hh{X}{\mssm}_t(x, X)=1$ for each~$t>0$ and~$x\in X$, hence
\begin{align}\label{eq:l:dInfty:1}
-2t\log \mssh_t^{\tym{\infty}}(\emparg, \mbfA)=-2t \log \prod_{i=1}^n\hh{X}{\mssm}_t(\emparg_i, A_i)=- \sum_{i=1}^n 2t \log\hh{X}{\mssm}_t(\emparg_i, A_i) \fstop
\end{align}
By tensorization of convergence in probability and~\eqref{eq:l:Sobolev1:1}
\begin{align*}
\nu^{\otym{n}}\text{-}\lim_{t\rar 0} \sum_{i=1}^n -2t \log \hh{X}{\mssm}_t(\emparg_i, A_i)= \sum_{i=1}^n \hr{\mssm,A_i}^2
\end{align*}
for every Borel probability measure~$\nu\sim \mssm$.
In fact, since~$\nu^{\otym{n}}$ is a finite measure, by standard arguments there exists a sequence~$(t_k)_k$, with~$t_k\rar 0$ as $k\rar\infty$, so that
\begin{align}\label{eq:l:dInfty:2}
\nu^{\asym{n}}\text{-}\klim \sum_{i=1}^n -2t \log \hh{X}{\mssm}_{t_k}(\emparg_i, A_i)= \sum_{i=1}^n \hr{\mssm,A_i}^2
\end{align}
for every probability~$\nu^{\asym{n}}\ll \mssm^{\otym{n}}$.
In particular, we may choose~$\nu^{\asym{n}}=\PP^\asym{n}$ for some fixed labeling map~$\lb$.
Let now
\begin{align*}
Y^\asym{\infty}_{\eps, k}\eqdef& \set{\mbfx=\seq{x_i}_i^\infty\in\mbfX^{\asym{\infty}}_\locfin(\Ed): \abs{-2t_k\log\mssh_{t_k}^{\tym{\infty}}(\mbfx,\mbfA)-\sum_{i=1}^n\hr{\mssm,A_i}(x_i)^2}>\eps} \comma
\\
Y^\asym{n}_{\eps, k}\eqdef& \set{\seq{x_i}_{i\leq n}\in X^{\tym{n}} : \abs{-\sum_{i=1}^n 2t_k\log\mssh_{t_k}^{\tym{\infty}}(x_i,A_i)-\sum_{i=1}^n\hr{\mssm,A_i}(x_i)^2}>\eps} \fstop
\end{align*}
By definition of~$\mbfA$, one has~$Y^\asym{\infty}_{\eps,k}\subset (\tr^n)^{-1}\tparen{Y^\asym{n}_{\eps,k}}$. Thus, by~\eqref{eq:l:dInfty:2} with~$\nu^{\asym{n}}=\PP^\asym{n}$,
\begin{align*}
\klim \PP^\asym{\infty} Y^\asym{\infty}_{\eps, k}\leq& \klim \PP^\asym{n} Y^\asym{n}_{\eps,k}=0\fstop
\end{align*}

The above inequality yields the convergence
\begin{align}\label{eq:l:dInfty:3}
\PP^{\asym{\infty}}\text{-}\klim \tparen{-2t_k \log \mssh_{t_k}^\tym{\infty}(\emparg, \mbfA)}= &\sum_{i=1}^n \hr{\mssm,A_i}^2\leq \sum_{i=1}^n \mssd(\emparg, A_i)^2=\mssd_{\tym{\infty}}(\emparg,\mbfA)^2 \quad \as{\PP^\asym{\infty}}\comma
\end{align}
where the inequality is a consequence of~\eqref{eq:l:Sobolev1:0.5}, and the second equality holds by definition of~$\mssd_\tym{\infty}$.

Now, set~$u_k\eqdef -2t_k \log \mssh^\dUpsilon_{t_k}\tparen{\emparg,\Lb(\tilde\mbfA)}$, and let~$\mbfY \subset \mbfX^\asym{\infty}_\locfin(\Ed)$ be a set of full $\PP^\asym{\infty}$-measure so that~\eqref{eq:l:dInfty:3} holds for every~$\mbfx\in\mbfY$.
For each~$k\in\N_1$ let now~$\Theta_{t_k}\tbraket{\car_{\Lb(\tilde\mbfA)}}$ be as in~\eqref{eq:IdentificationSemigroup}, set~$\Omega\eqdef \bigcap_{k=1}^\infty \Theta_{t_k}\tbraket{\car_{\Lb(\tilde\mbfA)}}$, and note that~$\PP\Omega=1$.
For every~$\gamma\in\Omega\cap \Lb(\mbfY)$ and each~$\mbfx\in\Lb^{-1}(\gamma)\cap \mbfY$
\begin{align}
\nonumber
\infty>\mssd_{\tym{\infty}}(\mbfx, \mbfA)^2\geq &\ \tparen{\PP^{\asym{\infty}}\text{-}\klim \tparen{-2 t_k \log \mssh_{t_k}^{\tym{\infty}}(\emparg,\mbfA)}}(\mbfx)
\\
\label{eq:l:dInfty:5}
=&\ \tparen{\PP^{\asym{\infty}}\text{-}\klim \tparen{-2 t_k \log \mssh_{t_k}^{\tym{\infty}}(\emparg,\tilde \mbfA)}}(\mbfx)
\\
\label{eq:l:dInfty:6}
=&\ \tparen{\PP\text{-}\klim u_k}(\gamma)
\\
\label{eq:l:dInfty:7}
\geq&\ \tparen{\PP\text{-}\klim u_k \car_{\set{u_k<\infty}}}(\gamma)\comma
\end{align}
where~\eqref{eq:l:dInfty:5} holds by Corollary~\ref{c:SCConfig1}, and~\eqref{eq:l:dInfty:6} holds by~\eqref{eq:IdentificationSemigroup}.
Applying~\cite[Thm.~5.2]{AriHin05} to the form~$\tparen{\EE{\dUpsilon}{\PP},\dom{\EE{\dUpsilon}{\PP}}}$, we have that
\begin{align}\label{eq:l:dInfty:8}
\PP\text{-}\klim u_k\car_{\ttset{u_k <\infty}} = \hr{\PP, \Lb(\tilde\mbfA)}(\gamma)\cdot \car_{\ttset{\hr{\PP,\Lb(\tilde\mbfA)}<\infty}}\quad  \as{\PP} \ \text{ on } \Omega\cap \Lb(\mbfY) \fstop
\end{align}
Combining~\eqref{eq:l:dInfty:7} and~\eqref{eq:l:dInfty:8} yields
\begin{align*}
\mssd_\tym{\infty}(\mbfx,\mbfA)\geq  \hr{\PP, \Lb(\tilde\mbfA)}(\gamma)\cdot \car_{\ttset{\hr{\PP,\Lb(\tilde\mbfA)}<\infty}}(\gamma)\comma
\end{align*}
for $\PP$-a.e.~$\gamma\in \Omega\cap \Lb(\mbfY)$ and each~$ \mbfx\in \Lb^{-1}(\gamma)\cap \mbfY$, and the proof is concluded if we show that
\begin{align}
\PP \tset{\hr{\PP,\Lb(\tilde\mbfA)}<\infty}=1\fstop
\end{align}
To this end, recall that~$\PP\Omega=1$ by construction. 
By~\cite[Prop.~5.1]{AriHin05} we have that 
\[
\tset{\TT{\dUpsilon}{\PP}_t\car_{\Lb(\tilde\mbfA)}=0}=\tset{\hr{\PP,\Lb(\tilde\mbfA)}=\infty} \as{\PP}\fstop
\]
Thus, by the equivalence in~\cite[Prop.~5.1]{AriHin05}, it suffices to show that, for some $t>0$,
\begin{align*}
\mssh^\dUpsilon_t\tparen{\emparg, \Lb(\tilde\mbfA)}=\TT{\dUpsilon}{\PP}_t\car_{\Lb(\tilde\mbfA)}>0\comma \quad \as{\PP}\fstop
\end{align*}
Respectively by~\eqref{eq:IdentificationSemigroup},~\ref{ass:SC}, and irreducibility combined with~$\mssm A_i>0$, we see that
\begin{align*}
\mssh^\dUpsilon_t\tparen{\gamma,\Lb(\tilde\mbfA)}= \mssh^\tym{\infty}_t\tparen{\lb(\gamma),\mbfA}=\prod_{i=1}^n \hh{X}{\mssm}_t \tparen{\lb(\gamma)_i, A_i}>0\comma \qquad \gamma\in \Omega\cap \Theta_t\tbraket{\car_{\Lb(\tilde\mbfA)}}\comma
\end{align*}
which concludes the proof.
\end{proof}
\end{lem}

We conclude this section by showing the distance-Sobolev-to-Lipschitz property for~$\dUpsilon$.

\begin{proof}[Proof of Theorem~\ref{t:dSLConfig2}]
Note that the space~$(X,\cdc,\mssd)$ is irreducible as a consequence of Proposition~\ref{p:IrreducibilityBase}.

We start by showing that there exists~$\Omega\in\A_\mrmv(\Ed)$ with~$\PP\Omega=1$ so that
\begin{align}\label{eq:dSLPoissonAE}
\mssd_\dUpsilon(\gamma,\eta)\geq \mssd_{\PP}(\gamma,\eta) \comma \qquad \gamma\in\Omega\comma \eta\in\dUpsilon \fstop
\end{align}
Fix~$\eta\in \dUpsilon$,~$\mbfy\in \Lb^{-1}(\eta)$, and let~$\seq{\mbfA_n}_n$ be a sequence of cylinder sets converging to~$\set{\mbfy}$ as in Lemma~\ref{l:AuxStoL}.
For~$n\in\N_1$ let~$\Omega[\mbfA_n]\in\A_\mrmv(\msE)$ be a set of full~$\PP$-measure such that~\eqref{eq:l:Sobolev1:0} holds everywhere on~$\Omega[\mbfA_n]$.
Further set~$\Omega\eqdef \bigcap_n \Omega[\mbfA_n]$, and note that~$\PP\Omega=1$.

Now, let~$\gamma\in\Omega$. If~$\mssd_\dUpsilon(\gamma,\eta)=\infty$ the assertion is trivial.
Thus, without loss of generality,~$\mssd_\dUpsilon(\gamma,\eta)<\infty$.
Set~$\mbfx\eqdef \lb_\mbfy(\gamma)$ where~$\lb_\mbfy$ is the labeling map, radially isometric around~$\eta$, constructed in \cite[Prop.~\ref{p:PoissonLabeling}\ref{i:p:PoissonLabeling:2}]{LzDSSuz21}.
Then,
\begin{align*}
\mssd_\dUpsilon(\gamma,\eta)=\mssd_{\tym{\infty}}(\mbfx,\mbfy)= \nliminf \mssd_{\tym{\infty}}(\mbfx,\mbfA_n)\geq \nliminf \mssd^\reg_{\PP, \Lb(\tilde\mbfA_n)}(\gamma)=\mssd_\PP\tparen{\gamma, \eta}\comma
\end{align*}
where the first equality holds by definition of~$\mbfy$ since~$\lb_\mbfx$ is distance preserving, the second equality is Lemma~\ref{l:AuxStoL}, the inequality holds by Lemma~\ref{l:Sobolev1}, and the last equality is Proposition~\ref{p:RegIntrinsic}\ref{i:p:RegIntrinsic:2}.

It remains to show how to extend~\eqref{eq:dSLPoissonAE} to~$\dUpsilon^\tym{2}$.
To this end, fix~$\gamma,\eta\in\dUpsilon$, and let~$\seq{\gamma_n}_n\subset \Omega$ be $\Ed$-vaguely converging to~$\gamma$.
Fix~$x_0\in X$ and, for each~$r>0$, set
\begin{align*}
\Lambda_{\eta,r}\eqdef \set{\eta\in\dUpsilon: \eta_{B_r(x_0)}=\gamma_{B_r(x_0)}}\fstop
\end{align*}
Then,
\begin{align*}
\mssd_\PP(\gamma,\eta)=&\ \mssd^\reg_{\PP,\set{\eta}}(\gamma)= \lim_{r\rar \infty} \mssd^\reg_{\PP,\Lambda_{\eta,r}}(\gamma)\leq  \lim_{r\rar \infty} \nliminf \mssd^\reg_{\PP,\Lambda_{\eta,r}}(\gamma_n)
\end{align*}
respectively by Proposition~\ref{p:RegIntrinsic}\ref{i:p:RegIntrinsic:1},~\ref{i:p:RegIntrinsic:2}, and~\ref{i:p:RegIntrinsic:3}.
We may continue the above chain of inequalities with
\begin{align*}
\mssd_\PP(\gamma,\eta)\leq &\ \lim_{r\rar \infty} \nliminf \mssd^\reg_{\PP,\Lambda_{\eta,r}}(\gamma_n) \leq \lim_{r\rar \infty} \nliminf \mssd_\PP(\gamma_n,\Lambda_{\eta,r})
\\
\leq&\ \lim_{r\rar \infty}\nliminf \mssd_\dUpsilon(\gamma_n, \Lambda_{\eta,r}) = \lim_{r\rar \infty} \mssd_\dUpsilon(\gamma, \Lambda_{\eta,r})= \mssd_\dUpsilon(\gamma,\eta)\comma
\end{align*}
where the second inequality holds by the second assertion in Proposition~\ref{p:RegIntrinsic}\ref{i:p:RegIntrinsic:1}, the third inequality holds by~\eqref{eq:dSLPoissonAE}, and the last two equalities respectively hold by~\cite[Prop.~4.29$(iii)$ and $(iv)$]{LzDSSuz21}.
\end{proof}

\subsubsection{Logarithmic-Harnack and entropy-cost inequalities}
The equivalence between the logarithmic Harnack inequality for~$\hh{X}{\mssm}_\bullet$ on a base space~$X$ and the logarithmic Harnack inequality for~$\mssh^\dUpsilon_\bullet$ on the corresponding configuration space~$\dUpsilon$ was shown in full generality by C.-S.~Deng in~\cite{Den14}.
Together with the identification of~$\hh^\dUpsilon_\bullet$ with the heat kernel~$\hh{\dUpsilon}{\PP}_\bullet$ representing~$\TT{\dUpsilon}{\PP}_\bullet$ (Prop.~\ref{p:KonLytRoe}), this show the following equivalence result.

\begin{thm}[Equivalence: Log-Harnack inequality]\label{t:Deng}
Let~$(\mcX,\cdc,\mssd)$ be an~\MLDS satisfying\linebreak \SCF, and~$K\in\R$.
Then, the following are equivalent:
\begin{enumerate}[$(a)$]
\item $(\mcX,\cdc,\mssd)$ satisfies~\ref{ass:LogH} with rate~$K$;
\item $\tparen{\dUpsilon^\sym{\infty},\SF{\dUpsilon}{\PP},\mssd_\dUpsilon}$ satisfies~\ref{ass:LogH} with rate~$K$.
\end{enumerate}

\begin{proof}
Choosing~$\phi(x,y)\eqdef\mssd(x,y)^2$ and~$C\eqdef \tparen{4I_{2K}(t)}^{-1}$ in~\cite[Thm.~3.3]{Den14}, the proofs of~\cite[Thm.s~3.1, 3.3]{Den14} carry over \emph{verbatim} to our setting, with the exception of the approximation argument in~\cite[proof of Thm.~3.1, part (c)]{Den14}.
In fact, the stated approximation holds only~$\PP$-a.e., but this is sufficient for all our purposes here.
\end{proof}
\end{thm}

In particular, by the very same argument for the analogous statement on manifolds in~\cite[Lem.~5.3]{ErbHue15}, the validity of~\ref{ass:LogH} for~$\dUpsilon$ implies the following \emph{entropy cost-inequality}
\begin{align}\label{eq:EntropyCost}
\Ent_\PP(\TT{\dUpsilon}{\PP}_t\mu)\leq \Ent_\PP(\nu)+\frac{W_{2,\mssd_\dUpsilon}(\mu,\nu)^2}{4 I_{2K}(t)}\comma \qquad \mu\in\msP^\PP(\dUpsilon)\comma \nu\in\dom{\Ent_\PP}\comma\quad t>0\fstop
\end{align}

\subsubsection{Wasserstein contractivity of the heat flow}
In order to state our main result for this section, let us first collect the necessary assumptions, strengthening~\SCF.

\begin{defs}[Quantitative~\SCF]\label{d:qSCF}
We say that an~\MLDS~$(\mcX,\cdc,\mssd)$ satisfies~\qSCF if it satisfies~\SCF and additionally
\begin{align}\tag*{$(\mathsf{qF_\Ed})_{\ref{d:qSCF}}$}\label{ass:qF}
\int_X\mssd(x,y)^2\, \hh{X}{\mssm}_t(x,\diff y)<\infty \comma \qquad x\in X\comma\quad t>0\fstop
\end{align}
\end{defs}

\begin{rem}\label{r:qSCF}
If~$(\mcX,\cdc,\mssd)$ satisfies~\qSCF, then~$\hh{X}{\mssm}(x,\diff\emparg)\in\msP_2(X)$ for every~$x\in X$.
Indeed, since~$\mssd$ is a distance (as opposed to: an extended distance), we have that $W_{2,\mssd}(\delta_{x_0},\delta_x)=\mssd(x_0,x)<\infty$ for every fixed~$x_0,x\in X$, and therefore~\ref{ass:qF} is equivalent to the requirement that
\begin{align}\label{eq:r:qSCF}
\int_X\mssd(x_0,y)^2\, \hh{X}{\mssm}_t(x,\diff y)<\infty \comma \qquad x\in X\comma \quad t>0\fstop
\end{align}
Even though we shall not need~\qSCF for \EMLDS's, we find more insightful to state it in its form above rather than as in~\eqref{eq:r:qSCF}, since on an \EMLDS~\eqref{eq:r:qSCF} should not be expected to hold for arbitrary~$x_0$, but only for~$x_0=x$.
\end{rem}

\begin{thm}[Equivalence: Wasserstein contractivity estimate]\label{t:WC}
Let~$(\mcX,\cdc,\mssd)$ be an~\MLDS satisfying~\qSCF, and~$c(t)$ be as in~\eqref{eq:WCConstant}.
Then, the following are equivalent:
\begin{enumerate}[$(a)$]
\item\label{i:t:WC:1} $(\mcX,\cdc,\mssd, \hh{X}{\mssm}_\bullet)$ satisfies~$(\kWC{\mssd})$ with rate~$c(t)$,
\item\label{i:t:WC:2} $\tparen{\dUpsilon, \SF{\dUpsilon}{\PP},\mssd_\dUpsilon,\mssh^\dUpsilon_\bullet}$ with~$\mssh^\dUpsilon_\bullet$ defined in~\eqref{eq:d:Heat:1} satisfies~$(\kWC{\mssd_\dUpsilon})$ with rate~$c(t)$.
\end{enumerate}

\begin{proof}
Throughout the proof, let~$t>0$ be fixed.

\noindent\ref{i:t:WC:1}$\implies$\ref{i:t:WC:2}.
It suffices to show the assertion when~$W_{2,\mssd_\dUpsilon}(\mu_0,\mu_1)<\infty$.

Since~$\mssd_\dUpsilon(\gamma,\eta)<\infty$ implies~$\gamma X=\eta X$ for every~$\gamma,\eta\in\dUpsilon$, we have that
\begin{align}\label{eq:t:WC:1.5}
W_{2,\mssd_\dUpsilon}(\mu_0,\mu_1)^2=\sum_{n=0}^\infty W_{2,\mssd_\dUpsilon}(\mu_0\mrestr{\dUpsilon^\sym{n}}, \mu_1\mrestr{\dUpsilon^\sym{n}})^2 \comma \qquad \mu_0,\mu_1\in\msP(\dUpsilon) \comma
\end{align}
where we extended~$W_{2,\mssd_\dUpsilon}$ to sub-probability measures on~$\dUpsilon$ in the obvious way.
Further recall Definition~\ref{d:HKExtension}.
By~\ref{ass:SC}, we have that~$\mssh^\dUpsilon_\bullet$ leaves~$\dUpsilon^\sym{n}$ invariant for every~$n\in\N_0$.
As a consequence,
\begin{align}\label{eq:t:WC:1.7}
t\mapsto (\mssh^\dUpsilon_t\mu) \dUpsilon^\sym{n} \quad \text{is constant}\comma\qquad n\in\N_0\comma \mu\in\msP(\dUpsilon) \fstop
\end{align}
Combining~\eqref{eq:t:WC:1.7} with~\eqref{eq:t:WC:1.5} proves that it suffices to show that
\begin{align*}
W_{2,\mssd_\dUpsilon}(\mssh^\dUpsilon_t\mu_0\mrestr{\dUpsilon^\sym{N}},\mssh^\dUpsilon_t\mu_1\mrestr{\dUpsilon^\sym{N}}) \leq c(t)\, W_{2,\mssd_\dUpsilon}(\mu_0\mrestr{\dUpsilon^\sym{N}},\mu_1\mrestr{\dUpsilon^\sym{N}}) \comma \qquad N\in\overline\N_0\comma \quad t>0\fstop
\end{align*}
We show the latter statement for~$N=\infty$, the assertion for finite~$n$ being similar and simpler, and therefore omitted.

Fix~$\gamma,\eta\in\dUpsilon^\sym{\infty}$ with~$\mssd_\dUpsilon(\gamma,\eta)<\infty$. We claim that
\begin{align}\label{eq:t:WC:1.8}
W_{2,\mssd_\dUpsilon}\tparen{\mssh^\dUpsilon_t(\gamma,\diff\emparg),\mssh^\dUpsilon_t(\eta,\diff\emparg)}\leq c(t)\, \mssd_\dUpsilon(\gamma,\eta) \fstop
\end{align}
Indeed, by \cite[Prop.~\ref{p:PoissonLabeling}\ref{i:p:PoissonLabeling:1}]{LzDSSuz21} there exist~$\mbfx,\mbfy\in\mbfX^\asym{\infty}_\locfin(\msE)$ with
\begin{align}\label{eq:t:WC:1}
\mssd_{\tym{\infty}}(\mbfx,\mbfy)=\mssd_\dUpsilon(\gamma,\eta)\fstop
\end{align}
In view of Remark~\ref{r:qSCF} for every~$i\in\N_1$, there exists~$\kappa_i\in \msP(X^\tym{2})$ an optimal coupling between~$\hh{X}{\mssm}_t(x_i,\diff\emparg)$ and~$\hh{X}{\mssm}_t(y_i,\diff\emparg)$.
Set~$\boldkappa\eqdef \bigotimes_{i=1}^\infty \kappa_i$, and note that~${\Lb^\tym{2}}_\pfwd\boldkappa$ defines a coupling of~$\mssh^\dUpsilon_t(\gamma,\diff\emparg)$ and~$\mssh^\dUpsilon(\eta,\diff\emparg)$, concentrated on~$\dUpsilon^\sym{\infty}$ by Corollary~\ref{c:SCConfig1}.
Thus,
\begin{align}
\label{eq:t:WC:2a}
W_{2,\mssd_\dUpsilon}\tparen{\mssh^\dUpsilon_t(\gamma,\emparg), \mssh^\dUpsilon_t(\eta,\emparg)}^2\leq&\ \int_{{\dUpsilon^\sym{\infty}}^\tym{2}} \mssd_\dUpsilon^2 \diff {\Lb^\tym{2}}_\pfwd \boldkappa
\\
\nonumber
\leq&\ \int_{X^\tym{\infty}} \mssd_{\tym\infty}^2 \diff \boldkappa = \sum_{i=1}^\infty \int_{X^\tym{2}} \mssd^2 \diff \kappa_i 
\\
\nonumber
=&\ \sum_{i=1}^\infty W_{2,\mssd}\tparen{\hh{X}{\mssm}_t(x_i,\diff\emparg),\hh{X}{\mssm}_t(y_i,\diff\emparg)}^2
\\
\label{eq:t:WC:2b}
\leq&\ c(t)^2 \sum_{i=1}^\infty \mssd(x_i,y_i)^2= c(t)^2\, \mssd_{\tym{\infty}}(\mbfx,\mbfy)^2
\\
\nonumber
=&\ c(t)^2\, \mssd_\dUpsilon(\gamma,\eta)^2\fstop
\end{align}
Here:~\eqref{eq:t:WC:2a} holds by definition of~$W_{2,\mssd_\dUpsilon}$ and since~${\Lb^\tym{2}}_\pfwd\boldkappa$ is a coupling between its marginals, \eqref{eq:t:WC:2b}~follows from the assumption of~$(\kWC{\mssd})$ with~$\mu_0=\delta_{x_i}$ and~$\mu_1=\delta_{y_i}$ for every~$i\in \N_1$, and the last equality is~\eqref{eq:t:WC:1}.
The rest of the proof now follows as in~\cite[Thm.~4.9]{ErbHue15}.

\medskip

\noindent\ref{i:t:WC:2}$\implies$\ref{i:t:WC:1}.
Since the Dirac embedding~$\delta\colon (X,\mssd)\to(\dUpsilon,\mssd_\dUpsilon)$ is an isometric embedding, the corresponding push-forward map~$\delta_\pfwd \colon \tparen{\msP(X), W_{2,\mssd}}\to \tparen{\msP(\dUpsilon^\sym{1}), W_{2,\mssd_\dUpsilon}}$ is an isometry of extended metric spaces.
Furthermore,~$\mssh^\dUpsilon_t\delta_\pfwd\mu=\delta_\pfwd(\hh{X}{\mssm}_t\mu)$ for every~$t>0$.
Thus,~$(\kWC{\mssd})$ for~$\mu_0,\mu_1\in\msP(X)$ follows applying the assumption~$(\kWC{\mssd_\dUpsilon})$ to~$\delta_\pfwd \mu_0$,~$\delta_\pfwd \mu_1$.
\end{proof}
\end{thm}

In light of Proposition~\ref{p:PropertiesRCD(Kinfty)Config2}\ref{i:p:PropertiesRCD(Kinfty)Config2:6} and Remark~\ref{r:kWCsgWC}, we have one implication for~$(\sgWC{\mssd_\dUpsilon})$ as well.

\begin{cor}[Semigroup-type Wasserstein contractivity estimate]\label{c:WC}
Let~$(\mcX,\cdc,\mssd)$ be an~\MLDS satisfying~\qSCF, and~$c(t)$ be as in~\eqref{eq:WCConstant}.
If~$(\mcX,\cdc,\mssd, \hh{X}{\mssm}_\bullet)$ satisfies~$(\sgWC{\mssd})$ with rate~$c(t)$, then~$\tparen{\dUpsilon, \SF{\dUpsilon}{\PP},\mssd_\dUpsilon,\mssh^\dUpsilon_\bullet}$ with~$\mssh^\dUpsilon_\bullet$ defined in~\eqref{eq:d:Heat:1} satisfies~$(\sgWC{\mssd_\dUpsilon})$ with rate~$c(t)$.
\end{cor}

\begin{rem}[Comparison with~{\cite{ErbHue15}, cf.\ Rmk.~\ref{r:ComparisonKonLytRoe}}]\label{r:ComparisonErbHue15}
The results in~\cite[Thm.~4.9]{ErbHue15} require the existence of a set~$\Theta\subset \dUpsilon$, with~$\PP\Theta=1$, of \emph{good configurations}, and, as a consequence, volume-growth estimates and Gaussian heat-kernel estimates on the base space, see~\cite[Thm.\ 2.4]{ErbHue15}.

The main difference in our approach is that, thanks to Lemmas~\ref{l:SCProduct},~\ref{l:CoincidenceHK} and Definition~\ref{d:HKExtension}, we can define the probability measure $\mssh^\dUpsilon_t(\gamma,\emparg)$ for every $\gamma \in \dUpsilon^{(\infty)}$.
Therefore, corollary~\ref{c:WC} does not require any such set~$\Theta$.
This considerably enlarges the scope of applications, in particular, to the case that $\mcX$ is an~$\RCD(K,\infty)$ where neither the heat kernel estimates nor the volume growth estimates needed in~\cite{ErbHue15} hold in general.
\end{rem}

\subsubsection{Evolution Variation Inequality}
In this section, we show how to lift the Evolution Variation Inequality on an~$\RCD^*(K,N)$ space to the corresponding configuration space.

\begin{thm}[Evolution Variation Inequality]\label{t:EVI}
Let~$(X,\mssd,\mssm)$ be an $\RCD^*(K,N)$ space, $K\in\R$,~$N\in (2,\infty)$.
Then, $(\dUpsilon,\SF{\dUpsilon}{\PP},\mssd_\dUpsilon)$ satisfies $(\EVI_{K,\mssd_\dUpsilon,\PP})$.
\end{thm}

As a standard consequence of the theorem we further have the $K$-convexity of the entropy.
\begin{cor}
Let~$(X,\mssd,\mssm)$ be an $\RCD^*(K,N)$ space, $K\in\R$,~$N\in (2,\infty)$. Then, the entropy~$\Ent_\PP$ is strongly $K$-convex on~$\tparen{\msP(\dUpsilon), W_{2,\mssd_\dUpsilon}}$, that is, for all~$\mu_0,\mu_1\in\dom{\Ent_\PP}$ satisfying $W_{2,\mssd_\dUpsilon}(\mu_0,\mu_1)<\infty$, and for every geodesic~$\seq{\mu_t}_{t\in [0,1]}$ connecting them, we have that
\begin{align*}
\Ent_\PP(\mu_t)\leq (1-t)\, \Ent_\PP(\mu_0) + t\, \Ent_\PP(\mu_1)-\tfrac{K}{2} t(1-t)\, W_{2,\mssd_\dUpsilon}(\mu_0,\mu_1)^2 \fstop
\end{align*}
\end{cor}

Furthermore, specializing~\cite[Thm.~8.3]{AmbErbSav16} to our setting, we also obtain that~$\TT{\dUpsilon}{\PP}_\bullet$ coincides with the gradient flow of~$\Ent_\PP$ on densities.

\begin{cor}\label{c:GradFlowEnt}
Let~$(X,\mssd,\mssm)$ be an $\RCD^*(K,N)$ space, $K\in\R$,~$N\in (2,\infty)$.
Further let~$\seq{\mu_t}_{t\in [0,\infty)}$ be a curve of probability densities~$\mu_t\eqdef \rho_t\PP$ satisfying~$\norm{\rho_t}_{L^\infty}\in L^\infty_\loc([0,\infty))$.
If~$\mu_t$ is a $W_{2,\mssd_\dUpsilon}$-gradient flow of the entropy~$\Ent_\PP$, then~$\rho_t= \TT{\dUpsilon}{\PP}_t\rho_0$ for every~$t\geq 0$.
\end{cor}

\begin{rem}[Comparison with~\cite{ErbHue15}]\label{r:ComparisonErbHue2}
We point out that the main ideas for all proofs in this section are already found in~\cite{ErbHue15}.
Additionally, most of the proofs in~\cite{ErbHue15} also apply with only minor modifications to our more general setting.
The proof of~$(\EVI_K)$, mostly based on the results of L.~Ambrosio, N.~Gigli, and G.~Savar\'e in~\cite{AmbGigSav15}, relies on the identification of the form~$\EE{\dUpsilon}{\PP}$ with the Cheeger energy~$\Ch[\mssd_\dUpsilon,\PP]$.
For configuration spaces over manifolds with Ricci-curvature lower bounds, this result is claimed in~\cite[Prop.~2.3]{ErbHue15} based on the incorrect assertion that functions in~$\Cyl{\Dz}$ are $\mssd_\dUpsilon$-Lipschitz (see~\cite[Ex.~4.35]{LzDSSuz21}).
In our setting, the above identification of~$\EE{\dUpsilon}{\PP}$ with~$\Ch[\mssd_\dUpsilon,\PP]$ is established by different techniques as one of the main results in~\cite{LzDSSuz21} (see Thm.~\ref{t:GeometricProperties}).
Our proof of $(\EVI_K)$ (Thm.~\ref{t:EVI}) also applies to $\RCD^*(K,N)$ spaces, including in particular weighted manifolds with Ricci-curvature lower bounds, for which the validity of~$(\EVI_K)$ was conjectured in~\cite[Rmk.~1.5]{ErbHue15}.
Indeed, as already noted in~\cite{ErbHue15}, the only missing argument there is the validity of the heat kernel bounds used in~\cite{ErbHue15} to establish the identification of heat kernels in Proposition~\ref{p:KonLytRoe}.
Here, we bypass this problem entirely by establishing Proposition~\ref{p:KonLytRoe} without any kind of heat-kernel bounds, as discussed in Remark~\ref{r:ComparisonKonLytRoe}.
\end{rem}

We start by recalling the following properties of~$\Ent_\PP$ in relation to the heat semigroup acting on measures.

\begin{prop}[Heat-kernel regularization~{\cite[Lem.~5.4]{ErbHue15}}]\label{p:HKRegularization}
Let~$(X,\mssd,\mssm)$ be an\linebreak $\RCD^*(K,N)$ space, $K\in\R$,~$N\in (2,\infty)$.
Further fix~$\mu\in \cl_{W_{2,\mssd_\dUpsilon}}\tparen{\dom{\Ent_\PP}}$ and~$t>0$, and set $\mu_t\eqdef \mssh^\dUpsilon_t\mu$.
Then,
\begin{enumerate*}[$(i)$]
\item $\mu_t\in\dom{\Ent_\PP}$;
\item $W_{2,\mssd_\dUpsilon}(\mu_t,\mu)<\infty$;
\item $\lim_{t\downarrow 0} W_{2,\mssd_\dUpsilon}(\mu_t,\mu)=0$.
\end{enumerate*}
\begin{proof}
We have that~$\tparen{\dUpsilon^\sym{\infty}, \SF{\dUpsilon}{\PP},\mssd_\dUpsilon,\mssh^\dUpsilon_\bullet}$ satisfies~$(\kWC{\mssd_\dUpsilon})$ combining Proposition~\ref{t:RCD} with Theorem~\ref{t:WC}.
Further note that: $\Ch[\mssd,\mssm]$ satisfies~$(\Rad{\mssd}{\mssm})$ by construction;
 every $\RCD^*(K,N)$ space satifies the tensorization assumption~\cite[Ass.~4.22]{LzDSSuz21} as noted in~\cite[Prop.~7.5$(v)$]{LzDSSuz21}.
As a consequence, by~\cite[Thm.~5.8]{LzDSSuz21} we have that~$\tparen{\EE{\dUpsilon}{\PP},\dom{\EE{\dUpsilon}{\PP}}}=\tparen{\Ch[\mssd_\dUpsilon,\PP],\dom{\Ch[\mssd_\dUpsilon,\PP]}}$, hence the results in~\cite[Lem.~5.2]{ErbHue15} hold in the present case with identical proof.
The rest of the proof follows exactly as in~\cite[Lem.~5.4]{ErbHue15}, having care to substitute~\cite[Lem.~5.3]{ErbHue15} with Theorem~\ref{t:Deng}, and~\cite[Lem.~5.1]{ErbHue15} with Lemma~\ref{l:ErbHue} below.
\end{proof}
\end{prop}

\begin{lem}\label{l:MeasurableSelectionCoupling}
Let~$(\mcX,\mssd)$ be a metric local structure.
There exists a $\A_\mrmv(\Ed)/\A_\mrmv(\Ed^\otym{2})$-measurable map~$\cpl^{\emparg,\emparg}\colon (\dUpsilon^\sym{\infty})^\tym{2}\to \dUpsilon^\sym{\infty}(\Ed^\otym{2})$ so that
\begin{align*}
\mssd_\dUpsilon(\gamma,\eta)^2= \int_{X^\tym{2}} \mssd^2\diff \cpl^{\gamma,\eta} \fstop
\end{align*}
\begin{proof}
By e.g.~\cite[Prop.~3.14(ii)]{LzDSSuz21}, the space~$\tparen{\dUpsilon,\T_\mrmv(\Ed)}$ is Polish.
As a consequence, and since~$\dUpsilon^\sym{n}$ is $\T_\mrmv(\Ed)$-closed for every~$n$, the space~$\dUpsilon^\sym{\infty}= \cap_n (\dUpsilon^\sym{n})^\complement$ is a $G_\delta$ subset of a Polish space, thus itself Polish.

Now, the proof of the same assertion in~\cite[Lem.~6.1]{ErbHue15} when~$\mcX$ is a manifold applies to our setting, having care that $\dUpsilon^\sym{\infty}$ is Polish, hence and every argument in~\cite{ErbHue15} can be restricted to~$\dUpsilon^\sym{\infty}$, and that we may replace~\cite[Lem.~4.1(i), (vi)]{RoeSch99}, used in~\cite{ErbHue15}, by the corresponding assertions in~\cite[Prop.~4.27$(ii)$, $(vii)$]{LzDSSuz21}.
\end{proof}
\end{lem}

\begin{lem}[$W_{2,\mssd_\dUpsilon}$-closure of~$\dom{\Ent_\PP}$]\label{l:ErbHue}
Let~$(\mcX,\cdc,\mssd)$ be an~\MLDS satisfying~$\RCD^*(K,N)$ for some~$K\in\R$,~$N\in (2,\infty)$. Then,
\begin{align*}
\cl_{W_{2,\mssd_\dUpsilon}}\tparen{\dom{\Ent_\PP}}= B^{W_{2,\mssd_\dUpsilon}}_\infty\tparen{\dom{\Ent_\PP}} \fstop
\end{align*}

\begin{proof}
Firstly, note that, since~$\mssm X=\infty$ by assumption,~$\PP\dUpsilon^\sym{\infty}=1$, and therefore $\dom{\Ent_\PP}\subset \msP(\dUpsilon^\sym{\infty})$.
Furthermore, for every~$\nu\in\msP(\dUpsilon^\sym{\infty})$ we have that~$W_{2,\mssd_\dUpsilon}(\nu,\mu)<\infty$ if and only if~$\mu\in\msP(\dUpsilon^\sym{\infty})$.
As a consequence, throughout the rest of the proof we will implicitly assume that every configuration on~$X$ is an element of~$\dUpsilon^\sym{\infty}$ and that every probability measure on~$\dUpsilon$ is concentrated on~$\dUpsilon^\sym{\infty}$.
We will regard every such measure as its restriction to~$\dUpsilon^\sym{\infty}$.

The inclusion `$\subset$' is by definition of closure, thus it suffices to show the reverse inequality.
Let~$\mu\in B^{W_{2,\mssd_\dUpsilon}}_\infty\tparen{\dom{\Ent_\PP}}$, and~$\nu\in\dom{\Ent_\PP}$ be so that~$W_{2,\mssd_\dUpsilon}(\mu,\nu)<\infty$.
It suffices to show that there exists~$\seq{\mu_n}_n\subset \dom{\Ent_\PP}$ and such that~$\nlim W_{2,\mssd_\dUpsilon}(\mu_n,\mu)=0$.
To this end, we adapt to our setting the proof of~\cite[Lem.~5.1]{ErbHue15}, further detailing all measurability statements.

\paragraph{Preliminaries}
For any choice of~$A\in\A$,~$r>0$, and~$x\in X$, let us denote by
\[
U^A_{x,r}\eqdef \frac{\car_{B_r(x)\cap A}\mssm}{\mssm (B_r(x)\cap A)}
\]
the normalized $\mssm$-measure of~$B_r(x)\cap A$.
Further fix~$x_0\in X$, and set~$B=B_n\eqdef B_n(x_0)$ for each~$n\in\N_1$.
Denote by~$\Geo(X)$ the space of all constant-speed geodesics in~$X$ parametrized on~$[0,1]$, and endowed with the uniform distance~$\mssd_\infty$ induced by~$\mssd$. 
It is not difficult to show that, since~$(X,\mssd)$ is complete and separable, so is~$\tparen{\Geo(X),\mssd_\infty}$.
As a consequence of~\cite[Lem.~2.11]{AmbGig11} and Aumann--Sainte-Beuve Measurable Selection Theorem~\cite[II.6.9.13]{Bog07}, there exists a $\Bo{\mssd}/\Bo{\mssd_\infty}$-measurable selection~$\GeoSel\colon X\to \Geo(X)$ so that~$\GeoSel(x)$ is a constant-speed geodesic connecting~$x_0$ to~$x$.
For~$r\in (0,1)$ we define a map
\begin{align}\label{eq:l:ClosureDEnt:0.5}
\chi\colon (r,x)\longmapsto\chi_r(x)\eqdef \begin{cases}\GeoSel(x)_{r} &\text{if } \mssd(x,\partial B)<r \\ x & \text{otherwise}\end{cases}\comma
\end{align}
and note that it is $\Bo{[0,1]}\otimes\Bo{\mssd}/\Bo{\mssd}$-measurable and that~$B^\mssd_r(\chi_r(x))\subset B$ for every~$r\in (0,1)$.

\paragraph{Construction of~$\mu_n$}
Let~$\cpl^{\emparg,\emparg}$ be the map defined in Lemma~\ref{l:MeasurableSelectionCoupling}.
Further define a map
\begin{align*}
\xi^{\emparg,\emparg}\colon(\dUpsilon^\sym{\infty})^\tym{2}\to \dUpsilon^\sym{\infty}\comma\qquad \xi^{\emparg,\emparg}\colon (\gamma,\eta)\longmapsto \pr^1_\pfwd (\cpl^{\gamma,\eta}_{B^\tym{2}})+\pr^2_\pfwd(\cpl^{\gamma,\eta}_{B^\complement\times B}) + \pr^2_\pfwd(\cpl^{\gamma,\eta}_{X\times B^\complement}) \fstop
\end{align*}
Note that~$\xi^{\emparg,\emparg}$ is $\A_\mrmv(\Ed)^\otym{2}/\A_\mrmv(\Ed)$-measurable, and that
\begin{align}
\label{eq:l:ClosureDEnt:1}
\xi^{\gamma,\eta}_{B^\complement}=&\ \tparen{\pr^1_\pfwd (\cpl^{\gamma,\eta}_{B^\tym{2}})}_{B^\complement}+\tparen{\pr^2_\pfwd(\cpl^{\gamma,\eta}_{B^\complement\times B})}_{B^\complement} + \tparen{\pr^2_\pfwd (\cpl^{\gamma,\eta}_{X\times B^\complement})}_{B^\complement} 
= 0 + 0 + (\cpl^{\gamma,\eta}_{B^\complement})_{B^\complement}= \eta_{B^\complement}\comma
\intertext{and}
\nonumber
\xi^{\gamma,\eta} B =&\ \tparen{\pr^1_\pfwd (\cpl^{\gamma,\eta}_{B^\tym{2}})}B + \tparen{\pr^2_\pfwd(\cpl^{\gamma,\eta}_{B^\complement\times B})}B + \tparen{\pr^2_\pfwd (\cpl^{\gamma,\eta}_{X\times B^\complement})} B
\\
\nonumber
=&\ \cpl^{\gamma,\eta}_{B^\tym{2}} (B\times X)+\cpl^{\gamma,\eta}_{B^\complement\times B} (X\times B)+\eta_{B^\complement} B = \cpl^{\gamma,\eta} (B\times B)+\cpl^{\gamma,\eta} (B^\complement \times B)+0
\\
\nonumber
=&\ \cpl^{\gamma,\eta} (X\times B)
\\
\label{eq:l:ClosureDEnt:2}
=&\ \eta B \fstop
\end{align}
Finally, let us set~$a_\eta\eqdef 1/\sqrt{4n\, \xi^{\gamma,\eta} B}=1/\sqrt{4n\, \eta B}$ and note that~$\eta\mapsto a_\eta$ is $\A_\mrmv(\Ed)$-measurable, by measurability of~$\xi^{\emparg,\emparg}$.
For every~$\gamma,\eta\in\dUpsilon^\sym{\infty}$, we define a probability measure~$\mcU^{\gamma,\eta}_n\in\msP(\dUpsilon^\sym{\infty})$ by
\begin{align*}
\mcU^{\gamma,\eta}_n (\Lambda)\eqdef \begin{cases} 
\displaystyle\int_\Lambda \delta_{\xi^{\gamma,\eta}_{B^\complement}+\sum_{i=1}^k \delta_{y_i}} \prod_{x\in \xi^{\gamma,\eta}_B} \diff U^B_{\chi_{a_\eta}(x),a_\eta}(y_i) &\text{if } \xi^{\gamma,\eta}_B\neq \emp\comma
\\
\delta_{\xi^{\gamma,\eta}_{B^\complement}}\Lambda & \text{otherwise}\comma
\end{cases}
\qquad \Lambda\in\A_\mrmv(\Ed) \fstop
\end{align*}

\paragraph{Claim 1: for every fixed~$n\in \N_1$ and~$\Lambda\in \A_\mrmv(\Ed)$ the map $(\gamma,\eta)\mapsto \mcU^{\gamma,\eta}_n\Lambda$ is $\A_\mrmv(\Ed)^\otym{2}$-measurable}
The claim follows combining: the $\A_\mrmv(\Ed)^\otym{2}/\A_\mrmv(\Ed)$-measurability of~$\xi^{\emparg,\emparg}$, the $\A_\mrmv(\Ed)$-measurability of $a_\emparg$, the $\A_\mrmv(\Ed)/\A_\mrmv(\Ed)$-measurability of the restriction map $\pr^B\colon\xi\mapsto \xi_B$, the continuity of the Dirac embedding~$\delta\colon \dUpsilon\to \msP(\dUpsilon)$, the $\Bo{\mssd}/\Bo{\mssd}$-measur\-abil\-ity of~$\chi$ as in~\eqref{eq:l:ClosureDEnt:0.5}, and the measur\-abil\-ity of integral mappings.

\medskip

Finally, fix~$\kappa\in\Opt(\mu,\nu)$, and define~$\mu_n\in\msP(\dUpsilon^\sym{\infty})$ by
\begin{align*}
\mu_n\Lambda \eqdef \int_{(\dUpsilon^\sym{\infty})^\tym{2}} \mcU^{\gamma,\eta}_n \Lambda \diff\kappa(\gamma,\eta)\comma \qquad \Lambda\in \A_\mrmv(\Ed)\fstop
\end{align*}

The proof is concluded if we show that the following assertions hold:

\paragraph{Claim 2: $\nlim W_{2,\mssd_\dUpsilon}(\mu_n,\mu)=0$}
Having care to replace~\cite[Lem.~2.6]{ErbHue15} with~\cite[Prop~4.27$(v)$]{LzDSSuz21}, the assertion holds exactly as in the proof of~\cite[Lem.~5.1, Claim~1]{ErbHue15}.

\paragraph{Claim 3: $\mu_n\ll \PP$ and~$\Ent_\PP(\mu_n)<\infty$ for all~$n\in \N_1$}
The proof of this claim holds as in~\cite[Lem.~5.1, Claim~2]{ErbHue15} noting that, since~$(X,\mssd,\mssm)$ is an~$\RCD^*(K,N)$ space, then~$\mssm B^\mssd_r(x)\gtrsim r^m$ locally uniformly in~$x$ for every sufficiently small~$r>0$ and some~$m>0$ by Bishop--Gromov's inequality~\cite[Thm.~2.3]{Stu06b}.
This suffices to the arguments in~\cite{ErbHue15}, since one only needs a lower bound on the volume of balls centered in~$B=B_n$, for each fixed~$n\in \N$.
\end{proof}
\end{lem}

We are now ready to prove the main result of this section.

\begin{proof}[Proof of Theorem~\ref{t:EVI}]
In light of the identification of the canonical form~$\tparen{\EE{\dUpsilon}{\QP},\dom{\EE{\dUpsilon}{\QP}}}$ with the Cheeger energy $\tparen{\Ch[\mssd_\dUpsilon,\QP],\dom{\Ch[\mssd_\dUpsilon,\QP]}}$ provided by Theorem~\ref{t:GeometricProperties}\ref{i:t:GeometricProperties:4}, a proof follows purely from the machinery of metric measure geometry (i.e.\ without resorting to Dirichlet-form theory).
Indeed, the proof of the analogous statement for manifolds in~\cite[Thm.~5.10]{ErbHue15} carries over \emph{verbatim} to our setting, having care to substitute
\begin{enumerate*}[$(a)$]
\item \cite[Prop.2.3]{ErbHue15} with~\ref{t:GeometricProperties}\ref{i:t:GeometricProperties:4};
\item \cite[Lem.~5.2]{ErbHue15} with Lemma~\ref{l:EntropyFisher};
\item \cite[Lem.~5.3]{ErbHue15} with Theorem~\ref{t:Deng};
\item the entropy-cost inequality~\eqref{eq:EntropyCost}, \cite[Lem.~5.4]{ErbHue15}, with Proposition~\ref{p:HKRegularization};
\item and finally noting that~\cite[Lem.~5.1]{ErbHue15} is not needed, since our Definition~\ref{d:EVI} of~$(\EVI_K)$ only requires gradient curves to start at points in~$\cl_{W_2}\dom{\Ent_\emparg}$, rather than in~$B^{W_2}_\infty(\dom{\Ent_\emparg})$.
\end{enumerate*}
\end{proof}

\subsubsection{Mixed Poisson measures}
We now show how to adapt the previously established results to the case of mixed Poisson measures.
We say that~$\lambda\in\msP(\R^+_0)$ is a \emph{L\'evy measure} if~$\lambda\set{0}=0$ and~$\lambda (1\wedge t)<\infty$.

\begin{defs}[Mixed Poisson measures]\label{d:MixedPoissonConfig2}
Let~$\mcX$ be a topological local structure, and~$\lambda\in\msP(\R^+_0)$ be a L\'evy measure.
The \emph{mixed Poisson measure} with intensity measure~$\mssm$ and L\'evy measure~$\lambda$ is the probability measure on~$\tparen{\dUpsilon,\A_\mrmv(\msE)}$ defined as
\begin{align}\label{eq:d:MixedPoissonConfig2}
\QP_{\lambda,\mssm}=\int_{\R^+} \PP_{s\cdot\mssm}\diff\lambda(s) \fstop
\end{align}
\end{defs}

As usual, we shall henceforth omit the specification of~$\mssm$.
Let us further recall the following well-known fact, see e.g.~\cite[\S7 and Ref.s therein]{LzDSSuz21}
\begin{equation}\label{eq:Skorokhod}
\PP_{s\, \mssm} \perp \PP_{t\, \mssm} \qquad \text{for every~$s,t>0$ with~$s\neq t$} \fstop
\end{equation}

The properties of the form~$\EE{\dUpsilon}{\QP_\lambda}$ constructed by replacing~$\PP$ with~$\QP_\lambda$ in~\eqref{eq:d:MixedPoissonConfig2} were established in~\cite{LzDSSuz21}.

\begin{thm}[{\cite[Cor.~7.16]{LzDSSuz21}}]\label{t:MixedPoissonGeometricProperties}
Let~$(\mcX,\cdc)$ be a \TLDS. Then, the form~$\tparen{\EE{\dUpsilon}{\QP_\lambda},\CylQP{\QP_\lambda}{\Dz}}$ is closable.
Its closure~$\tparen{\EE{\dUpsilon}{\QP_\lambda},\dom{\EE{\dUpsilon}{\QP_\lambda}}}$ is a Dirichlet form.
If additionally $(\mcX,\cdc,\mssd)$ is an~$\MLDS$ with~$\Dz\subset \bLip(\mssd)$ and satisfying~$(\Rad{\mssd}{\mssm})$, then~$\tparen{\EE{\dUpsilon}{\QP_\lambda},\dom{\EE{\dUpsilon}{\QP_\lambda}}}$ satisfies~$(\Rad{\mssd_\dUpsilon}{\QP_\lambda})$, and it is a quasi-regular strongly local recurrent (conservative) Dirichlet form.
\end{thm}

Based on the direct-integral representation for~$\tparen{\EE{\dUpsilon}{\QP_\lambda},\dom{\EE{\dUpsilon}{\QP_\lambda}}}$ established in~\cite{LzDSSuz21}, viz.
\begin{align}\label{eq:DirIntMixedPoisson}
\EE{\dUpsilon}{\QP_{\lambda,\mssm}}= \dint{\R^+_0} \EE{\dUpsilon}{\PP_{s\mssm}} \diff\lambda(s)\comma
\end{align}
we may generalize our results on curvature bounds to the case of mixed Poisson measures.
We start with the following corollary of Theorem~\ref{t:dSLConfig2}.

\begin{cor}
Let~$(\mcX,\cdc,\mssd)$ be an \MLDS satisfying \SCF and $(\SL{\mssm}{\mssd})$.
Then,
\begin{align*}
\mssd_\dUpsilon\geq \mssd_{\QP} \fstop
\end{align*}
\begin{proof}
The inequality~`$\geq$' is a consequence of~\eqref{eq:DirIntMixedPoisson} and~\cite[Prop.~7.15($iii$)]{LzDSSuz21}.
The opposite inequality holds as in Corollary~\ref{c:IdentifDist}.
\end{proof}
\end{cor}

Let us now turn to the curvature assertions.

\begin{thm}\label{t:MixedPoissonAll}
Fix~$K\in\R$, and let~$(\mcX,\cdc)$ be a \TLDS satisfying~\SCF. Then, the following assertions are equivalent:
\begin{enumerate}[$({a}_1)$]
\item\label{i:t:BEMixed:1} $(\mcX,\cdc)$ satisfies~$\BE(K,\infty)$;
\item\label{i:t:BEMixed:2} $(\dUpsilon,\SF{\dUpsilon}{\QP_\lambda})$ satisfies~$\BE(K,\infty)$ for every L\'evy measure~$\lambda\in \msP(\R^+_0)$.
\end{enumerate}

Further let~$(\mcX,\cdc,\mssd)$ be an \MLDS satisfying~\qSCF. Then, the assertions in each of the following pairs are equivalent:

\begin{enumerate}[$({b}_1)$]
\item\label{i:t:LogHMixed:1} $(\mcX,\cdc,\mssd)$ satisfies~\ref{ass:LogH} with rate~$K$;
\item\label{i:t:LogHMixed:2} $\tparen{\dUpsilon^\sym{\infty},\SF{\dUpsilon}{\QP_\lambda},\mssd_\dUpsilon}$ satisfies~\ref{ass:LogH} with rate~$K$ for every L\'evy measure~$\lambda\in \msP(\R^+_0)$.
\end{enumerate}

\begin{enumerate}[$({c}_1)$]
\item\label{i:t:WCMixed:1} $(\mcX,\cdc,\mssd, \hh{X}{\mssm}_\bullet)$ satisfies~$(\kWC{\mssd})$ with rate~$c(t)$ as in~\eqref{eq:WCConstant};
\item\label{i:t:WCMixed:2} $\tparen{\dUpsilon^\sym{\infty}, \SF{\dUpsilon}{\QP_\lambda},\mssd_\dUpsilon,\mssh^\dUpsilon_\bullet}$ with
\begin{align}\label{eq:t:Mixed:0}
\hh{\dUpsilon}{\QP_\lambda}_\bullet(\gamma,\diff\emparg)\eqdef \int_{\R^+_0} \hh{\dUpsilon}{\PP_{s\mssm}}_\bullet(\gamma,\diff\emparg) \diff\lambda(s)
\end{align}
satisfies~$(\kWC{\mssd_\dUpsilon})$ with rate~$c(t)$ as in~\eqref{eq:WCConstant} for every L\'evy measure~$\lambda\in \msP(\R^+_0)$.
\end{enumerate}

Finally, let~$(X,\mssd,\mssm)$ be an~$\RCD^*(K,N)$ space,~$K\in\R$,~$N\in (2,\infty)$. Then,
\begin{enumerate}[$(a)$]\setcounter{enumi}{3}
\item\label{i:t:EVIMixed:2} $(\dUpsilon,\SF{\dUpsilon}{\QP_\lambda},\mssd_\dUpsilon)$ satisfies~$(\EVI_{K,\mssd_\dUpsilon,\PP})$ for every L\'evy measure~$\lambda\in \msP(\R^+_0)$.
\end{enumerate}

\begin{proof}
We start by noting that~$\BE(K,\infty)$,~\ref{ass:LogH},~$(\kWC{\mssd})$,~$(\EVI_{K,\mssd,\mssm})$ are all invariant under the measure rescaling~$\mssm\mapsto s\mssm$ for every~$s>0$, cf.~e.g.~\cite[Prop.~4.13]{Stu06a}.
Thus, we may and henceforth will use without further mention that if~$(\mcX,\cdc,\mssd)$ satisfies either of the above, then~$(s\mcX,\cdc,\mssd)$ does so as well.
Choosing~$\lambda=\delta_1$ we have that~$\QP_\lambda=\PP$, thus it suffices to only show the forward implications, the backward ones being proved in Theorems~\ref{t:BE},~\ref{t:Deng}, and~\ref{t:WC} respectively.

Now, let~$\lambda$ be any L\'evy measure.

\ref{i:t:BEMixed:1}$\implies$\ref{i:t:BEMixed:2}
By the assumption and Theorem~\ref{t:BE} we have~$\BE(K,\infty)$ for~$\tparen{\dUpsilon,\SF{\dUpsilon}{\PP_{s\mssm}},\mssd_\dUpsilon}$ for every~$s>0$.
It suffices to~$\BE(K,\infty)$ for~$\dUpsilon$ for every cylinder function~$u\in\Cyl{\Dz}$.
The extension to~$u\in\dom{\EE{\dUpsilon}{\QP_\lambda}}$ follows as in the proof of Theorem~\ref{t:BE}\ref{i:t:BE:1}$\implies$\ref{i:t:BE:2}.
For every~$u\in\Cyl{\Dz}$, its non-relabelled continuous representative satisfies~$u\in \CylQP{\PP_{s\mssm}}{\Dz}$ for every~$s>0$, hence, integrating the above~$\BE(K,\infty)$ inequalities w.r.t.~$\lambda$ we have that
\begin{align}\label{eq:DIntBE:1}
\int_{\R^+_0} \SF{\dUpsilon}{\PP_{s\mssm}}\tparen{\TT{\dUpsilon}{\PP_{s\mssm}}_t u}\diff\lambda(s) \leq e^{-2Kt} \int_{\R^+_0} \TT{\dUpsilon}{\PP_{s\mssm}}_t\SF{\dUpsilon}{\PP_{s\mssm}}(u) \diff\lambda(s) \fstop
\end{align}
In light of direct-integral representation~\eqref{eq:DirIntMixedPoisson}, it respectively follows from~\cite[Prop.~2.13, Lem.~3.8]{LzDS20} that
\begin{align}\label{eq:DIntTSF}
\TT{\dUpsilon}{\QP_\lambda}_t=\dint{\R_0^+} \TT{\dUpsilon}{\PP_{s,\mssm}}_t \diff \lambda(s)\comma \qquad \SF{\dUpsilon}{\QP_\lambda}= \dint{\R_0^+} \SF{\dUpsilon}{\PP_{s\mssm}} \diff \lambda(s) \fstop
\end{align}
Combining~\eqref{eq:DIntBE:1} with~\eqref{eq:DIntTSF} for~$\TT{\dUpsilon}{\QP_\lambda}_\bullet$ and~$\SF{\dUpsilon}{\QP_\lambda}$, we have that
\begin{align*}
\int_{\R^+_0} \SF{\dUpsilon}{\PP_{s\mssm}}\tparen{\TT{\dUpsilon}{\PP_{s\mssm}}_t u}\diff\lambda(s) \leq&\ e^{-2Kt} \int_{\R^+_0} \TT{\dUpsilon}{\PP_{s\mssm}}_t\SF{\dUpsilon}{\PP_{s\mssm}}(u) \diff\lambda(s)
\\
=& e^{-2Kt} \TT{\dUpsilon}{\QP_\lambda}_t \int_{\R^+_0} \SF{\dUpsilon}{\PP_{s\mssm}}(u) \diff\lambda(s)
= e^{-2Kt} \TT{\dUpsilon}{\QP_\lambda}_t \SF{\dUpsilon}{\QP_\lambda}(u)\comma
\end{align*}
which concludes the inequality~$\BE(K,\infty)$ for~$\tparen{\dUpsilon,\SF{\dUpsilon}{\QP_\lambda},\mssd_\dUpsilon}$.

\medskip

\ref{i:t:LogHMixed:1}$\implies$\ref{i:t:LogHMixed:2}
Again it suffice to show~\ref{ass:LogH} on~$\tparen{\dUpsilon,\SF{\dUpsilon}{\QP_\lambda},\mssd_\dUpsilon}$ for all~$u\in\Cyl{\Dz}$.
By the assumption and Theorem~\ref{t:Deng} we have~\ref{ass:LogH} for~$\tparen{\dUpsilon,\SF{\dUpsilon}{\PP_{s\mssm}},\mssd_\dUpsilon}$ for every~$s>0$.
Similarly to the proof of \ref{i:t:BEMixed:1}$\implies$\ref{i:t:BEMixed:2}, by the direct-integral representation~\eqref{eq:DirIntMixedPoisson} combined with~\cite[Prop.~2.13]{LzDS20}, and by the assumption, we have that
\begin{align*}
\tparen{\TT{\dUpsilon}{\QP_\lambda}_t \log u}(\gamma)=&\ \int_{\R^+_0} \tparen{\TT{\dUpsilon}{\PP_{s\mssm}}_t \log u}(\gamma) \diff\lambda(s)
\\
\leq& \int_{\R^+_0} \log \tparen{\TT{\dUpsilon}{\PP_{s\mssm}}_t u}(\eta) \diff\lambda(s) +\int_{\R^+_0} \frac{\mssd_\dUpsilon(\gamma,\eta)^2}{4I_{2K}(t)} \diff\lambda(s)
\\
\leq&\ \log\int_{\R^+_0} \tparen{\TT{\dUpsilon}{\PP_{s\mssm}}_t u}(\eta) \diff\lambda(s) +\frac{\mssd_\dUpsilon(\gamma,\eta)^2}{4I_{2K}(t)}
\end{align*}
by Jensen's inequality and since~$\lambda$ is a probability measure.
A further application of~\cite[Prop.~2.13]{LzDS20} concludes the proof.

\medskip

\ref{i:t:WCMixed:1}$\implies$\ref{i:t:WCMixed:2}
By the assumption and Theorem~\ref{t:WC} we have~$(\kWC{\mssd_\dUpsilon})$ for~$\tparen{\dUpsilon,\SF{\dUpsilon}{\PP_{s\mssm}},\mssd_\dUpsilon}$ for every~$s>0$.
As in the proof of Theorem~\ref{t:WC}, it suffices to show~\eqref{eq:t:WC:1.8} with~$\hh{\dUpsilon}{\QP_\lambda}_t$ as in~\eqref{eq:t:Mixed:0} in place of~$\mssh^\dUpsilon_t$.
By the assumption, argueing as in the proof of Theorem~\ref{t:WC} and integrating~\eqref{eq:t:WC:1.8} w.r.t.~$\lambda$, we have that
\begin{align}\label{eq:t:Mixed:1}
\int_{\R^+_0} W_{2,\mssd_\dUpsilon}\tparen{\hh{\dUpsilon}{\PP_{s\mssm}}_t(\gamma,\diff\emparg),\hh{\dUpsilon}{\PP_{s\mssm}}_t(\eta,\diff\emparg)}^2 \diff\lambda(s) \leq c(t)^2\, \mssd_\dUpsilon(\gamma,\eta)^2 \fstop
\end{align}
Now, by convexity of the Wasserstein distance,
\begin{align}\label{eq:t:Mixed:2}
\begin{aligned}
W_{2,\mssd_\dUpsilon}&\paren{\int_{\R^+_0} \hh{\dUpsilon}{\PP_{s\mssm}}_t(\gamma,\diff\emparg) \diff\lambda(s),\int_{\R^+_0}\hh{\dUpsilon}{\PP_{s\mssm}}_t(\eta,\diff\emparg) \diff\lambda(s)}^2
\leq
\\
\leq&
\int_{\R^+_0} W_{2,\mssd_\dUpsilon}\tparen{\hh{\dUpsilon}{\PP_{s\mssm}}_t(\gamma,\diff\emparg),\hh{\dUpsilon}{\PP_{s\mssm}}_t(\eta,\diff\emparg)}^2 \diff\lambda(s) \fstop
\end{aligned}
\end{align}
Combining~\eqref{eq:t:Mixed:1} with~\eqref{eq:t:Mixed:2} yields the conclusion in view of~\eqref{eq:t:Mixed:0}.

\medskip

\ref{i:t:EVIMixed:2}
We show how to adapt the proof of Theorem~\ref{t:EVI} to the case of mixed Poisson measures.

Firstly, the entropy-cost inequality~\eqref{eq:EntropyCost} holds with~$\TT{\dUpsilon}{\QP_\lambda}_\bullet$ in place of~$\TT{\dUpsilon}{\QP}_\bullet$ and~$\Ent_{\QP_\lambda}$ in place of~$\Ent_\PP$ as a consequence of the logarithmic Harnack inequality for~$\tparen{\dUpsilon,\SF{\dUpsilon}{\QP_\lambda},\mssd_\dUpsilon}$ established in~\ref{i:t:LogHMixed:2} above, again arguing exactly as in~\cite[Lem.~5.3]{ErbHue15}.

Secondly, Proposition~\ref{p:HKRegularization} holds (with the same proof) with~$\QP_\lambda$ in place of~$\PP$ in light of the validity for~$\QP_\lambda$ of the Wasserstein contractivity estimate~$\kWC{\mssd_\dUpsilon}$ established in~\ref{i:t:WCMixed:2} above.

The rest of the proof also follows as in Theorem~\ref{t:EVI}, provided we show Lemma~\ref{l:ErbHue} with~$\QP_\lambda$ in place of~$\PP$.
Since~$\QP_\lambda$ too is completely independent, Claims~1 and~2 hold with identical proof.
As for Claim~3, this is again an adaptation of~\cite[Lem.~5.1, Claim~2]{ErbHue15}.
For this adaptation, it suffices to note that, as a consequence of the Poisson--Lebesgue representation of~$\PP_{s\,\mssm_B}$ (e.g.,~\cite[Eqn.~\eqref{eq:PoissonLebesgue}]{LzDSSuz21}),
\[
\tparen{\pr^B_\pfwd\QP_\lambda \dUpsilon^\sym{k}(B)}^{-1} \pr^B_\pfwd \QP_\lambda\mrestr{\dUpsilon^\sym{k}(B)}= \frac{\mssm_B^\otym{k}}{(\mssm B)^k}
\]
independently of~$\lambda$.
We omit the details.
\end{proof}
\end{thm}

\section{Applications}\label{s:ApplicationsConfig2}
We collect here some applications of the results on the curvature of~$\dUpsilon$ established above.

\subsection{Sobolev-to-Lipschitz property}
Let~$\QP_\lambda$ be a mixed Poisson measure as in Definition~\ref{d:MixedPoissonConfig2}.
In this section we establish the Sobolev-to-Lipschitz property for configuration spaces under the assumption of synthetic lower Ricci curvature bounds on the base space, which confirms the conjecture by R\"ockner--Schied~\cite[p.\ 331]{RoeSch99} originally stated in the case that $\mcX$ is a smooth Riemannian manifold.
We stress that this property is strictly stronger than both the continuous- and distance-continuous Sobolev-to-Lipschitz properties shown in~\cite{LzDSSuz21}.
Indeed, Theorem~\ref{t:SL} below can be understood as a self-improvement of~$(\dcSL{\mssd_\dUpsilon}{\QP_\lambda}{\mssd_\dUpsilon})$ to~$(\SL{\QP_\lambda}{\mssd_\dUpsilon})$ under the assumption of Ricci-curvature lower bounds. 

\begin{thm}[Sobolev-to-Lipschitz property]\label{t:SL}
Let~$(\mcX,\cdc,\mssd)$ be an~\MLDS satisfying\linebreak \SCF, $\BE(K,\infty)$, and~\ref{ass:LogH}.
Further assume that~$\tparen{\dUpsilon,\SF{\dUpsilon}{\QP_\lambda},\mssd_\dUpsilon}$ satisfies\linebreak $(\dcSL{\mssd_\dUpsilon}{\QP_\lambda}{\mssd_\dUpsilon})$.
Then, every~$u\in\tparen{\EE{\dUpsilon}{\QP_\lambda},\dom{\EE{\dUpsilon}{\QP_\lambda}}}$ with~$\norm{\SF{\dUpsilon}{\QP_\lambda}(u)}_{L^\infty(\QP_\lambda)}\leq L$ has a $\QP_\lambda$-measurable $\mssd_\dUpsilon$-Lipschitz $\QP_\lambda$-representative~$\rep u$ with~$\Li[\mssd_\dUpsilon]{\rep u}\leq L$.
In particular, $\tparen{\dUpsilon,\SF{\dUpsilon}{\QP_\lambda},\mssd_\dUpsilon}$ satisfies~$(\SL{\QP_\lambda}{\mssd_\dUpsilon})$ for the $\sigma$-algebra~$\A_\mrmv(\Ed)^{\QP_\lambda}$.
\end{thm}

As a consequence, we have the following.

\begin{cor}\label{c:SLManifolds}
Fix~$K\in\R$,~$N\in (2,\infty)$. Let~$(\mcX,\cdc,\mssd)$ be the \MLDS associated to either
\begin{itemize}
\item a metric measure space~$(X,\mssd,\mssm)$ satisfying both~$\RCD^*(K,N)$ and~$\CAT(0)$; or
\item a (complete) weighted Riemannian manifold with Ricci curvature uniformly bounded from below.
\end{itemize}
Then, the conclusion of Theorem~\ref{t:SL} holds.
\end{cor}

\begin{rem}[R\"ockner--Schied's conjecture]\label{r:RoeSch}
In the case when~$\mcX$ is a manifold, Corollary~\ref{c:SLManifolds} proves the validity of the Sobolev-to-Lipschitz property~$(\SL{\mssd_\dUpsilon}{\QP_\lambda})$ for the class of mixed Poisson measures, which was conjectured by M.~R\"ockner and A.~Schied in~\cite[Rmk.~p.~331]{RoeSch99}.
\end{rem}


\begin{proof}[Proof of Theorem~\ref{t:SL}]
Without loss of generality, up to multiplicative rescaling, we may and will assume that~$L=1$.
Assume first that~$u$ is additionally $\QP_\lambda$-essentially bounded.
By Theorem~\ref{t:MixedPoissonAll}, we have that~$\tparen{\dUpsilon^\sym{\infty},\SF{\dUpsilon}{\QP_\lambda},\mssd_\dUpsilon}$ satisfies both $\BE(K,\infty)$ and~\ref{ass:LogH}.
Now, let~$\rep u\in \mcL^\infty(\QP_\lambda)$ with~$u\in \DzB{\QP_\lambda}$.
By~\ref{ass:LogH} and~\cite[Prop.~3.1(1)]{WanYua11} we have that~$\TT{\dUpsilon}{\QP_\lambda}_t\colon \mcL^\infty(\QP_\lambda)\to \mcL^\infty(\QP_\lambda)$ is strongly $\mssd_\dUpsilon$-Feller, thus~$\TT{\dUpsilon}{\QP_\lambda}_t u$ is $\mssd_\dUpsilon$-continuous.
By~$\BE(K,\infty)$ and the uniform bound on~$\tnorm{\SF{\dUpsilon}{\QP_\lambda}u}_{L^\infty(\QP_\lambda)}\leq 1$ we conclude that
\[
\tnorm{\SF{\dUpsilon}{\QP_\lambda}(\TT{\dUpsilon}{\QP_\lambda}_t u)}_{L^\infty(\QP_\lambda)}\leq e^{-2Kt}\fstop
\]
By a linear rescaling and~$(\dcSL{\mssd_\dUpsilon}{\QP_\lambda}{\mssd_\dUpsilon})$ we conclude therefore the existence of a $\QP_\lambda$-representative~$\widerep{\TT{\dUpsilon}{\QP_\lambda}_tu}$ of~$\TT{\dUpsilon}{\QP_\lambda}_tu$ satisfying~$\widerep{\TT{\dUpsilon}{\QP_\lambda}_tu}\in e^{-K t}\bLipu(\mssd_\dUpsilon,\A_\mrmv(\Ed)^{{\QP_\lambda}})$.
Since~$t\mapsto \TT{\dUpsilon}{\QP_\lambda}_tu$ is strongly $L^2(\QP_\lambda)$-continuous, there exists~$\seq{t_n}_n\subset (0,1)$ with~$\nlim t_n=0$ and so that~$\widerep{\TT{\dUpsilon}{\QP_\lambda}_tu}$ converges to a $\QP_\lambda$-representative~$\rep u$ of~$u$ $\QP_\lambda$-a.e..
Thus, for some~$\Omega\in \A_\mrmv(\Ed)^{{\QP_\lambda}}$ with full $\QP_\lambda$-measure,
\begin{align*}
\begin{aligned}
\abs{\rep u(\gamma)-\rep u(\eta)}\leq&\ \nlim \abs{\widerep{\TT{\dUpsilon}{\PP}_{t_n}u}(\gamma)-\widerep{\TT{\dUpsilon}{\PP}_{t_n}u}(\eta)}
\\
\leq&\ \nlim e^{-K t_n} \mssd_\dUpsilon(\gamma,\eta)=\mssd_\dUpsilon(\gamma,\eta)\comma
\end{aligned}
\qquad \gamma,\eta\in \Omega\subset \dUpsilon^\sym{\infty}\fstop
\end{align*}
We conclude that the (upper) constrained McShane extension~$\overline{\rep u\restr_\Omega}$ of~$\rep u\restr_\Omega$ constructed in~\cite[Lem.~2.1]{LzDSSuz20} is $\QP_\lambda$-measurable, since~$\QP_\lambda\Omega=1$, and a $\mssd_\dUpsilon$-Lipschitz $\QP_\lambda$-representative of~$u$.

The case of unbounded~$u$ follows by a truncation argument, as we now show.
Without loss of generality, up to separately arguing on the positive/negative part of~$u$, we may assume that~$u\geq 0$.
For~$n\in \N_0$ set~$u_n\eqdef u\wedge n$, and note that~$u_n\in\DzB{\QP_\lambda}$. 
Thus, by the previous case, there exists a $\QP_\lambda$-measurable $\mssd_\dUpsilon$-Lipschitz representative~$\rep u_n$ of~$u_n$.
Then, let~$\bar u_n \eqdef \vee_{k\leq n} \rep u_k$, and note that~$\class[\QP_\lambda]{\bar u_n}\equiv u_n\equiv u$ $\QP_\lambda$-a.e.\ on~$\Lambda_n(u)\eqdef \set{u\leq n}$.
Since~$u$ is $\QP_\lambda$-a.e.\ finite, the sets~$\Lambda_n(u)$ are a monotone increasing $\QP_\lambda$-exhaustion, i.e.~$\QP_\lambda\tparen{\cup_n\Lambda_n(u)}=1$.
Furthermore,~$\bar u_n$ is pointwise monotone increasing.
Thus,~$\bar u\eqdef \nlim \bar u_n$ is an everywhere finite $\A_\mrmv(\Ed)^{\QP_\lambda}$-measurable $\QP_\lambda$-representative of~$u$.
It remains to show that~$\bar u$ is $\mssd_\dUpsilon$-Lipschitz.
Since~$\bar u_n$ is a maximum of finitely many $\mssd_\dUpsilon$-Lipschitz functions with~$\Li[\mssd_\dUpsilon]{\rep u_n}\leq 1$, it is as well $\mssd_\dUpsilon$-Lipschitz with~$\Li[\mssd_\dUpsilon]{\bar u_n}\leq 1$.
In particular, for every~$\gamma,\eta\in\dUpsilon$ with~$\mssd_\dUpsilon(\gamma,\eta)<\infty$, we have that
\begin{align*}
\abs{\bar u_n(\gamma)-\bar u_n(\eta)}\leq \mssd_\dUpsilon(\gamma,\eta) \fstop
\end{align*}
The conclusion follows letting~$n\to\infty$, since~$\bar u=\nlim \bar u_n$ pointwise on~$\dUpsilon$.
\end{proof}

\subsubsection{\texorpdfstring{$L^\infty$}{Linfty}-to-Lipschitz heat-semigroup regularization}
As a further application of the results on curvature bounds, let us show the following $L^\infty$-to-Lipschitz regularization property of the heat semigroup~$\TT{\dUpsilon}{\PP}_\bullet$.

\begin{thm}[$L^\infty$-to-Lipschitz regularization]\label{t:LinftyLip}
Let~$(\mcX,\cdc,\mssd)$ be an~\MLDS satisfying\linebreak \SCF, \linebreak$\BE(K,\infty)$, and~\ref{ass:LogH}.
Further assume that~$\tparen{\dUpsilon,\SF{\dUpsilon}{\QP_\lambda},\mssd_\dUpsilon}$ satisfies~$(\dcSL{\mssd_\dUpsilon}{\QP_\lambda}{\mssd_\dUpsilon})$.
Then, there exists a $\PP$-representative~$\widerep{\TT{\dUpsilon}{\QP_\lambda}_t u}$ of~$\TT{\dUpsilon}{\QP_\lambda}_t u$ satisfying
\begin{align*}
\widerep{\TT{\dUpsilon}{\QP_\lambda}_t u} \in \bLip(\mssd_\dUpsilon,\A_\mrmv(\Ed)^\PP)\comma \quad \bLi[\mssd_\dUpsilon]{\widerep{\TT{\dUpsilon}{\QP_\lambda}_t u}}\leq \frac{\norm{u}_{L^\infty(\QP_\lambda)}}{\sqrt{2I_{2K}(t)}}\comma \qquad u\in L^\infty(\QP_\lambda)\comma t>0 \fstop
\end{align*}
\end{thm}

\begin{rem}
Alongside Theorems~\ref{t:GeometricProperties}\ref{i:t:GeometricProperties:4} and~\ref{t:dSLConfig2}, Theorem~\ref{t:LinftyLip} further emphasizes the importance of the (non-separable) $\mssd_\dUpsilon$-topology in the study of the Markov diffusion~$\mbfM^{\QP_\lambda}$ properly associated to~$\tparen{\EE{\dUpsilon}{\QP_\lambda},\dom{\EE{\dUpsilon}{\QP_\lambda}}}$.
This is in sharp contrast with the vague topology~$\T_\mrmv(\Ed)$: indeed, for $\T$-continuous~$f\in \Cz(\msE)$, it is readily seen that~$\TT{\dUpsilon}{\QP_\lambda}_t f^\trid$ does not have a $\T_\mrmv$-continuous representative, even if~$f^\trid$ is $\T_\mrmv$-continuous.
In particular, we do not expect the weak Feller property~$\TT{\dUpsilon}{\QP_\lambda}_t \Cb(\T_\mrmv)\subset \Cb(\T_\mrmv)$ to hold.

Additionally, Theorem~\ref{t:LinftyLip} shows that~$\TT{\dUpsilon}{\PP}_\bullet$ is a regularizing operator. To exhibit any such `\emph{mollifying operator}' in infinite-dimensional settings is in general highly non-trivial.
\end{rem}

\begin{rem}
The assumption of~$(\dcSL{\mssd_\dUpsilon}{\QP_\lambda}{\mssd_\dUpsilon})$ is relevant to our proof, and we do not expect the above result to hold without this assumption.
This is in analogy with the case of base spaces, for which it is known that the sole Bakry--\'Emery curvature condition~$\BE(K,\infty)$ does \emph{not} imply any Sobolev-to-Lipschitz property, see~\cite{Hon18}.
\end{rem}

\begin{cor}\label{c:LinftyLipRCD}
Fix~$K\in\R$ and~$N\in (2,\infty)$. Let~$(\mcX,\cdc,\mssd)$ be the \MLDS associated to either
\begin{itemize}
\item a metric measure space~$(X,\mssd,\mssm)$ satisfying both~$\RCD^*(K,N)$ and~$\CAT(0)$; or
\item a (complete) Riemannian manifold with Ricci curvature uniformly bounded from below.
\end{itemize}
Then, the conclusion of Theorem~\ref{t:LinftyLip} holds.
\end{cor}

\begin{proof}[Proof of Theorem~\ref{t:LinftyLip}]
By Theorem~\ref{t:MixedPoissonAll}, we have that~$\tparen{\dUpsilon^\sym{\infty},\SF{\dUpsilon}{\QP_\lambda},\mssd_\dUpsilon}$ satisfies both $\BE(K,\infty)$ and \ref{ass:LogH}.
Now, let~$\rep u\in \mcL^\infty(\QP_\lambda)$. By~\ref{ass:LogH} and~\cite[Prop.~3.1(1)]{WanYua11}, $\TT{\dUpsilon}{\QP_\lambda}_t\colon \mcL^\infty(\QP_\lambda)\to \mcL^\infty(\QP_\lambda)$ is strongly $\mssd_\dUpsilon$-Feller, thus~$\TT{\dUpsilon}{\QP_\lambda}_t u$ is $\mssd_\dUpsilon$-continuous.
By the very same argument as in~\cite[Thm.~6.5]{AmbGigSav14b} we further have that
\begin{align*}
\sqrt{2 I_{2K}(t)}\, \tnorm{\SF{\dUpsilon}{\QP_\lambda}(\TT{\dUpsilon}{\QP_\lambda}_t u)}_{L^\infty(\QP_\lambda)}\leq \tnorm{\TT{\dUpsilon}{\QP_\lambda}_t u}_{L^\infty(\QP_\lambda)} \fstop
\end{align*}
(Note that the proof of~\cite[Eqn.~(6.17)]{AmbGigSav14b} does not rely on any property of the minimal weak upper gradient other than the latter being the square root of the square field of the Dirichlet form at hand).
Finally, it follows by~$(\dcSL{\mssd_\dUpsilon}{\QP_\lambda}{\mssd_\dUpsilon})$ and a linear rescaling that $\TT{\dUpsilon}{\QP_\lambda}_t u$ has a $\A_\mrmv(\Ed)^{\QP_\lambda}$-measurable $\mssd_\dUpsilon$-Lipschitz representative satisfying
\begin{align*}
\sqrt{2 I_{2K}(t)}\,\, \bLi[\mssd_\dUpsilon]{\widerep{\TT{\dUpsilon}{\QP_\lambda}_t u}}\leq \tnorm{\TT{\dUpsilon}{\QP_\lambda}_t u}_{L^\infty(\QP_\lambda)} \leq \norm{u}_{L^\infty(\QP_\lambda)}\comma
\end{align*}
where the last inequality holds by sub-Markovianity of~$\TT{\dUpsilon}{\QP_\lambda}_\bullet$.
\end{proof}

\subsection{Irreducibility}
Let~$\QP$ be either~$\PP$ or~$\QP_\lambda$ for some L\'evy measure~$\lambda$.
In this section we present some applications of the Sobolev-to-Lipschitz property to the question of \emph{irreducibility} of the form~$\tparen{\EE{\dUpsilon}{\QP},\dom{\EE{\dUpsilon}{\QP}}}$.

Let~$(\EE{X}{\mssm},\dom{\EE{X}{\mssm}}$ be a Dirichlet form with semigroup~$\TT{X}{\mssm}_\bullet$ on a topological local structure~$\mcX$.
We recall that a measurable~$A\in\A$ is ($\EE{X}{\mssm}$-)\emph{invariant} if
\begin{align*}
\TT{X}{\mssm}_t(\car_A f) = \car_A \TT{X}{\mssm}_t f \quad \as{\mssm} \comma \qquad f\in L^2(\mssm) \fstop
\end{align*}
The form~$\tparen{\EE{X}{\mssm},\dom{\EE{X}{\mssm}}}$ (the semigroup~$\TT{X}{\mssm}_\bullet$) is \emph{irreducible} if, whenever~$A$ is invariant, then $A$ is either $\mssm$-negligible or $\mssm$-co-negligible.
Several equivalent formulations of this notion are available for general Dirichlet spaces, see e.g.~\cite{FukOshTak11}.

Since the form~$\tparen{\EE{\dUpsilon}{\QP},\dom{\EE{\dUpsilon}{\QP}}}$ is strongly local recurrent (Thm.~\ref{t:MixedPoissonGeometricProperties}), its irreducibility amounts to say that the Markov $\T_\mrmv(\msE)$-diffusion process~$\mbfM^{\QP}$ properly associated to $\tparen{\EE{\dUpsilon}{\QP},\dom{\EE{\dUpsilon}{\QP}}}$, is ergodic. That is, for ($\T_\mrmv(\msE)$-quasi-)every fixed starting configuration in~$\dUpsilon$, the process~$\mbfM^\QP$ visits every $\T_\mrmv(\msE)$-open neighborhood of~$\dUpsilon$ a.s.
The ergodicity of~$\mbfM^\QP$ is one fundamental property of the process which cannot be accessed by its standard finite-dimensional approximations/localizations.
Indeed, if~$\mssm X<\infty$, or when considering restrictions of the constructions above to a bounded set~$B$, then the processes~$\mbfM^\QP$ is \emph{never} ergodic, since the number of particles ---~which is $\QP$-a.s.\ finite~--- is conserved, and we have that each of the $n$-particle spaces~$\dUpsilon^\sym{n}$ is invariant.

In the case of manifolds, a characterization of the ergodicity of~$\mbfM^\QP$ in terms of~$\QP$ in the class of mixed Poisson measures is provided in~\cite[Thm.~4.3]{AlbKonRoe98}, showing that~$\mbfM^\QP$ is ergodic if and only if~$\QP=\PP$.
Here, we extend this characterization to the non-smooth setting, and add a further equivalent statement, describing the irreducibility of the form~$\tparen{\EE{\dUpsilon}{\QP},\dom{\EE{\dUpsilon}{\QP}}}$ in terms of the set-to-set $L^2$-transportation distance~$\mssd_\dUpsilon$.

\bigskip

Let us start by showing that ---~in contrast to the case of metric spaces (as opposed to: \emph{extended} metric spaces), see Proposition~\ref{p:IrreducibilityBase}~---, irreducibility does \emph{not} follow from the Sobolev-to-Lipschitz property.

\begin{ese}[Sobolev-to-Lipschitz vs.\ irreducibility for mixed Poisson measures]\label{e:SLvsIrrMixedPoisson}
Let $(\mcX,\cdc,\mssd)$ be the \MLDS arising from a complete non-compact Riemannian manifold with Ricci curvature bounded below.
Then,~$\tparen{\dUpsilon,\SF{\dUpsilon}{\QP_\lambda},\mssd_\dUpsilon}$ satisfies $(\SL{\QP_\lambda}{\mssd_\dUpsilon})$ for every L\'evy measure~$\lambda$, yet the form~$\tparen{\EE{\dUpsilon}{\QP_\lambda}, \dom{\EE{\dUpsilon}{\QP_\lambda}}}$ is irreducible if and only if~$\lambda=\delta_s$ for some~$s>0$. %
\begin{proof}
The irreducibility statement is shown in~\cite[Thm.~4.3]{AlbKonRoe98}, while~$(\SL{\QP_\lambda}{\mssd_\dUpsilon})$ is shown in Theorem~\ref{t:SL} provided we show~$(\dcSL{\mssd_\dUpsilon}{\QP_\lambda}{\mssd_\dUpsilon})$.
This holds by~\cite[Thm.~5.23]{LzDSSuz21}, the assumptions of which are readily verified as in~\cite{LzDSSuz21}.
\end{proof}
\end{ese}

Having proved the conjecture in Remark~\ref{r:RoeSch} allows us to provide a characterization of the irreducibility of the form~$\EE{\dUpsilon}{\QP_\lambda}$ in purely geometrical terms.
For~$\Lambda_1,\Lambda_2\in \A_\mrmv(\Ed)^{\QP_\lambda}$ set
\begin{align*}
\mssd_{\dUpsilon,\Lambda_1}(\Lambda_2)\eqdef {\QP_\lambda}\text{-}\essinf_{\gamma\in \Lambda_1}  \inf_{\eta\in\Lambda_2} \mssd_\dUpsilon(\gamma,\eta) \fstop
\end{align*}

\begin{prop}\label{p:Irreducibility}
Let~$(\mcX,\cdc,\mssd)$ be an~\MLDS.
Further assume that~$\tparen{\dUpsilon,\SF{\dUpsilon}{{\QP_\lambda}},\mssd_\dUpsilon}$ satisfies $(\Rad{\mssd_\dUpsilon}{{\QP_\lambda}})$ and~$(\SL{{\QP_\lambda}}{\mssd_\dUpsilon})$, both w.r.t.~$\A_\mrmv(\Ed)^{\QP_\lambda}$.
Then, the following are equivalent:
\begin{enumerate}[$(a)$]
\item\label{i:p:Irreducibility:1} the form~$\tparen{\EE{\dUpsilon}{{\QP_\lambda}},\dom{\EE{\dUpsilon}{{\QP_\lambda}}}}$ is irreducible;
\item\label{i:p:Irreducibility:2} $\mssd_{\dUpsilon,\Lambda_1}(\Lambda_2)<\infty$ for each pair of sets~$\Lambda_1\in\A_\mrmv(\Ed)^{\QP_\lambda}$,~$\Lambda_2\in \Bo{\T_\mrmv(\Ed)}$, with~${\QP_\lambda} \Lambda_i>0$ for~$i=1,2$.
\end{enumerate}

\begin{proof}
\ref{i:p:Irreducibility:1}$\implies$\ref{i:p:Irreducibility:2}
Let~$\Lambda_1,\Lambda_2$ be as in~\ref{i:p:Irreducibility:2}, and consider the function
\[
\mssd_\dUpsilon(\emparg, \Lambda_2)\eqdef \inf_{\eta\in\Lambda_2} \mssd_\dUpsilon(\emparg,\eta)\fstop
\]
This function is $\A_\mrmv(\msE)^{\QP_\lambda}$-measurable by~\cite[Cor.~\ref{c:MeasurabilitydU}]{LzDSSuz21}.
By~\cite[Thm.~5.2]{LzDSSuz21}, the \EMLDS $\tparen{\dUpsilon,\SF{\dUpsilon}{{\QP_\lambda}},\mssd_\dUpsilon}$ satisfies~$(\Rad{\mssd_\dUpsilon}{{\QP_\lambda}})$ w.r.t.\ the $\sigma$-algebra~$\A_\mrmv(\Ed)^{\QP_\lambda}$.
As a consequence, it follows by the same proof of~\cite[Lem.~4.16]{LzDSSuz20} that~$\mssd_\dUpsilon(\emparg, \Lambda_2)\leq \hr{{\QP_\lambda},\Lambda_2}$, hence that
\begin{align*}
\mssd_{\dUpsilon,\Lambda_1}(\Lambda_2)\leq {\QP_\lambda}\text{-}\essinf_{\Lambda_1} \hr{{\QP_\lambda},\Lambda_2} \fstop
\end{align*}
Since~$\tparen{\EE{\dUpsilon}{{\QP_\lambda}},\dom{\EE{\dUpsilon}{{\QP_\lambda}}}}$ is irreducible, by~\cite[Lem.~2.16]{HinRam03} we have that~$ {\QP_\lambda}\text{-}\essinf_{\Lambda_1} \hr{{\QP_\lambda}, \Lambda_2}<\infty$, and the conclusion follows. 

\medskip

\ref{i:p:Irreducibility:2}$\implies$\ref{i:p:Irreducibility:1}
By~$(\SL{\mssd_\dUpsilon}{{\QP_\lambda}})$ for~$\A_\mrmv(\Ed)^{\QP_\lambda}$ and~\cite[Lem.~4.19]{LzDSSuz20} we have that for every~$\Lambda\in \Bo{\T_\mrmv(\Ed)}$ there exists~$\tilde\Lambda\in\A_\mrmv(\Ed)^{\QP_\lambda}$ with~${\QP_\lambda}(\tilde\Lambda\triangle\Lambda)=0$ and
\begin{align}\label{eq:p:Irreducibility:3}
\hr{{\QP_\lambda},\Lambda}\leq \mssd_\dUpsilon(\emparg,\tilde \Lambda)\quad \as{{\QP_\lambda}} \fstop
\end{align}
Since~$\Bo{\T_\mrmv(\Ed)}\subset \A_\mrmv(\Ed)\subset \Bo{\T_\mrmv(\Ed)}^{\QP_\lambda}=\A_\mrmv(\Ed)^{\QP_\lambda}$, there exists~$\Lambda^*\in \Bo{\T_\mrmv(\Ed)}$ with~$\Lambda^*\subset \tilde\Lambda$ and~${\QP_\lambda}(\tilde\Lambda\setminus \Lambda^*)={\QP_\lambda}(\Lambda\triangle\Lambda^*)=0$.
Furthermore, since~$\hr{{\QP_\lambda},\Lambda}$ is independent of the ${\QP_\lambda}$-representative of~$\Lambda$, we have that~$\hr{{\QP_\lambda},\Lambda}=\hr{{\QP_\lambda},\Lambda}=\hr{{\QP_\lambda},\tilde\Lambda}$.
Finally, since~$\Lambda^*\subset \tilde\Lambda$, we have that~$\mssd_\dUpsilon(\emparg, \tilde\Lambda)\leq \mssd_\dUpsilon(\emparg,\Lambda^*)$.
Combining this with~\eqref{eq:p:Irreducibility:3}, we thus have that~${\QP_\lambda}\text{-}\essinf_\Xi\hr{{\QP_\lambda},\Lambda}\leq \mssd_{\dUpsilon,\Xi}(\Lambda^*)<\infty$ for every~$\Xi\in\Bo{\T_\mrmv(\Ed)}$ by the assumption.
The conclusion now follows again by~\cite[Lem.~2.16]{HinRam03}.
\end{proof}
\end{prop}

Within the class of mixed Poisson measures over manifolds, we obtain the following characterization of irreducibility, consequence of the characterization of Poisson measures as the unique tail-trivial measures among mixed Poisson measures, together with Proposition~\ref{p:Irreducibility} and~\cite[Thm.~4.3]{AlbKonRoe98}.

\begin{cor}[Irreducibility and tail-triviality]\label{c:IrreducibilityMixed}
Let~$(\mcX,\cdc,\mssd)$ be the \MLDS arising from a weighted Riemannian manifold with Ricci curvature bounded below.
Then, the following are equivalent:
\begin{enumerate}[$(a)$]
\item $\QP_\lambda$ is tail trivial;
\item $\tparen{\EE{\dUpsilon}{\QP_\lambda},\dom{\EE{\dUpsilon}{\QP_\lambda}}}$ is irreducible;
\item $\lambda=\delta_s$ for some~$s>0$, i.e.~$\QP_\lambda=\PP$ is a Poisson measure;
\item $\mssd_{\dUpsilon,\Lambda_1}(\Lambda_2)<\infty$ for each pair of sets~$\Lambda_1\in\A_\mrmv(\Ed)^{\QP_\lambda}$,~$\Lambda_2\in \Bo{\T_\mrmv(\Ed)}$, with~$\QP_\lambda \Lambda_i>0$ for~$i=1,2$.
\end{enumerate}
\end{cor}

{\small

}
\end{document}